\documentclass[preprint,reqno]{imsart}

\RequirePackage[OT1]{fontenc}
\RequirePackage{amsthm,amsmath}
\RequirePackage[numbers]{natbib}
\RequirePackage[colorlinks,citecolor=blue,urlcolor=blue]{hyperref}
\usepackage{xcolor}

\arxiv{}
\usepackage{bbm}
\usepackage{enumitem}
\usepackage{mathtools}
\usepackage{amsmath}
\usepackage{amssymb}
\usepackage{mathrsfs} 
\usepackage[title,page]{appendix}
\startlocaldefs
\numberwithin{equation}{section}
\theoremstyle{plain}

\endlocaldefs
\newtheorem{theorem}{Theorem}[section]
\newtheorem{lemma}[theorem]{Lemma}
\newtheorem{remark}[theorem]{Remark}

\newtheorem{example}[theorem]{Example}
\allowdisplaybreaks
\usepackage[margin=2cm]{geometry}
\begin{document}

\begin{frontmatter}
\title{{Vector-valued statistics of binomial processes: \\ Berry-Esseen bounds in the convex distance}}
\runtitle{Multivariate normal approximations}

\begin{aug}
\author[A, B]{\fnms{Miko{\l}aj J.} \snm{Kasprzak}\ead[label=e2]{mikolaj.kasprzak@uni.lu}}
\and
\author[B]{\fnms{Giovanni} \snm{Peccati}\ead[label=e3]{giovanni.peccati@uni.lu}}
\runauthor{M.J. Kasprzak and G. Peccati}

\affiliation{University of Luxembourg}

\address[A]{Massachusetts Institute of Technology\\
Laboratory for Information and Decision Systems\\
77 Massachusetts Avenue\\
Cambridge, MA 02139\\
United States of America}
\address[B]{University of Luxembourg\\
Department of Mathematics\\
Maison du Nombre\\
6 Avenue de la Fonte\\
L-4364 Esch-sur-Alzette\\
Luxembourg\\}
\phantom{E-mail:\ }
\end{aug}

\begin{abstract} {We study the discrepancy between the distribution of a vector-valued functional of i.i.d. random elements and that of a Gaussian vector. Our main contribution is an explicit bound on the convex distance between the two distributions, holding in every dimension. Such a finding constitutes a substantial extension of the one-dimensional bounds deduced in Chatterjee (2007) and Lachi\`eze-Rey and Peccati (2017), as well as of the multidimensional bounds for smooth test functions and indicators of rectangles derived, respectively, in Dung (2019), and Fang and Koike (2021). Our techniques involve the use of Stein's method, combined with a suitable adaptation of the recursive approach inaugurated by Schulte and Yukich (2017): this yields rates of converge that have a presumably optimal dependence on the sample size. We develop several applications of a geometric nature, among which is a new collection of multidimensional quantitative limit theorems for the intrinsic volumes associated with coverage processes in Euclidean spaces. }

\end{abstract}


\end{frontmatter}
\section{Introduction}

\subsection{Overview}
{ Let $X = (X_1,...,X_n)$ be a vector of independent random elements with values in a measurable space $(\mathcal{X}, \mathbb{X})$, and let $f : \mathcal{X}^n \to \mathbb{R}^d$, $d\geq 1$, be a measurable mapping. The aim of this paper is to extend the techniques introduced in \cite{new_stein, peccati_lachieze-rey2017} (in the case $d=1$) in order to provide explicit bounds on the distance between the distribution of the $d$-dimensional random vector $f(X)$, and that of a suitable Gaussian element. Our principal contribution is an explicit bound on the {\bf convex distance} between the two distributions (see Theorem \ref{main3} below), holding for arbitrary values of $d$. As discussed in Section \ref{ss:comp} below, such findings substantially refine and complement the recent estimates from \cite{dung, FK_21}, where bounds on the normal approximation of $f(X)$ are obtained for smooth test functions of class $C^2$ and $C^3$ (see  \cite[Theorems 3.1 and 3.2]{dung}, as well as Theorem \ref{main} of the present paper), and for indicators of hyperrectangles (see \cite[Theorem 1.3]{FK_21}). }

\medskip

{ Our main results are based on the use of the multidimensional {\bf Stein's method} for normal approximations (see e.g. \cite[Chapter 12]{normal_approx}, and \cite[Chapter 4]{nourdin}), which is combined with a recursive method devised for dealing with indicators of convex sets, inaugurated in \cite{schulte_yukich} in the context of geometric functionals of Poisson random measures --- see \cite{lrpy, NPY_jtp, schulte_yukich_21} for further applications of this methodology. An important point (see again Section \ref{ss:comp}) is that the use of the approach of \cite{schulte_yukich} allows one to remove from the bounds all spurious logarithmic dependencies in the parameter $n$ (sample size), such as those appearing e.g. in the main bounds of \cite{FK_21, RR96}. Further references exploiting Stein's method in the context of multivariate normal approximations are \cite{diffusion, meckes, FK_21b, FK_21c, gotze, meckes09, RR96, rinott, reinert_roellin, reinert_roellin1, SS_05}. }

\medskip

We demonstrate the flexibility and scope of our results by developing two applications:

\begin{enumerate}

\item[(i)] in Section \ref{ss:covering}, to a multivariate central limit theorem (CLT) for the intrinsic volumes associated with the Euclidean Boolean model over increasing domains --- thus extending the findings of \cite{new_stein, goldstein_penrose, peccati_lachieze-rey2017} (which contain one-dimensional CLTs Euclidean volumes) and \cite{hug_last_schulte, schulte_yukich} (where multidimensonal CLTs are proved in the context of the Poisson-based Boolean model);

\item[(ii)] in Section \ref{ss:locdep}, to multivariate quantitative CLTs for statistics depending on $k$-nearest neighbour graphs, thus recovering a multivariate version of central results from \cite{new_stein, BB}. 

\end{enumerate}

%
%
%
%

\subsection{Structure of the paper}  Section \ref{section_notation} discusses the notational conventions adopted throughout the paper.  Our main abstract results are stated and discussed in Section \ref{s:main}. Section \ref{ss:covering} and Section \ref{ss:locdep} are devoted, respectively, to our main findings for coverage processes and for structures with local dependence.  The proofs of our abstract results are presented in full detail in Section \ref{proofs_main}, whereas Section \ref{covering process proofs} and Section \ref{proofs_ap} deal, respectively, with the proofs of our results for intrinsic volumes of Boolean models, and the proofs of quantitative CLTs for functionals of structures with local dependence. 
\subsection{Setup and notation}\label{section_notation}
\underline{\it Standard conventions.} For any $d\in\mathbbm{N}$, $\mathbbm{R}^d$ is always equipped with the Euclidean topology, the brackets $\left<\cdot,\cdot\right>$ denote the Euclidean inner product and $\|\cdot\|$ stands for the Euclidean norm (the dependence on the dimension is removed to simplify the notation). The symbol $\left<\cdot,\cdot\right>_{H.S.}$ indicates the Hilbert-Schmidt inner product, $\|\cdot\|_{H.S.}$ is the Hilbert-Schmidt norm and $\|\cdot\|_{op}$ is the (Euclidean) operator norm. Given $d\in\mathbbm{N}$ and a mapping $F:\mathbbm{R}^d\to\mathbbm{R}$ sufficiently smooth, we write $\text{Hess}(F)(x)$ for the Hessian of $F$ at $x\in \mathbbm{R}^d$ and, for $k\geq 0$, we use the symbol $D^kF$ to indicate the $k$-th (Fr\'echet) derivative of $F$, that is: $D^kF(x)$ stands is the $k$-linear form on $\mathbbm{R}^d$ given by $$D^kF(x)[u_1,\dots,u_k]:=\sum_{i_1,\dots,i_k=1}^d\frac{\partial^kF}{\partial x_{i_1}\dots\partial x_{i_k}}(x)(u_1)_{i_1}\dots(u_k)_{i_k},$$
for $x,u_1\dots,u_k\in\mathbbm{R}^d$ and $(u_i)_j$ denoting the $j$-th component of the vector $u_i$. We also set
\begin{equation}\label{e:mnorms}
M_k(F):=\sup_{x\in\mathbbm{R}^d}\left\|D^kF(x)\right\|_{op},\quad \tilde{M}_2(F):=\sup_{x\in\mathbbm{R}^d}\left\|\text{Hess}(F)(x)\right\|_{H.S.}.
\end{equation}
Objects such as $M_k(F)$ and $\tilde{M}_2(F)$ (and variations thereof) play an important role in several applications of the multidimensional Stein's method --- see \cite{DP_ejp, meckes, meckes09, reinert_roellin} for a sample. Every random object appearing below is assumed to be defined on a common probability space $(\Omega, \mathcal{F}, \mathbbm{P})$, with $\mathbbm{E}$ denoting expectation with respect to $\mathbbm{P}$. 

\smallskip

\noindent\underline{\it Convex distance.} Given $d\in\mathbbm{N}$ and random vectors $A,B$ with values in $\mathbbm{R}^d$, we define the {\bf convex distance} between the distributions of $A$ and $B$ as
\begin{equation}\label{e:convex}
d_{\text{convex}}(A,B):=\sup_{h\in\mathcal{I}_d}\left|\mathbbm{E}h(A)-\mathbbm{E}h(B)\right|,
\end{equation}
where $\mathcal{I}_d$ is the set of all indicator functions of measurable convex subsets of $\mathbbm{R}^d$. For $n\in\mathbbm{N}$, we also write $[n]:=\{1,\dots,n\}$.

\smallskip

\noindent\underline{\it Difference operators.} In this paper, we will extensively use some discrete operators defined as the difference between a given function of $X =(X_1,...,X_n)$ (or of some perturbation of it) and the same function computed on some partial resampling of $X$. These operators naturally emerge in functional inequalities connected to Efron-Stein estimates (see e.g. \cite[Chapters 3--5]{BLM}), and are pivotal objects in the theory developed in \cite{new_stein, dung, FK_21, peccati_lachieze-rey2017} and its extensions and applications (mainly via the use of the quantities $T_A$ defined in \eqref{e:ta} below --- see, for instance, \cite{chatterjee_sen, duerinckx, gloria_nolen}). Now fix $n\in\mathbbm{N}$, let $(\mathcal{X}, \mathbb{X})$ be a measurable space and let $X_1,X_2,\dots$ be a sequence of independent $\mathcal{X}$-valued random elements.  Let $X=(X_1,\dots,X_n)$ and write $X'=(X_1',\dots,X_n')$ and $\tilde{X}=(\tilde{X}_1,\dots,\tilde{X}_n)$ to indicate copies of $X$, such that $X,X',\tilde{X}$ are mutually independent.  For each $A\subset [n]$, we define
$$X_i^{A}=\begin{cases}
X_i', &\text{if } i\in A\\
X_i, &\text{if } i\not\in A.
\end{cases}$$
For $j\in [n]$, we write $X^j$ for $X^{\lbrace j\rbrace}$ and, for any $d\in\mathbbm{N}$ and $f:\mathcal{X}^n\to\mathbbm{R}^d$, we set
$$\Delta_jf(X)=f(X)-f(X^j).$$
Given $A\subset [n]$, we introduce the notation
\begin{equation}\label{e:ta}
T_A:=\sum_{j\not\in A}\left[\Delta_jf(X)\right]\left[\Delta_jf(X^A)\right]^{\bf T}\in \mathbbm{R}^{d\times d},
\end{equation}
and
$$T:=\frac{1}{2}\sum_{A\subsetneq [n]}k_{n,A}\, T_A,$$
where 
\begin{equation}\label{e:not}
k_{n,A} := \frac{1}{{n\choose |A|}(n-|A|)}, \quad A\subsetneq [n].
\end{equation}
Finally, for any $d\in\mathbbm{N}$ and function $g:\mathcal{X}^{2n}\to\mathbbm{R}^d$, we set the following notation
$$\tilde{\Delta}_ig(X,X'):=g(X,X')-g\left(\left(X_1,\dots,X_{i-1},\tilde{X}_i,X_{i+1},\dots,X_n\right),X'\right).$$

The following statement from \cite{new_stein} is exploited in several parts of the paper (see e.g. Remark \ref{remark_lemma1} below):
\begin{lemma}[Lemma 2.3 in \cite{new_stein}] \label{lemma1}
For any $g,h:\mathcal{X}^n\to\mathbb{R}$ such that $\mathbb{E}g(X)^2$ and $\mathbb{E}f(X)^2$ are both finite, we have
$${\rm Cov}(g(X),h(X))=\frac{1}{2}\sum_{A\subsetneq [n]}k_{n,A}\sum_{j\not\in A}\mathbb{E}\left[\Delta_jg(X)\Delta_jh(X^A)\right],$$
where $k_{n,A}$ is defined in \eqref{e:not} and ${\rm Cov}$ indicates the covariance.
\end{lemma}

\section{Main results}\label{s:main}
Fix $d\geq 1$ and let $f:\mathcal{X}^n\to\mathbbm{R}^d$ be a measurable function such that the random vector $$W:=f(X)$$ verifies the relations $\mathbbm{E}W=0$ and $\mathbbm{E}\|W\|^2<\infty$. Let $\Sigma\in\mathbbm{R}^{d\times d}$ be non-negative definite and let $N_{\Sigma}$ be $d$-dimensional vector with mean zero and covariance $\Sigma$.
We will now present the main results of the paper: detailed comparisons with the existing literature are gathered together in Section \ref{ss:comp}, whereas proofs can be found in Section \ref{proofs_main}. 
\subsection{General estimates}
We start with a collection of inequalities involving smooth test functions: they are both a multidimensional version of \cite[Theorem 2.2]{new_stein} and an alternative to the bounds appearing in \cite[Theorems 3.1 and 3.2]{dung}. 
\begin{theorem}\label{main}
 For any $F\in C^3(\mathbbm{R}^d)$,
\begin{align}
\left|\mathbbm{E}F(W)-\mathbbm{E}F(N_{\Sigma})\right|\leq \frac{\tilde{M}_2(F)}{2}\mathbbm{E}\left\|\mathbbm{E}[T|X]-\Sigma\right\|_{H.S}+\frac{M_3(F)}{12}\sum_{j=1}^n\mathbbm{E}\left\|\Delta_jf(X)\right\|^3.\label{non_neg_def_bound}
\end{align}
Moreover, if $\Sigma$ is positive definite, then, for any $F\in C^2\left(\mathbbm{R}^d\right)$,
\begin{align}
\left|\mathbbm{E}F(W)-\mathbbm{E}F(N_{\Sigma})\right|\leq \frac{M_1(F)\sqrt{2}}{\sqrt{\pi}}\|\Sigma^{-1}\|_{op}\mathbbm{E}\left\|\mathbbm{E}[T|X]-\Sigma\right\|_{H.S}+\frac{M_2(F)\sqrt{2\pi}}{16}\|\Sigma^{-1}\|_{op}\sum_{j=1}^n\mathbbm{E}\left\|\Delta_jf(X)\right\|^3.\label{pos_def_bound}
\end{align}

\end{theorem}

We now move on to a bound on the convex distance between the distribution of $W$ and that of a Gaussian centered Gaussian random vector $N_\Sigma$ such that $\Sigma$ is invertible. Our estimate is valid under the additional assumption that
\begin{equation}\label{e:a6}
\mathbbm{E}\|\Delta_jf(X)\|^6<\infty, \quad j\in[n],
\end{equation}
and is expressed in terms of the parameters $\gamma_1,\gamma_2, \gamma_3$ and $\gamma_4$ defined as follows:
\begin{align*}
\gamma_1:=&\sum_{j=1}^n\mathbbm{E}\|\Delta_jf(X)\|^3\\
\gamma_2:=&\left(\sum_{j=1}^n\mathbbm{E}\|\Delta_jf(X)\|^4\right)^{1/2}\\
\gamma_3:=&\\
&\!\!\!\!\!\!\left\lbrace\frac{3}{2} \sum_{i=1}^n\mathbbm{E}\left[\left(\sum_{A\subsetneq [n]}k_{n,A}\sum_{j\not\in A}\mathbbm{1}_{[\tilde{\Delta}_i\Delta_jf(X)\neq 0]}\sqrt{\|\Delta_jf(X)\|+\|\tilde{\Delta}_i\Delta_jf(X)\|}\left\|\Delta_{j}f(X^A)\right\|\left\|\Delta_{j}f(X)\right\|\right)^2\right]\right.\\
&+9\sum_{i=1}^n\mathbbm{E}\left[\left(\sum_{A\subsetneq [n]}k_{n,A}\sum_{j\not\in A}\sqrt{\|\Delta_jf(X)\|+\|\tilde{\Delta}_i\Delta_jf(X)\|}\left\|\tilde{\Delta}_i\Delta_{j}f(X^A)\right\|\left\|\Delta_{j}f(X)\right\|\right)^2\right]\\
&+9\sum_{i=1}^n\mathbbm{E}\left[\left(\sum_{A\subsetneq [n]}k_{n,A}\sum_{j\not\in A}\sqrt{\|\Delta_jf(X)\|+\|\tilde{\Delta}_i\Delta_jf(X)\|}\left\|\Delta_{j}f(X^A)\right\|\left\|\tilde{\Delta}_i\Delta_{j}f(X)\right\|\right)^2\right]\\
&+\left.9\sum_{i=1}^n\mathbbm{E}\left[\left(\sum_{A\subsetneq [n]}k_{n,A}\sum_{j\not\in A}\sqrt{\|\Delta_jf(X)\|+\|\tilde{\Delta}_i\Delta_jf(X)\|}\left\|\tilde{\Delta}_i\Delta_{j}f(X^A)\right\|\left\|\tilde{\Delta}_i\Delta_{j}f(X)\right\|\right)^2\right]\right\rbrace^{1/3},\\
\gamma_4:=&\\
&\!\!\!\!\!\!\left\lbrace\frac{3}{2}\sum_{i=1}^n\mathbbm{E}\left[\left(\sum_{A\subsetneq [n]}k_{n,A}\sum_{j\not\in A}\mathbbm{1}_{[\tilde{\Delta}_i\Delta_jf(X)\neq 0]}\sqrt{\|\Delta_jf(X)\|^2+\|\tilde{\Delta}_i\Delta_jf(X)\|^2}\left\|\Delta_{j}f(X^A)\right\|\left\|\Delta_{j}f(X)\right\|\vphantom{\sum_k^l}\right)^2\right]\right.\\
&+\frac{27}{4}\sum_{i=1}^n\mathbbm{E}\left[\left(\sum_{A\subsetneq [n]}k_{n,A}\sum_{j\not\in A}\sqrt{\|\Delta_jf(X)\|^2+\|\tilde{\Delta}_i\Delta_jf(X)\|^2}\cdot\left\|\tilde{\Delta}_i\Delta_{j}f(X^A)\right\|\left\|\Delta_{j}f(X)\right\|\right)^2\right]\\
&+\frac{27}{4}\sum_{i=1}^n\mathbbm{E}\left[\left(\sum_{A\subsetneq [n]}k_{n,A}\sum_{j\not\in A}\sqrt{\|\Delta_jf(X)\|^2+\|\tilde{\Delta}_i\Delta_jf(X)\|^2}\cdot\left\|\Delta_{j}f(X^A)\right\|\left\|\tilde{\Delta}_i\Delta_{j}f(X)\right\|\right)^2\right]\\
&+\left.\frac{27}{4}\sum_{i=1}^n\mathbbm{E}\left[\left(\sum_{A\subsetneq [n]}k_{n,A}\sum_{j\not\in A}\sqrt{\|\Delta_jf(X)\|^2+\|\tilde{\Delta}_i\Delta_jf(X)\|^2}\cdot\left\|\tilde{\Delta}_i\Delta_{j}f(X^A)\right\|\left\|\tilde{\Delta}_i\Delta_{j}f(X)\right\|\right)^2\right]\right\rbrace^{1/4},
\end{align*}
where we have used the notation \eqref{e:not}. 

\medskip

The next statement is the main contribution of our work.

\begin{theorem}\label{main3}
Let $\Sigma\in\mathbbm{R}^{d\times d}$ be positive-definite, suppose that \eqref{e:a6} is verified and let
$$\gamma:=\max\left\lbrace \sqrt{\mathbbm{E}\left\|\mathbbm{E}\left[T-\Sigma|X\right]\right\|_{H.S.}^2},\gamma_1,\gamma_2,\gamma_3,\gamma_4\right\rbrace.$$
 Then,
\begin{align}\label{convex_result}
&d_{convex}(W,N_{\Sigma})\leq 541d^4\max\{1,\|\Sigma^{-1}\|_{op}^2\}\, \gamma.
\end{align}
\end{theorem}

\medskip

\begin{remark}\label{remark_lemma1}{\rm 
If $\Sigma$ is the covariance matrix of $W$, then $\Sigma=\mathbbm{E}T$ (by virtue of Lemma \ref{lemma1}) and 
$$\mathbbm{E}\left\|\mathbbm{E}\left[T-\Sigma|X\right]\right\|_{H.S.}\leq\sqrt{\mathbbm{E}\left\|\mathbbm{E}\left[T-\Sigma|X\right]\right\|_{H.S.}^2}=\sqrt{\sum_{k,l=1}^d\text{Var}\left[\mathbbm{E}[T_{k,l}|X]\right]}$$
and the bounds in Theorems \ref{main} and \ref{main3} simplify accordingly.}
\end{remark}

\subsection{A simplified bound for symmetric statistics}

We will now present a useful statement (Lemma \ref{lemma_bn} below), allowing one to bound the quantity $\mathbbm{E}\left\|\mathbbm{E}\left[T-\mathbbm{E}T|X\right]\right\|_{H.S.}^2$ whenever $f$ is symmetric and the random elements $X_i$ are (independent and) identically distributed. We need some additional notation: denoting by $\{X'_i\}$ an independent copy of $\{X_i\}$, for any random vector $Z=(Z_1,\dots,Z_n)$ and every $A\subset [n]$, we set
\begin{align*}
Z_i^A=\begin{cases}
X_i',&\text{if } i\in A\\
Z_i,&\text{if } i\not\in A,
\end{cases}
\quad
\text{and}\quad
Z^A=\left(Z_1^A,\cdots Z_n^A\right);
\end{align*}
also, for $1\leq i\neq j\leq n$, we write
\begin{align*}
&\Delta_{i}f(Z):=f(Z)-f\left(Z^{\lbrace i\rbrace}\right)\quad\text{and}\quad\Delta_{i,j}f(Z):=f(Z)-f\left(Z^{\lbrace i\rbrace}\right)-f\left(Z^{\lbrace j\rbrace}\right)+f\left(Z^{\lbrace i,j\rbrace}\right).
\end{align*}
We say that $Z=(Z_1,\dots,Z_n)$ is a \textbf{recombination} of $\lbrace X,X',\tilde{X}\rbrace$ if $Z_i\in\lbrace X_i,X_i',\tilde{X}_i\rbrace$ for all $i\in[n]$.
Finally, we let
\begin{align*}
&B_n(f):=\sup_{(Y,Z,Z')}\mathbbm{E}\left[\mathbbm{1}_{\lbrace\Delta_{1,2}f(Y)\neq 0\rbrace}\left\|\Delta_1f(Z)\right\|^2\left\|\Delta_2f(Z')\right\|^2\right];\\
&B_n'(f):=\sup_{(Y,Y',Z,Z')}\mathbbm{E}\left[\mathbbm{1}_{\lbrace \Delta_{1,2}f(Y)\neq 0,\Delta_{1,3}f(Y')\neq 0\rbrace}\|\Delta_2f(Z)\|^2\|\Delta_3f(Z')\|^2\right],
\end{align*}
where the suprema run over all vectors $Y,Y',Z,Z'$ that are recombinations of $\lbrace X,X',\tilde{X}\rbrace$. The proof of the following result can be found in Section \ref{proof_lemma_bn}.

\begin{lemma}[cf. Theorem 5.1 of \cite{peccati_lachieze-rey2017}]\label{lemma_bn}
Suppose that $f$ is symmetric and that $X_1,\dots,X_n$ are i.i.d..Then,
\begin{align*}
\sqrt{\mathbbm{E}\left\|\mathbbm{E}[T|X]-\mathbbm{E}T\right\|_{H.S.}^2}\leq &4\sqrt{n}\left(\sqrt{nB_n(f)}+\sqrt{n^2B_n'(f)}+\sqrt{\mathbbm{E}\left\|\Delta_1f(X)\right\|^4}\right).
\end{align*}
\end{lemma}
\begin{example}\label{ex:clt}{\rm 
Fix $d\geq 1$ and let $X_1,\dots,X_n$ be i.i.d. $\mathbbm{R}^d$-valued random variables with mean zero, covariance matrix $\Sigma$ and satisfying $\mathbbm{E}\|X_1\|^4<\infty$. Let $W=f(X)=\frac{1}{\sqrt{n}}\sum_{i=1}^n X_i$. Adopt the notation of Theorem \ref{main3} and note that
\begin{align}\label{example1}
&\gamma_1=\frac{1}{n^{3/2}}\sum_{j=1}^n\mathbbm{E}\|X_j-X_j'\|^3\leq 8\mathbbm{E}\|X_1\|^3\,n^{-1/2};\quad\gamma_2=\frac{1}{n}\left(\sum_{j=1}^n\mathbbm{E}\|X_j-X_j'\|^4\right)^{1/2}\leq 4\sqrt{\mathbbm{E}\|X_1\|^4}\,n^{-1/2}.
\end{align}
Moreover, note that $\tilde{\Delta}_i\Delta_jf(X)=\mathbbm{1}_{[i=j]}\left(X_j-\tilde{X}_j\right)$ and that, for $A\subsetneq[n]$ and $j\not\in A$, $\Delta_j(X^A)=\Delta_j(X)$. It follows that, if $\mathbbm{E}\|X_1\|^6<\infty$ then, for $p=1,2$,
\begin{align}
\gamma_{p+2}=&\left\lbrace\frac{3}{2}\sum_{j=1}^n\mathbbm{E}\left[\left(\|X_j-X_j'\|^p+\|X_j-\tilde{X}_j\|^p\right)\left\|X_j-X_j'\right\|^4\right]n^{-2-p/2}\right.\notag\\
&+2\left(\frac{9}{2}+\frac{9}{2p}\right)\sum_{j=1}^n\mathbbm{E}\left[\left(\|X_j-X_j'\|^p+\|X_j-\tilde{X}_j\|^p\right)\left\|X_j-\tilde{X}_j\right\|^2\left\|X_j-X_j'\right\|^2\right]n^{-2-p/2}\notag\\
&\left.+\left(\frac{9}{2}+\frac{9}{2p}\right)\sum_{j=1}^n\mathbbm{E}\left[\left(\|X_j-X_j'\|^p+\|X_j-\tilde{X}_j\|^p\right)\left\|X_j-\tilde{X}_j\right\|^4\right]n^{-2-p/2}\right\rbrace^{1/(p+2)}\notag\\
\leq &2^{1+1/(p+2)}\left(30+\frac{27}{p}\right)^{1/(p+2)}\left(\mathbbm{E}\|X_1\|^{p+4}\right)^{1/(p+2)}n^{-1/2}.\label{example2}
\end{align}
Furthermore, since $f$ is symmetric, we can apply Lemma \ref{lemma_bn}. Adopting the notation thereof and letting $Y$ be a recombination of $\{X,X',\tilde{X}\}$, we have that
\begin{align*}
\Delta_{1,2}f(Y)=\frac{1}{\sqrt{n}}\left[\left(\sum_{i=1}^n Y_i\right)-\left(X_1'+\sum_{i=2}^nY_i\right)-\left(X_2'+Y_1+\sum_{i=3}^n Y_i\right)+\left(X_1'+X_2'+\sum_{i=3}^n Y_i\right)\right]=0.
\end{align*}
Therefore, noting that $\Sigma=\mathbbm{E}T$ (see Remark \ref{remark_lemma1}), we have that
\begin{align}\label{example3}
\sqrt{\mathbbm{E}\left\|\mathbbm{E}[T|X]-\Sigma\right\|_{H.S.}^2}\leq 16\sqrt{\mathbbm{E}\|X_1\|^4}\,n^{-1/2}.
\end{align}
Finally, combining Theorems \ref{main} and \ref{main3} with (\ref{example1})---(\ref{example3}) and letting $N_{\Sigma}$ be a centred normal random variable with covariance $\Sigma$, we obtain:
\begin{enumerate}
\item For $F\in C^3(\mathbbm{R}^d)$, \begin{equation} \label{e:clt1} | \mathbbm{E}[F(W)]-\mathbbm{E}[F(N_{\Sigma})] |\leq \left(8\tilde{M}_2(F)\sqrt{\mathbbm{E}\|X_1\|^4}+\frac{3M_3(F)}{4}\mathbbm{E}\|X_1\|^3\right)n^{-1/2};\end{equation}
\item If $\Sigma$ is positive definite then, for $F\in C^2(\mathbbm{R}^d)$, 
 \begin{equation} \label{e:clt2}|\mathbbm{E}[F(W)]-\mathbbm{E}[[F(N_{\Sigma})] |\leq \|\Sigma^{-1}\|_{op}\left(\frac{16\sqrt{2}\,M_1(F)}{\sqrt{\pi}}\sqrt{\mathbbm{E}\|X_1\|^4}+\frac{M_2(F)\sqrt{2\pi}}{2}\mathbbm{E}\|X_1\|^3\right)n^{-1/2};\end{equation}
\item If $\Sigma$ is positive definite and $\mathbbm{E}\|X_1\|^6<\infty$ then
 \begin{equation} \label{e:clt3}d_{convex}(W,N_{\Sigma})\leq 8656\,d^4\,\max\{1,\|\Sigma^{-1}\|_{op}^2\}\max\left\lbrace \sqrt{\mathbbm{E}\|X_1\|^4}, \,\mathbbm{E}\|X_1\|^3,\,\left(\mathbbm{E}\|X_1\|^{5}\right)^{1/3},\,\left(\mathbbm{E}\|X_1\|^{6}\right)^{1/4}\right\rbrace n^{-1/2}.\end{equation}
\end{enumerate}
As a consequence, in the three cases our estimates yield bounds on normal approximations converging to zero with an optimal rate which is commensurate to the square-root of the sample size. }
\end{example}

\subsection{Literature review} \label{ss:comp} As already discussed in the Introduction, our main theoretical bounds \eqref{non_neg_def_bound}, \eqref{pos_def_bound} and \eqref{convex_result} generalise the one-dimensional estimates stated in \cite[Theorem 2.2]{chatterjee} (1-dimensional Wasserstein distance) and \cite[Theorem 4.2]{peccati_lachieze-rey2017} (1-dimensional Kolmogorov distance). One noticeable difference between \eqref{convex_result} and the content of \cite[Theorem 4.2]{peccati_lachieze-rey2017} is that the latter allows one to consider random variables such that $\mathbbm{E} | \Delta_jf(X)|^4<\infty$ (which is weaker than \eqref{e:a6} above). One could also similarly relax the moment assumptions in our bounds, at the cost of considerably longer proofs: since this point would impact only marginally the scope of our applications, we decided not to pursue it further.  One can check that, in dimension $d=1$, relation \eqref{convex_result} yields upper bounds on the Kolmogorov distance that are commensurate to \cite[equation (4.2)]{peccati_lachieze-rey2017} (which requires integrability assumptions analogous to \eqref{e:a6}). 

In the multidimensional case, the most relevant references related to our work are \cite{dung, FK_21}. The proof and statements of Theorem 3.1 and 3.2 in \cite{dung} deal with the case of smooth test functions, and are close to our estimates \eqref{non_neg_def_bound}---\eqref{pos_def_bound} (albeit expressed in terms of slightly different quantities). Although we did not carry out the computations in full detail, it is reasonable to expect that, in the case of smooth test functions, the upper bounds from \cite{dung} would yield rates of convergence similar to ours in all applications we will develop below. 

Theorem 1.2 in \cite{FK_21} provides an upper bound on the quantity 
$$
{\bf D} := \sup_{R} \big| \mathbbm{P}[W\in R] - \mathbbm{P}[ N_\Sigma\in R] \big|, 
$$
where $R$ runs over the class of all rectangles of $\mathbbm{R}^d$, in such a way that ${\bf D}\leq d_{\rm convex} (W, N_\Sigma)$. The bounds on ${\bf D}$ derived in \cite{FK_21} are expressed in terms of the same difference operators appearing on the right-hand side of \eqref{convex_result}, and are obtained by combining G\"otze's `generator approach' towards Stein's method \cite{gotze} with some smoothing estimates from \cite{batrao}. Apart from the fact that the class of convex sets is strictly larger than that of rectangles, the main differences between \cite[Theorem 1.2]{FK_21} and \eqref{convex_result} are the following: (i) because of the involved smoothing techniques, the bounds from \cite{FK_21} tend to produce a suboptimal dependence on the sample size $n$, typically displaying additional multiplicative factors exploding at a logarithmic rate (see e.g. \cite[Corollary 1.3]{FK_21}), whereas our bounds have been devised in order to produce presumably optimal rates (that is, rates proportional to the inverse of the square root of the variance of $W$, as $n\to \infty$); (ii) the upper bounds deduced in \cite{FK_21} typically display a logarithmic-type dependence on the dimension $d$, which is in stark contrast with the $d^4$ prefactor appearing in our inequalities; (iii) the estimates from \cite{FK_21} require weaker integrability assumptions on the random vectors $\Delta_jf(X)$ (typically, finite fourth moments vs. finite sixth moments in our case). The better performance of our bounds described at Point (i) is a direct consequence of the recursive approach initiated in \cite{schulte_yukich}, and further developed in the present work (see also \cite{NPY_jtp, lrpy}). The discrepancy at Point (ii) is somehow consistent with the fundamental geometric fact that the Gaussian isoperimetric constant  for convex sets increases as a power of $d$, whereas the one of rectangles is bounded from above by a multiple of $\sqrt{\log d}$ (see e.g. \cite{nazarov, raic}, as well as \cite[Appendix A]{NPY_jtp}), and is likely to be unavoidable; note that a similar dependence on the dimension (of the order of $d^5$) also appears in the main bounds developed in \cite{schulte_yukich}. As already observed in the case of smooth test functions, one could in principle push further the approach developed in the present work (at the cost of considerably longer proofs) in order to relax the integrability requirements on the random vectors $\Delta_jf(X)$, thus addressing Point (iii).

It is also natural to compare the specific content of Example \ref{ex:clt} with the best available Berry-Esseen bounds for the classical multidimensional CLT, see e.g. \cite{bentkus, FK_21c}. In this case, our estimate \eqref{e:clt3} yields an upper bound converging to zero at the optimal rate $n^{-1/2}$, displaying nonetheless a suboptimal dimensional dependence of the order $d^4$ (to be contrasted with Bentkus' famous bounds from \cite{bentkus}, featuring a dimensional dependence of the order $d^{1/4}$), as well as requiring stronger integrability assumptions (i.e., existence of 6th moments). In principle, these shortcomings are a consequence of the fact that our methods are not optimised for dealing with the case of linear statistics, which often require careful ad-hoc arguments --- see e.g. \cite{bentkus, FK_21, FK_21c}. We refer the reader to \cite[Section 1.1]{FK_21}, and the references therein, for a detailed discussion of the state-of-the-art literature on the matter.


\section{Applications to coverage processes}\label{ss:covering}
We will now use our main estimates in order to assess the joint fluctuations of geometric quantities associated with a binomial version of {\bf the (Euclidean) Boolean model} in the so-called {\bf thermodynamic regime} (that is, with grains of a fixed radius). Our main results are gathered together in the statements of Theorem \ref{pos_def_teorem} and Theorem \ref{theorem_covering} below, whose proofs are postponed to Section \ref{covering process proofs}. As explained in Remark \ref{r:covlit}, our results are the first quantitative CLTs for geometric functionals attached to the binomial Boolean model, and should be compared with recent advances in a Poissonian setting \cite{hug_last_schulte, schulte_yukich}. We refer the reader to the monographs \cite{hall_book, penrose, geometry}, as well as \cite[Chapters 16 and 17]{last_penrose}, for a general introduction to Boolean models and for an overview of the history of the subject. All needed geometric notions (in particular, related to convex geometry and intrinsic volumes) can be found in \cite[Chapter 14]{geometry}.

Fix $d\geq 1$. For every $n\geq 1$, we write $E_n$ to indicate a cube of volume $n$ in $\mathbbm{R}^d$ and let $C_1,\dots,C_n$ denote i.i.d. random variables uniformly distributed on $E_n$ (customarily called \textbf{germs}). Let $K=B(0,R)\subset\mathbbm{R}^d$ be the closed ball of radius $R>0$ centered at the origin and define $X_i :=C_i+K$, $i=1,\dots,n$ and $X:=( X_1,\dots,X_n)$.
We consider the random set --- called the {\bf Boolean model} associated with $K$ and $C_1,...,C_n$ ---  formed by the union of $K$ translated by the germs, that is: 
$$F_n=\bigcup_{k=1}^n X_k.$$
Observe that, by construction, $F_n$ is a random element with values in the {\bf convex ring} $\mathcal{R}^d$ of $\mathbbm{R}^d$, defined as the collection of all finite unions of convex compact subsets of $\mathbb{R}^d$.

As anticipated, our aim is to study the fluctuations of the scaled joint law of the intrinsic volumes of $F_n$, i.e. the law of $f(X_1,\dots,X_n)$, for $f=(f_0,\dots,f_d)$ defined by
$$f_i(X_1,\dots,X_n)=\frac{1}{\sqrt{n}}\left(V_i(F_n)-\mathbbm{E}V_i(F_n)\right)$$
with $V_i$ denoting the $i$th {\bf intrinsic volume} (observe that the definition of $f$ changes with $n$, and that such a dependence is omitted for notational clarity). We recall (see \cite[Chapter 14.2]{geometry} for more details) that each $V_i$ is an additive real-valued mapping on $\mathcal{R}^d$, and that $V_d(B),\,  V_{d-1}(B)$ and $V_0(B)$ coincide, respectively, with the volume, half the surface area and the Euler characteristic of $B\in \mathcal{R}^d$. We will also use the so-called {\bf Wills functional} $\overline{V}$, which is defined for all compact convex sets $A$ as
$$\overline{V}(A)=\sum_{l=0}^d\kappa_{d-l}V_l(A),$$
with $\kappa_{d-l}$ denoting the volume of the $(d-l)$-dimensional unit ball. From now on, we write
$$W:=f(X)$$
and let $\Sigma_n$ denote the covariance matrix of $W$. Finally, we let $\Sigma\in\mathbbm{R}^{(d+1)\times(d+1)}$ be given by
\begin{align}
\Sigma_{i,j}:=&\sum_{k=2}^{\infty}\frac{1}{k!}\sum_{s=i}^d\sum_{r=j}^d\int_{\mathbbm{R}^d}\dots\int_{\mathbbm{R}^d}P_{i,s}(d)P_{j,r}(d)V_s(K\cap(K+x_2)\cap\dots\cap(K+x_k))\notag\\
&\hspace{5cm}\cdot V_r(K\cap(K+x_2)\cap\dots\cap(K+x_k))dx_2\dots dx_k,\quad i,j\in\{0,\dots,d\},\label{matrix_sigma}
\end{align}
where, for $i=0,\dots,d-1$ and $s=i,\dots,d$,
$$P_{i,s}(d):=e^{-V_d(K)}\Bigg[\mathbbm{1}_{[s=i]}+\frac{s!\kappa_s}{i!\kappa_i}\sum_{t=1}^{s-i}\frac{(-1)^t}{t!}\underset{r_1+\dots+r_t=td+i-s}{\sum_{i\leq r_1,\dots,r_t\leq d-1}}\prod_{m=1}^t\frac{r_m!\kappa_{r_m}}{d!\kappa_d}V_{r_m}(K)\Bigg]\quad\text{and}\quad P_{d,d}(d):=e^{-V_d(K)}.$$ 
The following result provides a useful upper bound on the speed of convergence of $\Sigma_n$ to $\Sigma$. 

\begin{theorem}\label{pos_def_teorem}
The matrix $\Sigma$ defined above is positive definite and for all $n\in\mathbbm{N}$, such that $n>\overline{V}(K)e$, and all $i,j=0,\dots,d$,
\begin{align*}
\left|(\Sigma_n)_{i,j}-\Sigma_{i,j}\right|\leq 116\cdot 108^d\cdot(R+1)^{4d}e^{6\cdot 9^d\cdot(R+1)^{2d}}(d!)^2n^{-1/d}.
\end{align*}
\end{theorem}

\medskip

The quantitative CLTs featured in the next statement are the main achievement of the present section. For $F\in C^2(\mathbbm{R}^{d+1})$, we adopt the notation
$$
M(F) := \max\{M_1(F),M_2(F)\};
$$
see \eqref{e:mnorms}.

\medskip

\begin{theorem}\label{theorem_covering}
Let $N_{\Sigma}$ denote a centred $(d+1)$-dimensional normal random vector with covariance $\Sigma$ given by \eqref{matrix_sigma}, and let $N_{\Sigma_n}$ denote a centred normal random variable with covariance $\Sigma_n$. Let $F\in C^2(\mathbbm{R}^{d+1})$. For $n>3^{d+1}(R+1)^d$, we have that, for $a = a(d) : = \min \{1, d/2\}$,
\begin{align*}
&\left|\mathbbm{E}F(W)-\mathbbm{E}F(N_{\Sigma})\right|\leq 11\cdot M(F)\cdot 81^d\cdot(R+1)^{4d}(d+2)^3\left(2^{2^d+2d}d^{d/2}\right)^3(d!)^3e^{6\cdot 9^d(R+1)^{2d}}\|\Sigma^{-1}\|_{op}n^{-a/d};\\
&d_{convex}(W,N_{\Sigma})\leq 5\cdot 10^5\cdot 108^d\cdot(R+1)^{4d}(d+1)^7\left(2^{2^d+2d}d^{d/2}\right)^3e^{102\cdot 9^d(R+1)^{2d}}(d!)^3\max\{1,\|\Sigma^{-1}\|_{op}^2\}n^{-a/d}.
\end{align*}
Moreover, there exists $N\in\mathbbm{N}$, such that for all $n\geq N$,
\begin{align*}
&\left|\mathbbm{E}F(W)-\mathbbm{E}F(N_{\Sigma_n})\right|\leq 11 \cdot M(F) \cdot 81^d\cdot(R+1)^{4d}(d+2)^3\left(2^{2^d+2d}d^{d/2}\right)^3(d!)^3e^{8(2R+1)^{2d}}\|\Sigma_n^{-1}\|_{op}n^{-1/2};\\
&d_{convex}(W,N_{\Sigma_n})\leq 4\cdot 10^5\cdot 81^d\cdot (R+1)^{4d}(d+1)^7\left(2^{2^d+2d}d^{d/2}\right)^3e^{456(2R+1)^d}(d!)^3\max\{1,\|\Sigma_n^{-1}\|_{op}^2\}n^{-1/2}.
\end{align*}
\end{theorem}
\begin{remark}{\rm The quantity 
$\|\Sigma^{-1}\|_{op}$ is equal to the inverse of the smallest eigenvalue of $\Sigma$, which may be bounded using techniques presented, for instance, in \cite{eigenvalue}.}
\end{remark}

\begin{remark}[Literature review]\label{r:covlit}{\rm One-dimensional quantitative CLTs in the Kolmogorov distance for the volume and the number of isolated balls of the binomial Boolean model are proved in \cite{goldstein_penrose} and \cite[Section 6.1]{peccati_lachieze-rey2017}, respectively, in the case of a deterministic radius $R$ (as in our setting) and in the case of i.i.d. random radii satisfying some adequate integrability assumptions. Similar bounds in the 1-Wasserstein distance for the fixed radius case are proved in \cite[Section 3.2]{new_stein}. These results are quantitative versions of classical CLTs proved e.g. in \cite{moran, PY01, penrose}. The results from \cite{new_stein, goldstein_penrose, peccati_lachieze-rey2017} are all obtained by using Stein's method, and yield presumably optimal rates of convergence (that is, rates proportional to $n^{-1/2}$, as in the second part of our Theorem \ref{theorem_covering}). The approach towards Stein's method adopted in \cite{goldstein_penrose} is based on the use of size-biased couplings, whereas references \cite{new_stein, peccati_lachieze-rey2017} exploit discrete integration by parts formulae based on the use of the same discrete operators $\Delta_j$ as those appearing in the present work. It is in principle possible to extend our results in order to accommodate the case of random radii; for the sake of conciseness (in particular, in view of the technical nature of the proofs) we will tackle this point elsewhere. Plainly, our Theorem \ref{theorem_covering} also yields one-dimensional CLTs for the intrinsic volumes $V_i(F_n)$, for all $i=0,1,...,d$; to the best of our knowledge, the CLTs in the case $ 0\leq i < d$ are new even in their qualitative versions.

Multidimensional quantitative CLTs for the intrinsic volumes (and more general geometric functionals) of the Poisson-based Boolean model are proved in \cite[Theorem 9.1]{hug_last_schulte} and \cite[Section 4.2]{schulte_yukich}, respectively, in the case of smooth test functions, and of the convex distance adopted in our work. In both instances, the authors are able to deal with random radii satisfying adequate moment conditions, and obtain rates of convergence of the same order as those displayed in Theorem \ref{theorem_covering} above (to see this, observe that the inradius of the cube $E_n$ is commensurate to $n^{1/d}$). The prefactors appearing in the bounds from \cite{schulte_yukich} also display super-exponential dependence on the dimensional parameter $d$. The limiting covariance structure in formula \eqref{matrix_sigma} above -- expressed as an infinite sum -- is also featured in \cite[formula (3.4)]{hug_last_schulte}, with the only difference that the order of summation there starts from $k=1$. We stress once again that the recursive approach for dealing with the distance $d_{convex}$ initiated in \cite{schulte_yukich} is one of the most crucial ingredients of our approach.}
\end{remark}

\section{An application to structures with local dependence}\label{ss:locdep}
In this section we extend the results of \cite[Subsection 2.3]{new_stein} --- about the normal approximation of random quantities based on structures with {\bf local dependence} --- to the multivariate setting, with specific focus on approximations in the convex distance. Proofs are gathered together in Section \ref{proofs_ap}. 
\medskip

We adopt the setup of Section \ref{section_notation}, fix $d\geq 1$ and denote by $G$ a \textbf{graphical rule}, that is, a map which associates with every $x\in\mathcal{X}^n$ an undirected graph $G(x)$ on $[n]:=\lbrace 1,\dots,n\rbrace$. We call $G$ \textbf{symmetric} if, for any permutation $\pi$ of $[n]$ and any $(x_1,\dots,x_n)\in\mathcal{X}^n$, the set of edges in $G\left(x_{\pi(1)},\dots,x_{\pi(n)}\right)$ is given by
$$\left\lbrace\lbrace \pi(i),\pi(j)\rbrace:\lbrace i,j\rbrace\in G(x_1,\dots,x_n)\right\rbrace.$$
For $m>n$, we say that a vector $x\in\mathcal{X}^n$ is {\bf embedded} in $y\in\mathcal{X}^m$ if there exist distinct $i_1,\dots,i_n\in [m]$ with $x_k=y_{i_k}$ for $1\leq k\leq n$. We call a graphical rule $G'$ an {\bf extension} of $G$ if, for any $x\in\mathcal{X}^n$ embedded in $y\in\mathcal{X}^m$, the graph $G(x)$ on $[n]$ is the naturally induced subgraph of the graph $G'(y)$ on $[m]$.

Now, taking $x,x'\in\mathcal{X}^n$ and $i\in [n]$, we let $x^i$ be the vector obtained by replacing $x_i$ with $x_i'$ in the vector $x$. For any two distinct elements $i$ and $j$ of $[n]$, we let $x^{ij}$ be the vector resulting from replacing $x_i$ with $x_i'$ and $x_j$ with $x_j'$. We also say that the coordinates $i$ and $j$ are \textbf{noninteracting} under the triple $(f,x,x')$ if
$$f(x)-f(x^j)=f(x^i)-f(x^{ij}).$$

Moreover, a graphical rule $G$ is called an \textbf{interaction rule} for a function $f$ if, for any $x,x'\in\mathcal{X}^n$ and $i,j\in[n]$, the event that $\lbrace i,j\rbrace$ is not an edge in the graphs $G(x),G(x^i),G(x^j)$ and $G(x^{ij})$ implies that $i$ and $j$ are noninteracting vertices under $(f,x,x')$. The following result, which is  one of the main achievements of the present section, is a multivariate analogue of \cite[Theorem 2.5]{new_stein} in the convex distance:
\begin{theorem}\label{main1}
Let $f:\mathcal{X}^n\to\mathbbm{R}^d$ be a measurable map that admits a symmetric interaction rule $G$ and let $M=\max_j\|\Delta_jf(X)\|$. Let $G'$ be an arbitrary symmetric extension of $G$ on $\mathcal{X}^{n+4}$ and set
$$\delta:=1+\text{degree of vertex 1 in } G'(X_1,\dots, X_{n+4}).$$
Write $W=f(X)$ and assume $\mathbbm{E}W=0$ and $\mathbbm{E}\|W\|^2<\infty$. Let $\Sigma$ denote the covariance matrix of $W$ and $N_{\Sigma}$ be a centred Gaussian $d$-dimensional vector with covariance $\Sigma$. For any $F\in C^3(\mathbbm{R^d})$, one has that
\begin{align}
\label{smooth_local}
\left|\mathbbm{E}F(W)-\mathbbm{E}F(N_{\Sigma})\right|\leq C\tilde{M}_2(F)\left(\mathbbm{E}\left(M^8\right)\right)^{1/4}\left(\mathbbm{E}\left(\delta^4\right)\right)^{1/4} n^{1/2}+\frac{M_3(F)}{12}\sum_{j=1}^n\mathbbm{E}\left\|\Delta_jf(X)\right\|^3,
\end{align}
for a universal constant $C>0$, not depending on $d$ or $n$.
If, in addition, $\Sigma$ is positive definite, then, for any $F\in C^2\left(\mathbbm{R}^d\right)$,
\begin{align}\label{smooth_local1}
&\left|\mathbbm{E}F(W)-\mathbbm{E}F(N_{\Sigma})\right|
\leq CM_1(F)\|\Sigma^{-1}\|_{op}\left(\mathbbm{E}\left(M^8\right)\right)^{1/4}\left(\mathbbm{E}\left(\delta^4\right)\right)^{1/4} n^{1/2}+\frac{M_2(F)\sqrt{2\pi}}{16}\|\Sigma^{-1}\|_{op}\sum_{j=1}^n\mathbbm{E}\left\|\Delta_jf(X)\right\|^3,
\end{align}
and
\begin{align}\label{convex_local}
&d_{convex}(W,\Sigma^{1/2}Z)
\leq Cd^4\max\left(1,\|\Sigma^{-1}\|_{op}^2\right)\max \left\lbrace\left(\mathbbm{E}\left(M^8\right)\right)^{1/4}\left(\mathbbm{E}\left(\delta^4\right)\right)^{1/4}n^{1/2},\gamma_1,\gamma_2,\tilde{\gamma}_3,\tilde{\gamma}_4\right\rbrace,
\end{align}
where $C>0$ is a universal constant not depending on $d$ or $n$, the quantities $\gamma_1,\gamma_2$ are defined as in Theorem {\rm \ref{main3}} and
\begin{align*}
\tilde{\gamma}_3=\left(\mathbbm{E}\left(M^{10}\right)\right)^{1/6}\left(\mathbbm{E}\left(\delta^4\right)\right)^{1/6}n^{1/3};\qquad
\tilde{\gamma}_4=\left(\mathbbm{E}\left(M^{12}\right)\right)^{1/8}\left(\mathbbm{E}\left(\delta^4\right)\right)^{1/8}n^{1/4}.
\end{align*}
\end{theorem}
We now present a result directly applying Theorem \ref{main1} to the normal approximation of nearest-neighbour statistics (cf.  \cite[Theorem 3.4]{new_stein}). Such a result represents a quantitative multivariate extension of Bickel \& Breiman's famous CLT for nearest neighbour statistics (see \cite{BB}, as well as \cite[Section 3.4]{new_stein}). To the best of our knowledge, the content of Theorem \ref{main2} represents the first quantitative multivariate extension of the findings of \cite{BB}.

\begin{theorem}\label{main2}
 Fix integers $m,k\geq 1$. Suppose  $X_1,\cdots,X_n$ are i.i.d $\mathbbm{R}^m$-valued random variables and that $\|X_1-X_2\|$ is a continuous random variable. Let $f:\left(\mathbbm{R}^m\right)^n\to\mathbbm{R}^d$ be a function taking the form
\begin{equation}\label{function}
f(x_1,\dots,x_n)=\frac{1}{\sqrt{n}}\sum_{l=1}^n f_l(x_1,\dots,x_n),
\end{equation}
where, for each $l$, the value $f_l(x_1,\dots, x_n)$ depends only on $x_l$ and its $k$ nearest neighbours. Suppose that\\ $\eta_p:=\max_l\mathbbm{E}\left\|f_l(X_1,\dots,X_n)\right\|^p$ is finite for some $p\geq 8$. Let $W=f(X_1,\dots,X_n)$ and assume $\mathbbm{E}W=0$ and $\mathbbm{E}\|W\|^2<\infty$. Let $\Sigma$ denote the covariance matrix of $W$ and $N_{\Sigma}$ be a centred Gaussian $d$-dimensional vector with covariance $\Sigma$. For any $F\in C^3(\mathbbm{R}^d)$ with bounded second and third derivative, 
\begin{align}\label{smooth_main2}
\left|\mathbbm{E}F(W)-\mathbbm{E}F(N_{\Sigma})\right|\leq C\frac{\alpha(d)^3k^4\eta_p^{2/p}}{n^{(p-8)/(2p)}}+C\frac{\alpha(d)^3k^3\eta_p^{3/p}}{n^{(p-6)/(2p)}},
\end{align}
where $\alpha(d)$ is the minimum number of $60^{\circ}$ cones at the origin required to cover $\mathbbm{R}^d$ and $C$ is a universal constant. The same conclusion holds if $\Sigma$ is positive definite and $F\in C^2(\mathbbm{R^d})$ with bounded first and second derivative.

Furthermore, if, for some $p\geq 12$, $\eta_p:=\max_l\mathbbm{E}\left\|f_l(X_1,\dots,X_n)\right\|^p$ is finite and $\Sigma$ is positive definite then
\begin{align}
d_{convex}(W,N_{\Sigma})
\leq& \tilde{C}d^4\max\left(1,\|\Sigma^{-1}\|_{op}^2\right)\notag\\
&\cdot\max\left(\frac{\alpha(d)^3k^4\eta_p^{2/p}}{n^{(p-8)/(2p)}},\frac{\alpha(d)^3k^3\eta_p^{3/p}}{n^{(p-6)/(2p)}},\frac{\alpha(d)^2k^2\eta_p^{2/p}}{n^{(p-4)/(2p)}},\frac{\alpha(d)^{10/3}k^{4/3}\eta_p^{5/(3p)}}{n^{(3p-20)/(6p)}},\frac{\alpha(d)^{2}k^{5/2}\eta_p^{3/(2p)}}{n^{(p-6)/(2p)}}\right),\label{convex_main2}
\end{align}
for a universal constant $\tilde{C}$, not depending on $n$ or $d$.
\end{theorem}
\begin{remark}{\rm 
	Our bounds for smooth test functions in (\ref{smooth_local}) and (\ref{smooth_local1}) are of exactly the same order as the bound in \cite[Theorem 2.5]{new_stein}. The convex-distance bound (\ref{convex_local}) features additional terms which are somewhat difficult to compare to \cite[Theorem 2.5]{new_stein}. It is, however, possible to draw such useful comparisons in the  context of the specific application covered by Theorem \ref{main2}. The smooth distance bound in (\ref{smooth_main2}) is again of the same order as the bound in \cite[Theorem 3.4]{new_stein}. Indeed, concentrating only on the dependence on $n$, the bound is of order 
	\begin{align}\label{smooth_order}
	\max(\eta_p^{3/p}  n^{3/p-1/2}, \eta_p^{2/p}  n^{4/p-1/2}).
	\end{align}
	In order to extend our result in Theorem \ref{main2} to the convex distance case, we require a stronger integrability assumption (we need $\eta_p$ to exist for some $p\geq 12$, while in our and \cite{new_stein} smooth distance bounds, it is enough to have it finite for some $p\geq8$). If we assume that $\eta_p$ is non-decreasing   with $n$ then our convex-distance bound (\ref{convex_main2}) is of the same order in $n$ as the smooth-distance bound (\ref{smooth_main2}) and the smooth-distance bound of  \cite[Theorem 3.4]{new_stein}. If $\eta_p$ decreases with $n$ then we get a bound of order $\max(\eta_p^{2/p}  n^{4/p-1/2}, \eta_p^{3/(2p)}  n^{3/p-1/2}, \eta_p^{5/(3p)}  n^{10/(3p)-1/2})$ and it might be different from (\ref{smooth_order}) (depending on how fast $\eta_p$ decreases). 
	
\smallskip	
	
Following the lines of \cite[Section 3.4]{new_stein}, one could use Theorem \ref{main2} in order to deduce quantitative multivariate CLTs for several classes of statistics --- like e.g. vertex degrees in random geometric graphs \cite[Chapter 4]{penrose} or dimensional estimators \cite{LB}. Details are omitted so as to keep the length of the paper within bounds.

	}

\end{remark}

\section{Proofs of abstract results}\label{proofs_main}
\subsection{Proof of Theorem \ref{main}}
\subsubsection{Introduction}
Fix a function $F\in C^2(\mathbbm{R}^d)$ and an $\epsilon\in\mathbbm{R}$. Let $\varphi_{\epsilon^2 I}$ be the density of a $d$-dimensional Gaussian vector with covariance matrix $\epsilon^2I$. Let $F_{\epsilon}=F\star\varphi_{\epsilon^2 I}$ and note that $F_{\epsilon}\in C^{\infty}(\mathbbm{R}^d)$ and $\sup_{x\in\mathbbm{R^d}}\left|F_{\epsilon}(x)-F(x)\right|\xrightarrow{\epsilon\to 0}0$. In order to upper-bound $\left|\mathbbm{E}F(W)-\mathbbm{E}F(N_{\Sigma})\right|$, we will first upper-bound $\left|\mathbbm{E}F_{\epsilon}(W)-\mathbbm{E}F_{\epsilon}(N_{\Sigma})\right|$. This will be achieved by an application of \cite[Lemma 1]{meckes09}, which implies that there exists a function $\rho_{\epsilon}\in C^{\infty}(\mathbbm{R}^d)$ solving the corresponding Stein equation for $N_{\Sigma}$ and test function $F_{\epsilon}$, i.e. satisfying
\begin{align}
F_{\epsilon}(y)-\mathbbm{E}F_{\epsilon}(N_{\Sigma})=\left<\nabla \rho_{\epsilon}(y),y\right>-\left<\text{Hess}(\rho_{\epsilon}(y)),\Sigma\right>_{H.S.},\quad\text{for all }y\in\mathbbm{R}^d.\label{stein_epsilon}
\end{align}
We shall also apply \cite[Lemma 2]{meckes09} stating that, for any non-negative definite $\Sigma$,
\begin{align}
&M_k(\rho_{\epsilon})\leq \frac{1}{k}M_k(F_{\epsilon}),\,\,\forall k\geq 1;\quad \text{and}\quad \tilde{M}_2(\rho_{\epsilon})\leq\frac{1}{2}\tilde{M}_2(F_{\epsilon})
\end{align}
and for positive-definite $\Sigma$,
\begin{align}
&M_1(\rho_{\epsilon})\leq M_0(F_{\epsilon})\|\Sigma^{-1/2}\|_{op}\sqrt{\frac{\pi}{2}};\quad \tilde{M}_2(\rho_{\epsilon})\leq M_1(F_{\epsilon})\|\Sigma^{-1/2}\|_{op}\sqrt{\frac{2}{\pi}};\quad M_3(\rho_{\epsilon})\leq M_2(F_{\epsilon})\|\Sigma^{-1/2}\|_{op}\frac{\sqrt{2\pi}}{4}.\label{der_bounds}
\end{align}
We now prove a lemma which will let us conclude the argument.
\subsubsection{Auxiliary lemma}
\begin{lemma}[c.f. Lemma 2.4 and Lemma 4.1 of \cite{new_stein}]\label{lemma2}
For any $\rho\in C^3(\mathbbm{R}^d)$, with bounded second and third derivatives, we have
$$|\mathbbm{E}\left<\nabla\rho(W),W\right>-\mathbbm{E}\left<\text{Hess}(\rho(W)),\Sigma\right>_{H.S.}|\leq\tilde{M}_2(\rho)\mathbbm{E}\left\|\mathbbm{E}[T|W]-\Sigma\right\|_{H.S}+\frac{M_3(\rho)}{4}\sum_{j=1}^n\mathbbm{E}\left\|\Delta_jf(X)\right\|^3.$$
\end{lemma}
\begin{proof}
For each $A\subset[n]$ and $j\not\in A$, let
$$R_{A,j}:=\sum_{i=1}^d\Delta_j\left(\frac{\partial\rho}{\partial x_i}\circ f\right)(X)\cdot\Delta_jf_i(X^A)\in\mathbbm{R};\qquad\tilde{R}_{A,j}=\sum_{i,k=1}^d\text{Hess}_{i,k}(\rho)(W)\cdot\Delta_jf_k(X)\cdot\Delta_jf_i(X^A)\in\mathbbm{R}.$$
Applying Lemma \ref{lemma1} to $g=\frac{\partial\rho}{\partial x_i}\circ f$ and $h=f_i$ and recalling that $\mathbbm{E}W=0$, we obtain:
\begin{align}
\mathbbm{E}\left<\nabla\rho(W),W\right>
=&\frac{1}{2}\sum_{A\subsetneq [n]}\frac{1}{{n\choose |A|}(n-|A|)}\sum_{j\not\in A}\mathbbm{E}[R_{A,j}].\label{first_part}
\end{align}
Now, from the definition of $T$, we have
\begin{align}
\left<\text{Hess}(\rho(W)),T\right>_{H.S.}
&=\frac{1}{2}\sum_{A\subsetneq [n]}\frac{1}{{n\choose |A|}(n-|A|)}\sum_{j\not\in A}\tilde{R}_{A,j}.\label{second_part}
\end{align}
Moreover, using Taylor's theorem and H\"older's inequality,
\begin{align}
\mathbbm{E}\left|R_{A,j}-\tilde{R}_{A,j}\right|
=&\mathbbm{E}\left|D\rho(f(X))[\Delta_jf(X^A)]-D\rho(f(X^j))[\Delta_jf(X^A)]-D^2\rho(f(X))[\Delta_jf(X^A),\Delta_jf(X)]\right|\nonumber\\
\leq &\frac{M_3(\rho)}{2}\mathbbm{E}\left[\left\|\Delta_jf(X)\right\|^2\left\|\Delta_jf(X^A)\right\|\right]\leq \frac{M_3(\rho)}{2}\mathbbm{E}\left\|\Delta_jf(X)\right\|^3.\label{third_part}
\end{align}
Combining (\ref{first_part}), (\ref{second_part}) and (\ref{third_part}), we get
\begin{align}
\left|\mathbbm{E}\left<\nabla\rho(W),W\right>-\left<\text{Hess}(\rho)(W),T\right>_{H.S.}\right|
\leq& \frac{M_3(\rho)}{4}\sum_{A\subsetneq [n]}\frac{1}{{n\choose |A|}(n-|A|)}\sum_{j\not\in A}\mathbbm{E}\left\|\Delta_jf(X)\right\|^3=\frac{M_3(\rho)}{4}\sum_{j=1}^n\mathbbm{E}\left\|\Delta_jf(X)\right\|^3.\label{lemma2_1}
\end{align}
Moreover,
\begin{align}
\left|\mathbbm{E}\left<\text{Hess}\left(\rho(W)\right),T-\Sigma\right>_{H.S.}\right|=&\left|\mathbbm{E}\left<\text{Hess}\left(\rho(W)\right),\mathbbm{E}\left[T-\Sigma\,|W\right]\right>_{H.S.}\right|
\leq \tilde{M}_2(\rho)\mathbbm{E}\left\|\mathbbm{E}[T|W]-\Sigma\right\|_{H.S.}\label{lemma2_2}
\end{align}
and the result follows by (\ref{lemma2_1}) and (\ref{lemma2_2}).
\end{proof}
\subsubsection{Concluding argument in the proof of Theorem \ref{main}}
We use the well-known formula for the derivative of a convolution and Young's convolution inequality in a manner similar to that of the proof of \cite[Theorem 3]{meckes09}. Since $\|\varphi_{\epsilon^2I}\|_{L^1}=1$, we obtain that $M_k(F_\epsilon)\leq M_k(F)$, for all $k\geq 1$, and $\tilde{M}_2(F_\epsilon)\leq \tilde{M}_2(F)$. The bounds (\ref{non_neg_def_bound}), (\ref{pos_def_bound}) now follow by using (\ref{stein_epsilon})-(\ref{der_bounds}) and Lemma \ref{lemma2} and taking $\epsilon\to 0$.\qed

\subsection{Proof of Theorem \ref{main3}}
\subsubsection{Introduction}
Let $\varphi_{\Sigma}$ denote the density of $N_{\Sigma}$. Let $K$ be a measurable convex subset of $\mathbbm{R}^d$ and let $h=\mathbbm{1}_K$. Fix $t\in(0,1)$. We introduce the following smoothed version of $h$:
$$h_{t,\Sigma}(y):=\int_{\mathbbm{R}^d} h(\sqrt{t}z + \sqrt{1-t} y)\varphi_{\Sigma}(z)dz=\mathbbm{E}h(\sqrt{t}N_{\Sigma}+\sqrt{1-t}y),\quad y\in\mathbbm{R}^d.$$
It follows from \cite[Lemma 2.2]{schulte_yukich} that
\begin{align}\label{schulte_bound}
d_{convex}(W,N_{\Sigma})\leq\frac{4}{3}\sup_{h\in\mathcal{I}_d}\left|\mathbbm{E}h_{t,\Sigma}(W)-\mathbbm{E}h_{t,\Sigma}(N_{\Sigma})\right|+\frac{20}{\sqrt{2}}d\frac{\sqrt{t}}{1-t}.
\end{align}
In order to upper-bound (\ref{schulte_bound}), we will study the following function $f_{t,h,\Sigma}:\mathbbm{R}^d\to\mathbbm{R}$:
\begin{align}\label{f_def}
f_{t,h,\Sigma}(y):=\frac{1}{2}\int_t^1\frac{1}{1-s}\int_{\mathbbm{R}^d}\left(h\left(\sqrt{s} z + \sqrt{1-s} y\right)-h(z)\right)\varphi_{\Sigma}(z)dz ds,
\end{align}
which, as noted in \cite[page 12]{schulte_yukich} and proved in \cite[Lemma 1]{meckes09} and \cite[Lemma 3.3]{nourdin_peccati_reveillac}, satisfies:
\begin{align}\label{stein_equation}
h_{t,\Sigma}(y)-\mathbbm{E}h_{t,\Sigma}(N_{\Sigma})=\left<\nabla f_{t,h,\Sigma}(y),y\right>-\left<\text{Hess}(f_{t,h,\Sigma})(y),\Sigma\right>_{H.S.},
\end{align}
for all $y\in\mathbbm{R}^d$, i.e. solves the corresponding Stein equation for $N_{\Sigma}$.

Changing slightly the argument in the proof of Lemma \ref{lemma2}, we note that
\begin{align*}
\left|\mathbbm{E}\left[\left<\text{Hess}(f_{t,h,\Sigma})(W),T-\Sigma\right>_{H.S}\right]\right|=&\left|\mathbbm{E}\left[\left<\text{Hess}(f_{t,h,\Sigma})(W),\mathbbm{E}[T-\Sigma |W]\right>_{H.S}\right]\right|\\
\leq&\sqrt{\mathbbm{E}\left\|\text{Hess}(f_{t,h,\Sigma})(W)\right\|_{H.S.}^2}\sqrt{\mathbbm{E}\left\|\mathbbm{E}[T-\Sigma |W]\right\|_{H.S.}^2}.
\end{align*}
By \cite[Proposition 2.3]{schulte_yukich},
$$\sup_{h\in\mathcal{I}_d} \mathbbm{E}\left\|\text{Hess}(f_{t,h,\Sigma})(W)\right\|_{H.S.}^2\leq \|\Sigma^{-1}\|_{op}^2\left(d^2(\log t)^2 d_{convex}(W,N_{\Sigma})+530 d^{17/6}\right),$$
for all $t\in(0,1)$  and therefore:
\begin{align}\label{j1}
\left|\mathbbm{E}\left<\text{Hess}(f_{t,h,\Sigma})(W),T-\Sigma\right>_{H.S}\right|\leq \|\Sigma^{-1}\|_{op}d\sqrt{(\log t)^2 d_{convex}(W,N_{\Sigma})+530 d^{5/6}}\sqrt{\mathbbm{E}\left\|\mathbbm{E}[T-\Sigma |W]\right\|_{H.S.}^2}.
\end{align}
An upper bound on (\ref{schulte_bound}) will be obtained by using (\ref{stein_equation})  and combining (\ref{j1}) with an upper-bound on\\ $\left|\mathbbm{E}\left<\nabla f_{t,h,\Sigma}(W),W\right>-\left<\text{Hess}(f_{t,h,\Sigma})(W),T\right>_{H.S.}\right|$, which we will work out in the subsequent steps of the proof.
To this end, we shall follow an argument similar to that of the proof of \cite[Theorem 1.2]{schulte_yukich}. 
\subsubsection{Step 1}
As before, in the proof of Lemma \ref{lemma2}, for any $A\in[n]$ such that $j\not\in A$, let
$$R_{A,j}:=\sum_{i=1}^d\Delta_j\left(\frac{\partial f_{t,h,\Sigma}}{\partial x_i}\circ f\right)(X)\cdot \Delta_jf_i(X^A)\in\mathbbm{R};\quad\tilde{R}_{A,j}:=\sum_{i,k=1}^d\text{Hess}_{i,k}(f_{t,h,\Sigma})(W)\cdot\Delta_jf_k(X)\cdot \Delta_jf_i(X^A)\in\mathbbm{R}.$$
For $A\subsetneq [n]$, let $k_{n,A}=\frac{1}{{n\choose|A|}(n-|A|)}$. It follows that
\begin{align}
\left|\mathbbm{E}\left<\nabla f_{t,h,\Sigma}(W),W\right>-\left<\text{Hess}(f_{t,h,\Sigma}),T\right>_{H.S.}\right|=\frac{1}{2}\left|\sum_{A\subsetneq[n]}k_{n,A}\sum_{j\not\in A}\mathbbm{E}\left[R_{A,j}-\tilde{R}_{A,j}\right]\right|.\label{remaining_bound}
\end{align}
Now, using repeatedly Taylor's theorem with integral remainder (in identities $(\ast)$ below), we obtain, for $j\not\in A$, that
\begin{align}
&\mathbbm{E}\left[\tilde{R}_{A,j}-R_{A,j}\right]\notag\\
=&\mathbbm{E}\left[\left(Df_{t,h,\Sigma}\left(f(X^j)\right)-Df_{t,h,\Sigma}\left(f(X)\right)\right)[\Delta_jf(X^A)]-D^2f_{t,h,\Sigma}\left(f(X)\right)[\Delta_jf(X^A),-\Delta_jf(X)]\right]\notag\\
\stackrel{(\ast)}=&\mathbbm{E}\int_0^1D^2f_{t,h,\Sigma}\left(f(X)-u\Delta_jf(X)\right)[\Delta_jf(X^A),-\Delta_jf(X)]-D^2f_{t,h,\Sigma}\left(f(X)\right)[\Delta_jf(X^A),-\Delta_jf(X)]]du\notag\\
=&\mathbbm{E}\int_0^1\left(D^2f_{t,h,\Sigma}\left(f(X)-u\Delta_jf(X)\right)-D^2f_{t,h,\Sigma}\left(f(X)\right)\right)[\Delta_jf(X^A),-\Delta_jf(X)]du\notag\\
\stackrel{(\ast)}=&\mathbbm{E}\int_0^1\int_0^1D^3f_{t,h,\Sigma}\left(f(X)-uv\Delta_jf(X)\right)[\Delta_jf(X^A),\Delta_jf(X),u\Delta_jf(X)]dvdu\notag\\
= &\mathbbm{E}\int_0^1\int_0^1D^3f_{t,h,\Sigma}\left(f(X)-v\Delta_jf(X)\right)[\Delta_jf(X^A),\Delta_jf(X),u\Delta_jf(X)]dvdu\notag\\
&+\mathbbm{E}\int_0^1\int_0^1\left(D^3f_{t,h,\Sigma}\left(f(X)-uv\Delta_jf(X)\right)-D^3f_{t,h,\Sigma}\left(f(X)-v\Delta_jf(X)\right)\right)[\Delta_jf(X^A),\Delta_jf(X),u\Delta_jf(X)]dvdu\notag\\
\stackrel{(\ast)}=&\mathbbm{E}\left(D^2f_{t,h,\Sigma}\left(f(X)\right)-D^2f_{t,h,\Sigma}\left(f(X^j)\right)\right)[\Delta_jf(X^A),\Delta_jf(X)]\notag\\
&+\mathbbm{E}\int_0^1\int_0^1\left(D^3f_{t,h,\Sigma}\left(f(X)-uv\Delta_jf(X)\right)-D^3f_{t,h,\Sigma}\left(f(X)-v\Delta_jf(X)\right)\right)[\Delta_jf(X^A),\Delta_jf(X),u\Delta_jf(X)]dvdu\notag\\
=&\mathbbm{E}\int_0^1\int_0^1\left(D^3f_{t,h,\Sigma}\left(f(X)-uv\Delta_jf(X)\right)-D^3f_{t,h,\Sigma}\left(f(X)-v\Delta_jf(X)\right)\right)[\Delta_jf(X^A),\Delta_jf(X),u\Delta_jf(X)]dvdu\notag\\
=:&J(j,A),\label{def_j}
\end{align}
The second to last identity holds because, for $j\not\in A$, $(X,X^j,X^A,X^{A\cup\{j\}})\stackrel{\mathcal{D}}=(X^j,X,X^{A\cup\{j\}},X^A)$ and so
\begin{align*}
&\mathbbm{E}\left(D^2f_{t,h,\Sigma}\left(f(X)\right)-D^2f_{t,h,\Sigma}\left(f(X^j)\right)\right)[\Delta_jf(X^A),\Delta_jf(X)]\\
=&\mathbbm{E} D^2f_{t,h,\Sigma}\left(f(X)\right)\left[f(X^A)-f(X^{A\cup\{j\}}),f(X)-f(X^j)\right]-\mathbbm{E} D^2f_{t,h,\Sigma}\left(f(X^j)\right)\left[f(X^{A\cup\{j\}})-f(X^A),f(X^j)-f(X)\right]\\
=&0.
\end{align*}

Now, using (\ref{f_def}) and \cite[(2.2)]{schulte_yukich}, we obtain that,  for $J(j,A)$, defined in (\ref{def_j}),
\begin{align*}
J(j,A)
= &-\frac{1}{2}\sum_{k,l,m=1}^d\mathbbm{E}\int_0^1\int_0^1\int_t^1\frac{\sqrt{1-s}}{s^{3/2}}\int_{\mathbbm{R}^d}\bigg[h\left(\sqrt{s}z+\sqrt{1-s}\left(f(X)-uv\Delta_jf(X)\right)\right)\\
&-h\left(\sqrt{s}z+\sqrt{1-s}\left(f(X)-v\Delta_jf(X)\right)\right)\bigg]\frac{\partial^3\varphi_{\Sigma}}{\partial y_k\partial y_l\partial y_m}(z)u\left(\Delta_jf(X^A)\right)_k\left(\Delta_jf(X)\right)_l\left(\Delta_jf(X)\right)_mdz\,ds\,dv\,du.
\end{align*}
Using the abbreviation
\begin{align*}
U_{k,l,m}:=&\underset{s,u\in[0,1]}{\sup_{z\in\mathbbm{R}^d}}\mathbbm{E}\Bigg\{\sum_{A\subsetneq [n]}k_{n,A}\sum_{j\not\in A}\int_0^1\Bigg|\bigg[h\left(\sqrt{s}z+\sqrt{1-s}\left(f(X)-uv\Delta_jf(X)\right)\right)\\
&-h\left(\sqrt{s}z+\sqrt{1-s}\left(f(X)-v\Delta_jf(X)\right)\right)\bigg]\left(\Delta_jf(X^A)\right)_k\left(\Delta_jf(X)\right)_l\left(\Delta_jf(X)\right)_m\vphantom{\frac{1}{{n\choose a}}}\Bigg|dv\Bigg\},
\end{align*}
for $k,l,m\in[d]$, and the Cauchy-Schwarz inequality, we obtain
\begin{align}
\frac{1}{2}\left|\sum_{A\subsetneq [n]}k_{n,A}\sum_{j\not\in A}J(j,A)\right|\leq& \frac{1}{2\sqrt{t}}\sum_{k,l,m=1}^d\int_{\mathbbm{R}^d}\left|\frac{\partial^3\varphi_{\Sigma}}{\partial y_k\partial y_l\partial y_m}(z)\right|dz U_{k,l,m}\notag\\
\leq &\frac{1}{2\sqrt{t}}\int_{\mathbbm{R}^d}\left(\sum_{k,l,m=1}^d \left(\frac{\partial^3\varphi_{\Sigma}}{\partial y_k\partial y_l\partial y_m}(z)\right)^2\right)^{1/2}dz\left(\sum_{k,l,m=1}^d \left(U_{k,l,m}\right)^2\right)^{1/2}.\label{j_bound}
\end{align}
By \cite[page 22 and (3.9)]{schulte_yukich},
\begin{align}\label{gauss_der_bound}
\int_{\mathbbm{R}^d}\left(\sum_{k,l,m=1}^d \left(\frac{\partial^3\varphi_{\Sigma}}{\partial y_k\partial y_l\partial y_m}(z)\right)^2\right)^{1/2}dz\leq \sqrt{6}d^{3/2}\left\|\Sigma^{-1}\right\|^{3/2}_{op}.\end{align}
Therefore, by (\ref{j_bound}),  in order to upper-bound (\ref{remaining_bound}), it remains to upper-bound $U_{k,l,m}$, for $k,l,m\in[d]$.
\subsubsection{Step 2 - an upper bound on $U_{k,l,m}$}
For $j$ such that $\Delta_jf(X)\neq 0$, we define $r\left(\Delta_jf(X)\right):=\frac{1}{\|\Delta_jf(X)\|}\Delta_jf(X)$. We also let $w=v\|\Delta_jf(X)\|$. We obtain
\begin{align}
U_{k,l,m}
\leq& \underset{s,u\in[0,1]}{\sup_{z\in\mathbbm{R}^d}}\mathbbm{E}\Bigg\{\sum_{A\subsetneq [n]}k_{n,A}\sum_{j\not\in A}\int_0^{\|\Delta_jf(X)\|}\Bigg|h\left(\sqrt{s}z+\sqrt{1-s}\left(f(X)-uwr\left(\Delta_jf(X)\right)\right)\right)\notag\\
&-h\left(\sqrt{s}z+\sqrt{1-s}\left(f(X)-wr\left(\Delta_jf(X)\right)\right)\right)\Bigg|dw\cdot\mathbbm{1}_{\left[\|\Delta_jf(X)\|\leq 1\right]} \Bigg|\left(\Delta_jf(X^A)\right)_k\frac{\left(\Delta_jf(X)\right)_l}{\|\Delta_jf(X)\|}\left(\Delta_jf(X)\right)_m\vphantom{\sum_{A}\frac{{n\choose k}}{{n\choose k}}}\Bigg|\Bigg\}\notag\\
&+\underset{s,u\in[0,1]}{\sup_{z\in\mathbbm{R}^d}}\mathbbm{E}\Bigg\{\sum_{A\subsetneq [n]}k_{n,A}\sum_{j\not\in A}\int_0^{1}\Bigg|\bigg[h\left(\sqrt{s}z+\sqrt{1-s}\left(f(X)-uv\Delta_jf(X)\right)\right)\notag\\
&-h\left(\sqrt{s}z+\sqrt{1-s}\left(f(X)-v\Delta_jf(X)\right)\right)\bigg]\mathbbm{1}_{\left[\|\Delta_jf(X)\|> 1\right]} \left(\Delta_jf(X^A)\right)_k\left(\Delta_jf(X)\right)_l\left(\Delta_jf(X)\right)_m\vphantom{\sum_{A}\frac{{n\choose k}}{{n\choose k}}}\Bigg|dv\Bigg\}\notag\\
&=: U_{k,l,m}^{(1)}+U_{k,l,m}^{(2)}.\label{u_bound}
\end{align}
Recall that $h(\cdot)=\mathbbm{1}_{[\cdot\in K]}$ for a measurable convex set $K\subseteq \mathbbm{R}^d$. It follows that
\begin{align*}
U_{k,l,m}^{(2)}\leq &\sum_{A\subsetneq [n]}k_{n,A}\sum_{j\not\in A}\mathbbm{E}\bigg\{\mathbbm{1}_{\left[\|\Delta_jf(X)\|> 1\right]}\left|\left(\Delta_jf(X^A)\right)_k\left(\Delta_jf(X)\right)_l\left(\Delta_jf(X)\right)_m\right|\bigg\}\\
\leq &\sum_{A\subsetneq [n]}k_{n,A}\mathbbm{E}\bigg\{\|\Delta_jf(X)\|\left|\left(\Delta_jf(X^A)\right)_k\left(\Delta_jf(X)\right)_l\left(\Delta_jf(X)\right)_m\right|\bigg\}
\end{align*}
and so
\begin{align}
\sqrt{\sum_{k,l,m=1}^d \left(U_{k,l,m}^{(2)}\right)^2}
\leq &\sum_{k,l,m=1}^d\left|U_{k,l,m}^{(2)}\right|\notag\\
\leq &\sum_{A\subsetneq [n]}k_{n,A}\sum_{j\not\in A}\sum_{k,l,m=1}^d\mathbbm{E}\left[\|\Delta_jf(X)\|\left|\left(\Delta_jf(X^A)\right)_k\left(\Delta_jf(X)\right)_l\left(\Delta_jf(X)\right)_m\right|\right]\notag\\
\leq&d^{3/2}\sum_{A\subsetneq [n]}k_{n,A}\sum_{j\not\in A}\mathbbm{E}\left[\left\|\Delta_jf(X)\right\|^3\left\|\Delta_jf(X^A)\right\|\right]\notag\\
\leq &d^{3/2}\sum_{j=1}^n\mathbbm{E}\|\Delta_jf(X)\|^4=d^{3/2}\gamma_2^2.\label{u2_bound}
\end{align}
Now, we need to bound $\sum_{k,l,m=1}^d\left(U_{k,l,m}^{(1)}\right)^2$. We define $K_{s,z}:=\frac{1}{\sqrt{1-s}}\left(K-\sqrt{s}z\right)$. Then, for $v,u\in[0,1]$ and $r$ and $w$ defined at the top of this subsection,
\begin{align*}
&\left|h\left(\sqrt{s}z+\sqrt{1-s}\left(f(X)-uwr\left(\Delta_jf(X)\right)\right)\right)-h\left(\sqrt{s}z+\sqrt{1-s}\left(f(X)-wr\left(\Delta_jf(X)\right)\right)\right)\right|\\
=&\left|\mathbbm{1}_{\left[f(X)-uwr\left(\Delta_jf(X)\right)\in K_{s,z}\right]}-\mathbbm{1}_{\left[f(X)-wr\left(\Delta_jf(X)\right)\in K_{s,z}\right]}\right|\\
\leq &\mathbbm{1}_{\left[\text{dist}(f(X),\partial K_{s,z})\leq w\right]}.
\end{align*}
where $\partial K_{s,z}$ is the boundary of $K_{s,z}$ and $\text{dist}(f(X),\partial K_{s,z}):=\underset{{y\in\partial K_{s,z}}}{\inf}\|y-f(X)\|$. Thus, we have that
\begin{align}
U_{k,l,m}^{(1)}
\leq& \sup_{z\in\mathbbm{R}^d,s\in[0,1]}\mathbbm{E}\Bigg\{\sum_{A\subsetneq [n]}k_{n,A}\sum_{j\not\in A}\int_0^1\mathbbm{1}_{\left[\text{dist}(f(X),\partial K_{s,z})\leq w\right]}\mathbbm{1}_{\left[w\leq \|\Delta_jf(X)\|\right]}\left|\left(\Delta_jf(X^A)\right)_l\left(\Delta_jf(X)\right)_m\right|dw\Bigg\}\notag\\
= &\underset{s\in[0,1]}{\sup_{z\in\mathbbm{R}^d}}\Bigg\{\sum_{A\subsetneq [n]}k_{n,A}\sum_{j\not\in A}\int_0^1\mathbbm{P}\left[\text{dist}(f(X),\partial K_{s,z})\leq w\right]\mathbbm{E}\left[\mathbbm{1}_{\left[w\leq \|\Delta_jf(X)\|\right]}\left|\left(\Delta_jf(X^A)\right)_l\left(\Delta_jf(X)\right)_m\right|\right]dw\Bigg\}\notag\\
&+ \sup_{z\in\mathbbm{R}^d,s\in[0,1]}\int_0^1\mathbbm{E}\Bigg\{\mathbbm{1}_{\left[\text{dist}(f(X),\partial K_{s,z})\leq w\right]}\sum_{A\subsetneq [n]}k_{n,A}\sum_{j\not\in A}\bigg(\mathbbm{1}_{\left[w\leq \|\Delta_jf(X)\|\right]}\left|\left(\Delta_jf(X^A)\right)_l\left(\Delta_jf(X)\right)_m\right|\notag\\
&\hspace{8cm}-\mathbbm{E}\left[\mathbbm{1}_{\left[w\leq \|\Delta_jf(X)\|\right]}\left|\left(\Delta_jf(X^A)\right)_l\left(\Delta_jf(X)\right)_m\right|\right]\bigg)\Bigg\} dw\notag\\
=:&R_{l,m}^{(1)}+R_{l,m}^{(2)}.\label{u1_bound}
\end{align}
\subsubsection{Step 3 - upper bounds on $R_{l,m}^{(1)}$ and $R_{l,m}^{(2)}$}
Now, we apply \cite[Lemma 3.3]{schulte_yukich} to obtain
\begin{align}
&\hspace{-2cm}\sqrt{d\sum_{l,m=1}^d\left(R_{l,m}^{(1)}\right)^2}
\leq\sqrt{d}\sum_{l,m=1}^d\left|R_{l,m}^{(1)}\right|\notag\\
\leq& 2d\left\|\Sigma^{-{1/2}}\right\|_{op}\sum_{l,m=1}^d\sum_{A\subsetneq [n]}k_{n,A}\sum_{j\not\in A}\mathbbm{E}\left\lbrace\int_0^1w\mathbbm{1}_{\left[w\leq \|\Delta_jf(X)\|\right]}dw\left|\left(\Delta_jf(X^A)\right)_l\left(\Delta_jf(X)\right)_m\right|\right\rbrace\notag\\
&+2d_{convex}(W,N_{\Sigma})\sqrt{d}\sum_{l,m=1}^d\sum_{A\subsetneq [n]}k_{n,A}\sum_{j\not\in A}\mathbbm{E}\left\lbrace\int_0^1\mathbbm{1}_{\left[w\leq \|\Delta_jf(X)\|\right]}dw\left|\left(\Delta_jf(X^A)\right)_l\left(\Delta_jf(X)\right)_m\right|\right\rbrace\notag\\
\leq&d\left\|\Sigma^{-1/2}\right\|_{op}\sum_{l,m=1}^d\sum_{A\subsetneq [n]}k_{n,A}\sum_{j\not\in A}\mathbbm{E}\left\lbrace\|\Delta_jf(X)\|^2\left|\left(\Delta_jf(X^A)\right)_l\left(\Delta_jf(X)\right)_m\right|\right\rbrace\notag\\
&+2d_{convex}(W,N_{\Sigma})\sqrt{d}\sum_{l,m=1}^d\sum_{A\subsetneq [n]}k_{n,A}\sum_{j\not\in A}\mathbbm{E}\left\lbrace\|\Delta_jf(X)\|\,\left|\left(\Delta_jf(X^A)\right)_l\left(\Delta_jf(X)\right)_m\right|\right\rbrace\notag\\
\leq&d^2\|\Sigma^{-1/2}\|_{op}\sum_{A\subsetneq [n]}k_{n,A}\sum_{j\not\in A}\mathbbm{E}\left[\|\Delta_jf(X^A)\|\|\Delta_jf(X)\|^3\right]\notag\\
&+\sum_{A\subsetneq [n]}k_{n,A}\sum_{j\not\in A}2d_{convex}(W,N_{\Sigma})d^{3/2}\mathbbm{E}\left[\|\Delta_jf(X^A)\|\|\Delta_jf(X)\|^2\right]\notag\\
\leq &d^2\|\Sigma^{-1/2}\|_{op}\sum_{j=1}^n\mathbbm{E}\|\Delta_jf(X)\|^4+2d_{convex}(W,N_{\Sigma})d^{3/2}\sum_{j=1}^n\mathbbm{E}\|\Delta_jf(X)\|^3\notag\\
=&d^2\|\Sigma^{-1/2}\|_{op}\gamma_2^2+2d_{convex}(W,N_{\Sigma})d^{3/2}\gamma_1.\label{r1_bound}
\end{align}

For $R^{(2)}_{k,l}$ we apply the Cauchy-Schwarz inequality and \cite[Lemma 3.3]{schulte_yukich} to obtain
\begin{align}
R_{l,m}^{(2)}
\leq &\sup_{z\in\mathbbm{R}^d,s\in[0,1]}\int_0^1\mathbbm{P}\left(\text{dist}(f(X),\partial K_{s,z})\leq w\right)^{1/2}\notag\\
&\cdot\Bigg(\mathbbm{E}\Bigg|\sum_{A\subsetneq [n]}k_{n,A}\sum_{j\not\in A} \mathbbm{E}\left[\mathbbm{1}_{\left[w\leq \|\Delta_jf(X)\|\right]}\left|\left(\Delta_jf(X^A)\right)_l\left(\Delta_jf(X)\right)_m\right|\,\bigg|\,X\right]\notag\\
&\hphantom{...............}-\sum_{A\subsetneq [n]}k_{n,A}\sum_{j\not\in A}\mathbbm{E}\left\lbrace\mathbbm{1}_{\left[w\leq \|\Delta_jf(X)\|\right]}\left|\left(\Delta_jf(X^A)\right)_l\left(\Delta_jf(X)\right)_m\right|\right\rbrace\Bigg|^2\Bigg)^{1/2}dw\notag\\
\leq&\int_0^1\left(2d_{convex}(W,N_{\Sigma})+2\sqrt{d}\|\Sigma^{-1/2}\|_{op}w\right)^{1/2}\notag\\
&\cdot\left(\text{Var}\left\lbrace\sum_{A\subsetneq [n]}k_{n,A}\sum_{j\not\in A}\mathbbm{E}\left[\mathbbm{1}_{\left[w\leq \|\Delta_jf(X)\|\right]}\left|\left(\Delta_jf(X^A)\right)_l\left(\Delta_jf(X)\right)_m\right|\,\bigg|\,X\right]\right\rbrace\right)^{1/2}dw\notag\\
\leq &\left(2d_{convex}(W,N_{\Sigma})V_{l,m}^{(1)}+2\sqrt{d}\left\|\Sigma^{-1/2}\right\|_{op}V_{l,m}^{(2)}\right)^{1/2},\label{r2_bound}
\end{align}
by the Cauchy-Schwarz ineuqality, where
\begin{align}
&V_{l,m}^{(1)}:=\int_0^1\text{Var}\left\lbrace\sum_{A\subsetneq [n]}k_{n,A}\sum_{j\not\in A}\mathbbm{E}\left[\mathbbm{1}_{\left[w\leq \|\Delta_jf(X)\|\right]}\left|\left(\Delta_jf(X^A)\right)_l\left(\Delta_jf(X)\right)_m\right|\,\bigg|\,X\right]\right\rbrace dw\notag\\
&V_{l,m}^{(2)}:=\int_0^1w\text{Var}\left\lbrace\sum_{A\subsetneq [n]}k_{n,A}\sum_{j\not\in A}\mathbbm{E}\left[\mathbbm{1}_{\left[w\leq \|\Delta_jf(X)\|\right]}\left|\left(\Delta_jf(X^A)\right)_l\left(\Delta_jf(X)\right)_m\right|\,\bigg|\,X\right]\right\rbrace dw.\label{v_def}
\end{align}
The existence of the variances in the definition of $V_{l,m}^{(1)}$ and $V_{l,m}^{(2)}$ is proved below, in (\ref{variance_exists}).
\subsubsection{Step 4 - upper bounds on $V_{l,m}^{(1)}$ and $V_{l,m}^{(2)}$}
We first compute the difference operators. 
Note that, for all $i,j\in[n]$,
\begin{align*}
\left|\tilde{\Delta}_i\mathbbm{1}_{\left[w\leq\|\Delta_jf(X)\|\right]}\right|=&\left|\mathbbm{1}_{\left[w\leq\|\Delta_jf(X)\|-\tilde{\Delta}_i\|\Delta_jf(X)\|\right]}-\mathbbm{1}_{\left[w\leq\|\Delta_jf(X)\|\right]}\right|\leq \mathbbm{1}_{\left[w\leq\|\Delta_jf(X)\|+\left|\,\tilde{\Delta}_i\|\Delta_jf(X)\|\,\right|\right]}\quad\text{and}\\
\left|\,\tilde{\Delta}_i\|\Delta_jf(X)\|\,\right|=&\left|\,\|\Delta_jf(X)-\tilde{\Delta}_i\Delta_jf(X)\|-\|\Delta_jf(X)\|\,\right|\leq \|\tilde{\Delta}_i\Delta_jf(X)\|.
\end{align*}
Hence
\begin{align*}
\left|\tilde{\Delta}_i\mathbbm{1}_{\left[w\leq\|\Delta_jf(X)\|\right]}\right|\leq \mathbbm{1}_{\left[w\leq\|\Delta_jf(X)\|+\|\tilde{\Delta}_i\Delta_jf(X)\|\right]}\mathbbm{1}_{\left[\tilde{\Delta}_i\Delta_jf(X)\neq 0\right]}.
\end{align*}
Therefore, for all $i\in[n]$,
\begin{align}
&\left(\sum_{A\subsetneq [n]}k_{n,A}\sum_{j\not\in A}\tilde{\Delta}_i\left[\mathbbm{1}_{\left[w\leq \|\Delta_jf(X)\|\right]}\left|\left(\Delta_jf(X^A)\right)_l\left(\Delta_jf(X)\right)_m\right|\right]\right)^2\notag\\
\leq &\Bigg\{\sum_{A\subsetneq [n]}k_{n,A}\sum_{j\not\in A}\bigg[\mathbbm{1}_{\left[w\leq\|\Delta_jf(X)\|+\|\tilde{\Delta}_i\Delta_jf(X)\|\right]}\mathbbm{1}_{\left[\tilde{\Delta}_i\Delta_jf(X)\neq 0\right]}\Big(\left|\left(\Delta_jf(X^A)\right)_l\left(\Delta_jf(X)\right)_m\right|\notag\\
&\hspace{4cm}+\left|\tilde{\Delta}_i\left|\left(\Delta_jf(X^A)\right)_l\left(\Delta_jf(X)\right)_m\right|\right|\Big)+\mathbbm{1}_{\left[w\leq\|\Delta_jf(X)\|\right]}\left|\tilde{\Delta}_i\left|\left(\Delta_jf(X^A)\right)_l\left(\Delta_jf(X)\right)_m\right|\right|\bigg]\Bigg\}^2\notag\\
\leq &3\sum_{A_1,A_2\subsetneq [n]}\hspace{-2mm}k_{n,A_1}k_{n,A_2}\hspace{-1mm}\sum_{j_1\not\in A_1}\sum_{j_2\not\in A_2}\Bigg[\mathbbm{1}_{\left[\tilde{\Delta}_i\Delta_{j_1}f(X)\neq 0,\tilde{\Delta}_i\Delta_{j_2}f(X)\neq 0\right]}\mathbbm{1}_{\left[w\leq\min\left\lbrace\|\Delta_{j_1}f(X)\|+\|\tilde{\Delta}_i\Delta_{j_1}f(X)\|,\|\Delta_{j_2}f(X)\|+\|\tilde{\Delta}_i\Delta_{j_2}f(X)\|\right\rbrace\right]}\notag\\
&\hspace{5cm}\cdot\bigg(\left|\left(\Delta_{j_1}f(X^A)\right)_l\left(\Delta_{j_1}f(X)\right)_m\right|\,\left|\left(\Delta_{j_2}f(X^A)\right)_l\left(\Delta_{j_2}f(X)\right)_m\right| \notag\\
&\hspace{7cm}+      \left|\tilde{\Delta}_i\left|\left(\Delta_{j_1}f(X^A)\right)_l\left(\Delta_{j_1}f(X)\right)_m\right|\right|\,\left|\tilde{\Delta}_i\left|\left(\Delta_{j_2}f(X^A)\right)_l\left(\Delta_{j_2}f(X)\right)_m\right|\right|\bigg)\Bigg]\notag\\
&+3\sum_{A_1,A_2\subsetneq [n]}k_{n,A_1}k_{n,A_2}\sum_{j_1\not\in A_1}\sum_{j_2\not\in A_2}\bigg[\mathbbm{1}_{\left[w\leq\min\left[\|\Delta_{j_1}f(X)\|,\|\Delta_{j_2}f(X)\|\right]\right]}\vphantom{\sum_i^i}\notag\\
&\hspace{7cm}\cdot\left|\tilde{\Delta}_i\left|\left(\Delta_{j_1}f(X^A)\right)_l\left(\Delta_{j_1}f(X)\right)_m\right|\right|\,\left|\tilde{\Delta}_i\left|\left(\Delta_{j_2}f(X^A)\right)_l\left(\Delta_{j_2}f(X)\right)_m\right|\right|\bigg].\label{ineq_v}
\end{align}

We now use the fact that, for any $g:\mathcal{X}^{2n}\to\mathbbm{R}$ and  $U=g(X,X')$, $\text{Var}(\mathbbm{E}(U|X))\leq \mathbbm{E}\left(\text{Var}(U|X')\right)$ (see \cite[Lemma 4.4]{new_stein}). Together with the Efron-Stein inequality \cite{steele}, this implies that, for such $U$,
\begin{align}\label{general_fact}
\text{Var}(\mathbbm{E}(U|X))\leq \mathbbm{E}\left(\text{Var}(U|X')\right)\leq\mathbbm{E}\Bigg\{\frac{1}{2}\sum_{i=1}^{n}\mathbbm{E}\left[\left(\tilde{\Delta}_i U\right)^2\,\bigg|\,X'\right]\Bigg\}=\frac{1}{2}\sum_{i=1}^{n}\mathbbm{E}\left[\left(\tilde{\Delta}_i U\right)^2\right].
\end{align}
The general fact (\ref{general_fact}) and the inequality (\ref{ineq_v}) together lead to the following estimates:
\begin{align*}
V_{l,m}^{(1)}\leq &\frac{3}{2}\sum_{i=1}^n\sum_{A_1,A_2\subsetneq [n]}k_{n,A_1}k_{n,A_2}\sum_{j_1\not\in A_1}\sum_{j_2\not\in A_2}\mathbbm{E}\Bigg\{\mathbbm{1}_{\left[\tilde{\Delta}_i\Delta_{j_1}f(X)\neq 0,\tilde{\Delta}_i\Delta_{j_2}f(X)\neq 0\right]}\\
&\cdot\min\left[\|\Delta_{j_1}f(X)\|+\|\tilde{\Delta}_i\Delta_{j_1}f(X)\|,\|\Delta_{j_2}f(X)\|+\|\tilde{\Delta}_i\Delta_{j_2}f(X)\|\right]\bigg(\left|\left(\Delta_{j_1}f(X^A)\right)_l\left(\Delta_{j_1}f(X)\right)_m\right|\\
&\cdot\left|\left(\Delta_{j_2}f(X^A)\right)_l\left(\Delta_{j_2}f(X)\right)_m\right| +      \left|\tilde{\Delta}_i\left|\left(\Delta_{j_1}f(X^A)\right)_l\left(\Delta_{j_1}f(X)\right)_m\right|\right|\,\left|\tilde{\Delta}_i\left|\left(\Delta_{j_2}f(X^A)\right)_l\left(\Delta_{j_2}f(X)\right)_m\right|\right|\bigg)\Bigg\}\\
&+\frac{3}{2}\sum_{i=1}^n\sum_{A_1,A_2\subsetneq [n]}k_{n,A_1}k_{n,A_2}\sum_{j_1\not\in A_1}\sum_{j_2\not\in A_2}\mathbbm{E}\bigg\{\min\left[\|\Delta_{j_1}f(X)\|,\|\Delta_{j_2}f(X)\|\right]\\
&\hspace{5cm}\cdot\left|\tilde{\Delta}_i\left|\left(\Delta_{j_1}f(X^A)\right)_l\left(\Delta_{j_1}f(X)\right)_m\right|\right|\,\left|\tilde{\Delta}_i\left|\left(\Delta_{j_2}f(X^A)\right)_l\left(\Delta_{j_2}f(X)\right)_m\right|\right|\bigg\};\\
V_{l,m}^{(2)}\leq&\frac{3}{2}\sum_{i=1}^n\sum_{A_1,A_2\subsetneq [n]}k_{n,A_1}k_{n,A_2}\sum_{j_1\not\in A_1}\sum_{j_2\not\in A_2}\mathbbm{E}\Bigg\{\mathbbm{1}_{\left[\tilde{\Delta}_i\Delta_{j_1}f(X)\neq 0,\tilde{\Delta}_i\Delta_{j_2}f(X)\neq 0\right]}\\
&\cdot\min\left[\|\Delta_{j_1}f(X)\|^2+\|\tilde{\Delta}_i\Delta_{j_1}f(X)\|^2,\|\Delta_{j_2}f(X)\|^2+\|\tilde{\Delta}_i\Delta_{j_2}f(X)\|^2\right]\bigg(\left|\left(\Delta_{j_1}f(X^A)\right)_l\left(\Delta_{j_1}f(X)\right)_m\right|\\
&\cdot\left|\left(\Delta_{j_2}f(X^A)\right)_l\left(\Delta_{j_2}f(X)\right)_m\right| +      \left|\tilde{\Delta}_i\left|\left(\Delta_{j_1}f(X^A)\right)_l\left(\Delta_{j_1}f(X)\right)_m\right|\right|\,\left|\tilde{\Delta}_i\left|\left(\Delta_{j_2}f(X^A)\right)_l\left(\Delta_{j_2}f(X)\right)_m\right|\right|\bigg)\Bigg\}\\
&+\frac{3}{4}\sum_{i=1}^n\sum_{A_1,A_2\subsetneq [n]}k_{n,A_1}k_{n,A_2}\sum_{j_1\not\in A_1}\sum_{j_2\not\in A_2}\mathbbm{E}\bigg\{\min\left[\|\Delta_{j_1}f(X)\|^2,\|\Delta_{j_2}f(X)\|^2\right]\vphantom{\sum_i^i}\\
&\hspace{5cm} \cdot\left|\tilde{\Delta}_i\left|\left(\Delta_{j_1}f(X^A)\right)_l\left(\Delta_{j_1}f(X)\right)_m\right|\right|\,\left|\tilde{\Delta}_i\left|\left(\Delta_{j_2}f(X^A)\right)_l\left(\Delta_{j_2}f(X)\right)_m\right|\right|\bigg\}.
\end{align*}

This, in turn, implies that, for $p=1,2$
\begin{align}
V_{l,m}^{(p)}
\leq &\frac{3}{2}\sum_{i=1}^n\mathbbm{E}\left(\sum_{A\subsetneq [n]}k_{n,A}\sum_{j\not\in A}\mathbbm{1}_{[\tilde{\Delta}_i\Delta_jf(X)\neq 0]}\sqrt{\|\Delta_jf(X)\|^p+\|\tilde{\Delta}_i\Delta_jf(X)\|^p}\left|\left(\Delta_{j}f(X^A)\right)_l\left(\Delta_{j}f(X)\right)_m\right|\right)^2\notag\\
&+\frac{3}{2}\sum_{i=1}^n\mathbbm{E}\left(\sum_{A\subsetneq [n]}k_{n,A}\sum_{j\not\in A}\mathbbm{1}_{[\tilde{\Delta}_i\Delta_jf(X)\neq 0]}\sqrt{\|\Delta_jf(X)\|^p+\|\tilde{\Delta}_i\Delta_jf(X)\|^p}\left|\tilde{\Delta}_i\left|\left(\Delta_{j}f(X^A)\right)_l\left(\Delta_{j}f(X)\right)_m\right|\right|\right)^2\notag\\
&+\frac{3}{2p}\sum_{i=1}^n\mathbbm{E}\left(\sum_{A\subsetneq [n]}k_{n,A}\sum_{j\not\in A}\sqrt{\|\Delta_jf(X)\|^p}\left|\tilde{\Delta}_i\left|\left(\Delta_{j}f(X^A)\right)_l\left(\Delta_{j}f(X)\right)_m\right|\right|\right)^2\notag\\
\leq &\frac{3}{2}\sum_{i=1}^n\mathbbm{E}\left(\sum_{A\subsetneq [n]}k_{n,A}\sum_{j\not\in A}\mathbbm{1}_{[\tilde{\Delta}_i\Delta_jf(X)\neq 0]}\sqrt{\|\Delta_jf(X)\|^p+\|\tilde{\Delta}_i\Delta_jf(X)\|^p}\left|\left(\Delta_{j}f(X^A)\right)_l\left(\Delta_{j}f(X)\right)_m\right|\right)^2\notag\\
&+\left(\frac{3}{2}+\frac{3}{2p}\right)\sum_{i=1}^n\mathbbm{E}\left(\sum_{A\subsetneq [n]}k_{n,A}\sum_{j\not\in A}\sqrt{\|\Delta_jf(X)\|^p+\|\tilde{\Delta}_i\Delta_jf(X)\|^p}\left|\tilde{\Delta}_i\left|\left(\Delta_{j}f(X^A)\right)_l\left(\Delta_{j}f(X)\right)_m\right|\right|\right)^2.\label{ineq_v_2}
\end{align}
Now, note that
\begin{align}
\left|\tilde{\Delta}_i\left|\left(\Delta_{j}f(X^A)\right)_l\left(\Delta_{j}f(X)\right)_m\right|\right|
\leq &\left|\tilde{\Delta}_i\left(\left(\Delta_{j}f(X^A)\right)_l\left(\Delta_{j}f(X)\right)_m\right)\right|\notag\\
&\hspace{-4cm}=\left|\tilde{\Delta}_i\left(\Delta_{j}f(X^A)\right)_l\left(\Delta_{j}f(X)\right)_m+\left(\Delta_{j}f(X^A)\right)_l\tilde{\Delta}_i\left(\Delta_{j}f(X)\right)_m-\tilde{\Delta}_i\left(\Delta_{j}f(X^A)\right)_l\tilde{\Delta}_i\left(\Delta_{j}f(X)\right)_m\right|\label{tilde_bound}
\end{align}
and therefore, using (\ref{ineq_v_2}), for $p=1,2$,
\begin{align*}
V_{l,m}^{(p)}
\leq &\frac{3}{2}\sum_{i=1}^n\mathbbm{E}\left[\left(\sum_{A\subsetneq [n]}k_{n,A}\sum_{j\not\in A}\mathbbm{1}_{[\tilde{\Delta}_i\Delta_jf(X)\neq 0]}\sqrt{\|\Delta_jf(X)\|^p+\|\tilde{\Delta}_i\Delta_jf(X)\|^p}\left|\left(\Delta_{j}f(X^A)\right)_l\left(\Delta_{j}f(X)\right)_m\right|\right)^2\right]\\
&+\left(\frac{9}{2}+\frac{9}{2p}\right)\sum_{i=1}^n\mathbbm{E}\left[\left(\sum_{A\subsetneq [n]}k_{n,A}\sum_{j\not\in A}\sqrt{\|\Delta_jf(X)\|^p+\|\tilde{\Delta}_i\Delta_jf(X)\|^p}\left|\tilde{\Delta}_i\left(\Delta_{j}f(X^A)\right)_l\left(\Delta_{j}f(X)\right)_m\right|\right)^2\right]\\
&+\left(\frac{9}{2}+\frac{9}{2p}\right)\sum_{i=1}^n\mathbbm{E}\left[\left(\sum_{A\subsetneq [n]}k_{n,A}\sum_{j\not\in A}\sqrt{\|\Delta_jf(X)\|^p+\|\tilde{\Delta}_i\Delta_jf(X)\|^p}\left|\left(\Delta_{j}f(X^A)\right)_l\tilde{\Delta}_i\left(\Delta_{j}f(X)\right)_m\right|\right)^2\right]\\
&+\left(\frac{9}{2}+\frac{9}{2p}\right)\sum_{i=1}^n\mathbbm{E}\left[\left(\sum_{A\subsetneq [n]}k_{n,A}\sum_{j\not\in A}\sqrt{\|\Delta_jf(X)\|^p+\|\tilde{\Delta}_i\Delta_jf(X)\|^p}\left|\tilde{\Delta}_i\left(\Delta_{j}f(X^A)\right)_l\tilde{\Delta}_i\left(\Delta_{j}f(X)\right)_m\right|\right)^2\right].
\end{align*}
Now
\begin{align*}
&\frac{3}{2}\sum_{l,m=1}^d\sum_{i=1}^n\mathbbm{E}\left[\left(\sum_{A\subsetneq [n]}k_{n,A}\sum_{j\not\in A}\mathbbm{1}_{[\tilde{\Delta}_i\Delta_jf(X)\neq 0]}\sqrt{\|\Delta_jf(X)\|^p+\|\tilde{\Delta}_i\Delta_jf(X)\|^p}\left|\left(\Delta_{j}f(X^A)\right)_l\left(\Delta_{j}f(X)\right)_m\right|\right)^2\right]\\
\leq& \frac{3}{2}\sum_{i=1}^n\mathbbm{E}\left[\left(\sum_{A\subsetneq [n]}k_{n,A}\sum_{j\not\in A}\mathbbm{1}_{[\tilde{\Delta}_i\Delta_jf(X)\neq 0]}\sqrt{\|\Delta_jf(X)\|^p+\|\tilde{\Delta}_i\Delta_jf(X)\|^p}\sum_{l,m=1}^d\left|\left(\Delta_{j}f(X^A)\right)_l\left(\Delta_{j}f(X)\right)_m\right|\right)^2\right]\\
\leq& \frac{3}{2}d^2\sum_{i=1}^n\mathbbm{E}\left[\left(\sum_{A\subsetneq [n]}k_{n,A}\sum_{j\not\in A}\mathbbm{1}_{[\tilde{\Delta}_i\Delta_jf(X)\neq 0]}\sqrt{\|\Delta_jf(X)\|^p+\|\tilde{\Delta}_i\Delta_jf(X)\|^p}\left\|\Delta_{j}f(X^A)\right\|\left\|\Delta_{j}f(X)\right\|\right)^2\right]
\end{align*}
and similar inequalities hold for the other terms. Therefore, for $p=1,2$,
\begin{align}
\sum_{l,m=1}^dV_{l,m}^{(p)}
\leq &\frac{3}{2}d^2\sum_{i=1}^n\mathbbm{E}\left[\left(\sum_{A\subsetneq [n]}k_{n,A}\sum_{j\not\in A}\mathbbm{1}[\tilde{\Delta}_i\Delta_jf(X)\neq 0]\sqrt{\|\Delta_jf(X)\|^p+\|\tilde{\Delta}_i\Delta_jf(X)\|^p}\left\|\Delta_{j}f(X^A)\right\|\left\|\Delta_{j}f(X)\right\|\right)^2\right]\notag\\
&+d^2\left(\frac{9}{2}+\frac{9}{2p}\right)\sum_{i=1}^n\mathbbm{E}\left[\left(\sum_{A\subsetneq [n]}k_{n,A}\sum_{j\not\in A}\sqrt{\|\Delta_jf(X)\|^p+\|\tilde{\Delta}_i\Delta_jf(X)\|^p}\left\|\tilde{\Delta}_i\Delta_{j}f(X^A)\right\|\left\|\Delta_{j}f(X)\right\|\right)^2\right]\notag\\
&+d^2\left(\frac{9}{2}+\frac{9}{2p}\right)\sum_{i=1}^n\mathbbm{E}\left[\left(\sum_{A\subsetneq [n]}k_{n,A}\sum_{j\not\in A}\sqrt{\|\Delta_jf(X)\|^p+\|\tilde{\Delta}_i\Delta_jf(X)\|^p}\left\|\Delta_{j}f(X^A)\right\|\left\|\tilde{\Delta}_i\Delta_{j}f(X)\right\|\right)^2\right]\notag\\
&+d^2\left(\frac{9}{2}+\frac{9}{2p}\right)\sum_{i=1}^n\mathbbm{E}\left[\left(\sum_{A\subsetneq [n]}k_{n,A}\sum_{j\not\in A}\sqrt{\|\Delta_jf(X)\|^p+\|\tilde{\Delta}_i\Delta_jf(X)\|^p}\left\|\tilde{\Delta}_i\Delta_{j}f(X^A)\right\|\left\|\tilde{\Delta}_i\Delta_{j}f(X)\right\|\right)^2\right]\notag\\
=& d^2\gamma_{2+p}^{2+p}.\label{final_v}
\end{align}
We also note that, using (\ref{general_fact}) and (\ref{tilde_bound}) as above,
\begin{align}
&\mathbbm{E}\left[\left(\sum_{A\subsetneq [n]}k_{n,A}\sum_{j\not\in A}\mathbbm{E}\left[\left.\left|\left(\Delta_jf(X^A)\right)_l\right|\,\left|\left(\Delta_jf(X)\right)_m\right|\right|X\right]\right)^2\right]\notag\\
\leq &\left(\sum_{A\subsetneq [n]}k_{n,A}\sum_{j\not\in A}\mathbbm{E}\left[\left|\left(\Delta_jf(X^A)\right)_l\right|\,\left|\left(\Delta_jf(X)\right)_m\right|\right]\right)^2\hspace{-2mm}+\hspace{-1mm}\frac{1}{2}\sum_{i=1}^n\mathbbm{E}\left[\left(\sum_{A\subsetneq [n]}k_{n,A}\sum_{j\not\in A}\tilde{\Delta}_i\left(\left|\left(\Delta_jf(X^A)\right)_l\right|\,\left|\left(\Delta_jf(X)\right)_m\right|\right)\right)^2\right]\notag\\
\leq &\left(\sum_{j=1}^n\mathbbm{E}\|\Delta_jf(X)\|^2\right)^2+\frac{3}{2}\sum_{i=1}^n\mathbbm{E}\left(\sum_{A\subsetneq [n]}k_{n,A}\sum_{j\not\in A}\tilde{\Delta}_i\left(\Delta_jf(X^A)\right)_l\left(\Delta_jf(X)\right)_m\right)^2+\frac{3}{2}\sum_{i=1}^n\mathbbm{E}\Bigg[\Bigg(\sum_{A\subsetneq [n]}\sum_{j\not\in A}k_{n,A}\notag\\
&\cdot\left(\Delta_jf(X^A)\right)_l\tilde{\Delta}_i\left(\Delta_jf(X)\right)_m\Bigg)^2\Bigg]+\frac{3}{2}\sum_{i=1}^n\mathbbm{E}\left[\left(\sum_{A\subsetneq [n]}k_{n,A}\sum_{j\not\in A}\tilde{\Delta}_i\left(\Delta_jf(X^A)\right)_l\tilde{\Delta}_i\left(\Delta_jf(X)\right)_m\right)^2\right]\notag\\
\leq &\left(\sum_{j=1}^n\mathbbm{E}\|\Delta_jf(X)\|^2\right)^2+3\sum_{i=1}^n\left(\sum_{j=1}^n\left(\mathbbm{E}\left\|\tilde{\Delta}_i\Delta_jf(X)\right\|^4\mathbbm{E}\left\|\Delta_jf(X)\right\|^4\right)^{1/4}\right)^2+\frac{3}{2}\sum_{i=1}^n\left(\sum_{j=1}^n\left(\mathbbm{E}\left\|\tilde{\Delta}_i\Delta_jf(X)\right\|^4\right)^{1/2}\right)^2\notag\\
<&\infty,\label{variance_exists}
\end{align}
which shows that the variance in the definition of $V_{l,m}^{(1)}$ and $V_{l,m}^{(2)}$ exists. 
\subsubsection{Step 5 - concluding argument}
Using (\ref{remaining_bound}) - (\ref{v_def}) and (\ref{final_v}), we obtain that, for all $t\in(0,1)$,
\begin{align}
&\left|\mathbbm{E}\left<\nabla f_{t,h,\Sigma}(W),W\right>-\left<\text{Hess}(f_{t,h,\Sigma}),T\right>_{H.S.}\right|\notag\\
\leq&\frac{\sqrt{6}d^{3/2}}{2\sqrt{t}}\|\Sigma^{-1}\|^{3/2}_{op}\left(\sqrt{\sum_{k,l,m=1}^d\left(U_{k,l,m}^{(2)}\right)^2}+\sqrt{d\sum_{l,m=1}^d\left(R_{l,m}^{(1)}\right)^2}+\sqrt{d\sum_{l,m=1}^d\left(R_{l,m}^{(2)}\right)^2}\right)\notag\\
\leq& \frac{\sqrt{6}d^{3/2}}{2\sqrt{t}}\|\Sigma^{-1}\|^{3/2}_{op}\Bigg(d^{3/2}\gamma_2^2+d^2\|\Sigma^{-1/2}\|_{op}\gamma_2^2+2d_{convex}(W,N_{\Sigma})d^{3/2}\gamma_1\notag\\
&\hspace{3cm}+d^{1/2}\sqrt{2\cdot d_{convex}(W,N_{\Sigma})d^2\gamma_3^3+2\|\Sigma^{-1/2}\|_{op}d^{5/2}\gamma_4^4}\Bigg).\label{final1}
\end{align}
Therefore, using (\ref{schulte_bound}), (\ref{stein_equation}), (\ref{j1}) and (\ref{final1}),
\begin{align*}
d_{convex}(W,N_{\Sigma})
\leq& \frac{4}{3}\|\Sigma^{-1}\|_{op}d\sqrt{(\log t)^2d_{convex}(W,N_{\Sigma})+530d^{5/6}}\sqrt{\mathbbm{E}\left\|\mathbbm{E}[T-\Sigma |W]\right\|_{H.S.}^2}\\
 &+\frac{2\sqrt{6}d^{3/2}}{3\sqrt{t}}\|\Sigma^{-1}\|^{3/2}_{op}\Bigg(d^{3/2}\gamma_2^2+d^2\|\Sigma^{-1/2}\|_{op}\gamma_2^2+2d_{convex}(W,N_{\Sigma})d^{3/2}\gamma_1\notag\\
&\hphantom{............................}+d^{3/2}\sqrt{2\cdot d_{convex}(W,N_{\Sigma})\gamma_3^3+2d^{1/2}\|\Sigma^{-1/2}\|_{op}\gamma_4^4}\Bigg) +\frac{20}{\sqrt{2}}d\frac{\sqrt{t}}{1-t},
\end{align*}
for all $t\in (0,1)$. Assuming that $t\in (0,1/2)$, we have $t^{1/4}|\log t|\leq 2$ and $1-t\geq 1/2$ and so, in this case,
\begin{align*}
d_{convex}(W,N_{\Sigma})
\leq& \frac{4}{3}\|\Sigma^{-1}\|_{op}d\sqrt{\frac{4}{\sqrt{t}}d_{convex}(W,N_{\Sigma})+530d^{5/6}}\sqrt{\mathbbm{E}\left\|\mathbbm{E}[T-\Sigma |W]\right\|_{H.S.}^2}\\
&+\frac{2\sqrt{6}d^{3/2}}{3\sqrt{t}}\|\Sigma^{-1}\|^{3/2}_{op}\Bigg(d^{3/2}\gamma_2^2+d^2\|\Sigma^{-1/2}\|_{op}\gamma_2^2+2d_{convex}(W,N_{\Sigma})d^{3/2}\gamma_1\notag\\
&\hphantom{............................}+d^{3/2}\sqrt{2\cdot d_{convex}(W,N_{\Sigma})\gamma_3^3+2d^{1/2}\|\Sigma^{-1/2}\|_{op}\gamma_4^4}\Bigg) +\frac{40}{\sqrt{2}}d\sqrt{t}.
\end{align*}
Suppose that $\gamma<\frac{1}{\sqrt{2}}$ (otherwise (\ref{convex_result}) clearly holds)  and choose $\sqrt{t}=\max\left\lbrace \frac{\sqrt{2}}{80d}d_{convex}(W,N_{\Sigma}),\gamma\right\rbrace$. Then
\begin{align*}
d_{convex}(W,N_{\Sigma})
\leq& \frac{4}{3}\|\Sigma^{-1}\|_{op}d\sqrt{\frac{320d}{\sqrt{2}}+530d^{5/6}}\,\,\gamma+\frac{2\sqrt{6}d^{3/2}}{3}\|\Sigma^{-1}\|^{3/2}_{op}\Bigg(d^{3/2}\gamma+d^2\|\Sigma^{-1/2}\|_{op}\gamma\notag\\
&+80\sqrt{2}d^{5/2}\gamma+d^{3/2}\sqrt{80\sqrt{2}d+2d^{1/2}\|\Sigma^{-1/2}\|_{op}}\,\,\gamma\Bigg)+\frac{40}{\sqrt{2}}d\gamma+\frac{1}{2}d_{convex}(W,N_{\Sigma})
\end{align*}
and (\ref{convex_result}) follows, as required. \qed
\subsection{Proof of Lemma \ref{lemma_bn}}\label{proof_lemma_bn}
We will follow a strategy similar to that used in the proof of \cite[Theorem 5.1]{peccati_lachieze-rey2017}. Recall that $T:=\sum_{A\subsetneq[n]}k_{n,A}T_A/2$ for $k_{n,A}=\frac{1}{{n\choose |A|}(n-|A|)}$. Note that, using the fact that, for any $g:\mathcal{X}^{2n}\to\mathbbm{R}$ and  $U=g(X,X')$, $\text{Var}(\mathbbm{E}(U|X))\leq \mathbbm{E}\left(\text{Var}(U|X')\right)$ (see \cite[Lemma 4.4]{new_stein}), we have
\begin{align}
\sqrt{\mathbbm{E}\left\|\mathbbm{E}[T|X]-\mathbbm{E}T\right\|_{H.S.}^2}
\leq& \frac{1}{2}\sum_{A\subsetneq [n]}k_{n,A}\sqrt{\mathbbm{E}\left\|\mathbbm{E}[T_A|X]-\mathbbm{E}T_A\right\|_{H.S.}^2}
\leq\frac{1}{2}\sum_{A\subsetneq [n]}k_{n,A}\sqrt{\mathbbm{E}\left\|T_A-\mathbbm{E}[T_A|X']\right\|_{H.S.}^2}.\label{t_bound}
\end{align}
Introduce the substitution operator
$$\tilde{S}_i(X)=\left(X_0,\dots,\tilde{X}_i,\dots,X_n\right).$$
Fix $A\subsetneq [n]$ and note that, by the Efron-Stein inequality \cite{steele},
\begin{align}
\frac{1}{2}\sum_{A\subsetneq [n]}k_{n,A}\sqrt{\mathbbm{E}\left\|T_A-\mathbbm{E}[T_A|X']\right\|_{H.S.}^2}\leq \frac{1}{\sqrt{8}}\sum_{A\subsetneq [n]}k_{n,A}\sqrt{\sum_{i=1}^{n}\mathbbm{E}\left\|\tilde{\Delta}_iT_A\right\|_{H.S.}^2}.\label{t_bound1}
\end{align}
Recall also that $T_A=\sum_{j\not\in A}\left[\Delta_jf(X)\right]\left[\Delta_jf(X^A)\right]^T\in\mathbbm{R}^{d\times d}$ and so
\begin{align}
&\sum_{i=1}^{n}\mathbbm{E}\left\|\tilde{\Delta}_i T_A\right\|_{H.S.}^2=\sum_{i=1}^{n}\sum_{j,k\not\in A}\mathbbm{E}\left<\tilde{\Delta}_i\left(\Delta_jf(X)\Delta_jf(X^A)^T\right),\tilde{\Delta}_i\left(\Delta_kf(X)\Delta_kf(X^A)^T\right)\right>_{H.S.}.\label{efron_ta_covering}
\end{align}
Moreover, note that for $j\not\in A$ and a fixed $i\in[n]$,
\begin{align}
\tilde{\Delta}_i\left(\Delta_jf(X)\Delta_jf(X^A)^T\right)=&\tilde{\Delta}_i\Delta_jf(X)\Delta_jf(X^A)^T+\Delta_jf(\tilde{S}_i(X))\tilde{\Delta}_i\Delta_jf(X^A)^T.\label{efron_ta2_covering}
\end{align}
We shall now analyse each summand in (\ref{efron_ta_covering}) separately. In order to do this, we introduce the vector $\bar{X}$ by $\bar{X}_i=\tilde{X}_i$ and $\bar{X}_l=X_l'$, if $l\neq i$ and for any $x\in \mathcal{X}^n$ and any mapping $\varphi:\mathcal{X}^n\to\mathbbm{R}^d$, we define
\begin{align*}
\bar{\Delta}_l\varphi(x)=\varphi(x)-\varphi(x_1,\dots,x_{l-1},\bar{X}_l,x_{l+1},\dots,x_{l+1},\dots,x_n).
\end{align*}
Note that, if $i,j,k$ are pairwise distinct, then, by (\ref{efron_ta2_covering}), the  corresponding summands in (\ref{efron_ta_covering}) are upper-bounded by
\begin{align}
&4\sup_{(Y,Y',Z,Z')}\mathbbm{E}\left|\left<\bar{\Delta}_i\bar{\Delta}_jf(Y)\bar{\Delta}_jf(Y')^T,\bar{\Delta}_i\bar{\Delta}_kf(Z)\bar{\Delta}_kf(Z')^T\right>_{H.S.}\right|\notag\\
=&4\sup_{(Y,Y',Z,Z')}\mathbbm{E}\left|\left<\Delta_i\Delta_jf(Y)\Delta_jf(Y')^T,\Delta_i\Delta_kf(Z)\Delta_kf(Z')^T\right>_{H.S.}\right|\notag\\
\leq &4\sup_{(Y,Y',Z,Z')}\mathbbm{E}\left[\mathbbm{1}_{\lbrace  \Delta_{i,j}f(Y)\neq 0, \Delta_{i,k}f(Z)\neq 0\rbrace}\left(\left\|\Delta_jf(Y)\Delta_jf(Y')^T\right\|_{H.S.}+\left\|\Delta_jf(Y^i)\Delta_jf(Y')^T\right\|_{H.S.}\right)\right.\notag\\
&\left.\phantom{......................}\cdot\left(\left\|\Delta_kf(Y)\Delta_kf(Y')^T\right\|_{H.S.}+\left\|\Delta_kf(Y^i)\Delta_kf(Y')^T\right\|_{H.S.}\right)\right]\notag\\
=&4\sup_{(Y,Y',Z,Z')}\mathbbm{E}\left[\mathbbm{1}_{\lbrace  \Delta_{i,j}f(Y)\neq 0, \Delta_{i,k}f(Z)\neq 0\rbrace}\left(\left\|\Delta_jf(Y)\right\|\,\left\|\Delta_jf(Y')\right\|+\left\|\Delta_jf(Y^i)\right\|\,\left\|\Delta_jf(Y')\right\|\right)\right.\notag\\
&\left.\phantom{......................}\cdot\left(\left\|\Delta_kf(Y)\right\|\,\left\|\Delta_kf(Y')\right\|+\left\|\Delta_kf(Y^i)\right\|\,\left\|\Delta_kf(Y')\right\|\right)\right]\notag\\
\leq &16\sup_{\left(Y,Y',Z,Z'\right)}\mathbbm{E}\left[\mathbbm{1}_{\lbrace  \Delta_{i,j}f(Y)\neq 0, \Delta_{i,k}f(Y')\neq 0\rbrace}\left\|\Delta_jf(Z)\right\|^2\left\|\Delta_kf(Z')\right\|^2\right]
=16B_n',\label{ta_bound1}
\end{align}
where $Y,Y',Z,Z'$ are recombinations of $\{X,X',\tilde{X}\}$.
Now, using similar computations and the Cauchy-Schwarz inequality, we obtain that, in the case $i\neq j=k$, the corresponding summands in (\ref{efron_ta_covering}) are upper-bounded by
\begin{align}
&4\sup_{(Y,Y',Z,Z')}\mathbbm{E}\left|\left<\bar{\Delta}_i\bar{\Delta}_jf(Y)\bar{\Delta}_jf(Y')^T,\bar{\Delta}_i\bar{\Delta}_kf(Z)\bar{\Delta}_kf(Z')^T\right>_{H.S.}\right|\notag\\
\leq &4\sup_{(Y,Y',Z,Z')}\mathbbm{E}\left[\left\|\bar{\Delta}_i\bar{\Delta}_jf(Y)\right\|\,\left\|\bar{\Delta}_jf(Y')\right\|\,\left\|\bar{\Delta}_i\bar{\Delta}_kf(Z)\right\|\,\left\|\bar{\Delta}_kf(Z')\right\|\right]\notag\\
\leq&4\sup_{(Y,Y')}\mathbbm{E}\left[\left\|\Delta_j\Delta_if(Y)\right\|^2\left\|\Delta_jf(Y')\right\|^2\right]\notag\\
\leq &4\sup_{(Y,Y')}\mathbbm{E}\left[\mathbbm{1}_{\lbrace \Delta_{i,j}f(Y)\neq 0\rbrace}\left(\left\|\Delta_if(Y)\right\|+\left\|\Delta_if(Y^j)\right\|\right)^2\left\|\Delta_jf(Y')\right\|^2\right]\notag\\
\leq &16\sup_{(Y,Y',Z)}\mathbbm{E}\left[\mathbbm{1}_{\lbrace \Delta_{i,j}f(Y)\neq 0\rbrace}\left\|\Delta_if(Z)\right\|^2\left\|\Delta_jf(Y')\right\|^2\right]=16B_n(f).\label{ta_bound2}
\end{align}

We note that, in the case $i=j$, using $\tilde{\Delta}_i\Delta_i(\cdot)=\tilde{\Delta}_i(\cdot)$, the Hilbert-Schmidt norm of the right-hand side of (\ref{efron_ta2_covering}) can be bounded by
\begin{align}
\left\|\tilde{\Delta}_if(X)\Delta_if(X^A)^T\right\|_{H.S.}+\left\|\Delta_if(\tilde{S}_i(X))\tilde{\Delta}_if(X^A)^T\right\|_{H.S.}
&= \left\|\tilde{\Delta}_if(X)\right\|\,\left\|\Delta_if(X^A)\right\|+\left\|\Delta_if(\tilde{S}_i(X))\right\|\,\left\|\tilde{\Delta}_if(X^A)\right\|\nonumber\\
&\hspace{-3.5cm}\leq \frac{1}{2}\left[\left\|\tilde{\Delta}_if(X)\right\|^2+\left\|\Delta_if(X^A)\right\|^2+\left\|\Delta_if(\tilde{S}_i(X))\right\|^2+\left\|\tilde{\Delta}_if(X^A)\right\|^2\right].\label{efron_ta3_covering}
\end{align}
If $i\not\in A$ and $i=j=k$ then it follows from (\ref{efron_ta3_covering}), that the corresponding summands in (\ref{efron_ta_covering}) are smaller than
\begin{align}
\frac{1}{4}\mathbbm{E}\left[\left(\left\|\Delta_if(X)\right\|^2+\left\|\tilde{\Delta}_if(X^A)\right\|^2+\left\|\Delta_if(\tilde{S}_i(X))\right\|^2+\left\|\tilde{\Delta}_if(X^A)\right\|^2\right)^2\right]\leq 4\mathbbm{E}\left\|\Delta_if(X)\right\|^4.\label{ta_bound3}
\end{align}

Finally, the corresponding summands in (\ref{efron_ta_covering}) in the case $i=j\neq k$ are upper-bounded by
\begin{align}
&4\sup_{(Y,Y',Z)}\mathbbm{E}\left|\left<\bar{\Delta}_if(Y)^T\bar{\Delta}_if(Y),\bar{\Delta}_i\bar{\Delta}_kf(Y')^T\bar{\Delta}_kf(Z)\right>_{H.S.}\right|\notag\\
\leq &4\sup_{(Y,Y',Z)}\mathbbm{E}\left[\mathbbm{1}_{\lbrace\Delta_{i,k}f(Y)\neq 0\rbrace}\left\|\Delta_if(Y')\right\|^2\left(\left\|\Delta_kf(Y)\right\|+\left\|\Delta_kf(Y^i)\right\|\right)\left\|\Delta_kf(Z)\right\|\right]\notag\\
\leq &8B_n(f).\label{ta_bound4}
\end{align}

Therefore, by (\ref{t_bound}) - (\ref{efron_ta_covering}), (\ref{ta_bound1}), (\ref{ta_bound2}), (\ref{ta_bound3}) and (\ref{ta_bound4}),
\begin{align*}
&\sqrt{\mathbbm{E}\left\|\mathbbm{E}[T|X]-\mathbbm{E}T\right\|_{H.S.}^2}\\
\leq& \sqrt{2n}\sum_{A\subsetneq [n]}k_{n,A}\sqrt{\sum_{j,k\not\in A}\left[\mathbbm{1}_{[j=k=1]}\mathbbm{E}\left\|\Delta_1f(X)\right\|^4+\left(\mathbbm{1}_{[j=k\neq 1]}+\mathbbm{1}_{[k\neq j=1]}\right)B_n(f)+\mathbbm{1}_{[1\neq k\neq j\neq 1]}B_n'(f)\right]}\\
\leq & \sqrt{2n}\left(\sqrt{\mathbbm{E}\left\|\Delta_1f(X)\right\|^4}\sum_{A\subsetneq[n]:1\not\in A}k_{n,A}+\sqrt{2B_n(f)}\sum_{A\subsetneq[n]}k_{n,A}\sqrt{n-|A|}+\sqrt{B_n'(f)}\sum_{A\subsetneq[n]}k_{n,A}(n-|A|)\right)\\
\leq &4\sqrt{n}\left(\sqrt{nB_n(f)}+\sqrt{n^2B_n'(f)}+\sqrt{\mathbbm{E}\|\Delta_1f(X)\|^4}\right),
\end{align*}
where we used the fact that $\sum_{A\subsetneq [n]}k_{n,A}\sqrt{n-|A|}=\sum_{k=1}^n\frac{1}{\sqrt{k}}\leq \int_0^n\frac{dx}{\sqrt{x}}=2\sqrt{n}.$\qed

\section{Proofs of the results related to covering processes}\label{covering process proofs}
In the arguments below we shall use the following notation. For any $k\in[n]$ and a set of $k$ pairwise distinct indices $T=\{i_1,\dots,i_k\}\subset[n]$ and for any $Z=(Z_1,\dots,Z_n)$ such that $Z_l\subset\mathbbm{R}^d$ for all $l\in[n]$,
$$F^T_n(Z)=F^{(i_1,i_2,\dots,i_k)}_n(Z):=\underset{l\not\in\lbrace i_1,\dots,i_k\rbrace}{\bigcup_{1\leq l\leq n}}Z_l.$$
We will also often write $F_n^T:=F^T_n(X)$ and $F^{(i_1,i_2,\dots,i_k)}_n:=F^{(i_1,i_2,\dots,i_k)}_n(X)$.
 For such $Z$ and for any $1\leq i\neq j\leq n$, let $Z^{\{j\}}=\left(Z_1,\dots,Z_{j-1},X_j',Z_{j+1},\dots,Z_n\right)$, $Z_{\{i\}}=\left(Z_1,\dots,Z_{i-1},\tilde{X}_i,Z_{i+1},\dots,Z_n\right)$ and $$Z^{\{i\},\{j\}}=\left(Z^{\{i\},\{j\}}_1,\dots,Z^{\{i\},\{j\}}_n\right),$$ where, for $l\in [n]$,
$Z^{\{i\},\{j\}}_l=Z^{\{j\}}_l$, if $l\neq i$ and $Z^{\{i\},\{j\}}_l=\tilde{X}_i$, if $l=i$. Furthermore, for $i\neq j$ we define
\begin{align*}
&\Delta_{j}f(Z)=f(Z)-f\left(Z^{\{j\}}\right),\quad \tilde{\Delta}_i\Delta_{j}f(Z)=f(Z)-f\left(Z^{\{j\}}\right)-f\left(Z_{\{i\}}\right)+f\left(Z^{\{i\},\{j\}}\right).
\end{align*}
By $\mathcal{K}^d$ we will denote the set of compact convex subsets of $\mathbbm{R}^d$. We will also often use the well-known bound $\min(\kappa_0,\dots,\kappa(d))\geq\frac{1}{d!}$, which readily gives that $V_i(L)\leq d!\overline{V}(L)$, for any $i=0,\dots,d$ and $L\in\mathcal{K}^d$.
Finally, we will repeatedly use the fact that, for any $i=0,\dots,d$, $V_i$ is a translation invariant and additive functional on the ring of finite unions of  convex bodies in $\mathbb{R}^d$ and it is non-negative and monotone on $\mathcal{K}^d$.

\subsection{Initial lemmas}
We start by quoting a result from \cite{last_penrose}.
\begin{lemma}[Proposition 22.5 in \cite{last_penrose}]\label{integral_lemma}
For all $L,M\in \mathcal{K}^d$,
$$\int_{\mathbbm{R}^d}\overline{V}\left((x+L)\cap M\right) dx\leq \overline{V}(L)\overline{V}(M).$$
\end{lemma}
Now, we prove a useful estimate.
\begin{lemma}\label{penrose_lemma}
For all $L\in\mathcal{K}^d$, $k=0,\dots,n-1$, and $i=0,\dots,d$,
\begin{align*}
&A)\quad\mathbbm{E}\left|\overline{V}\left(L\cap \bigcup_{l=k+1}^n X_l\right)\right|\leq e^{\overline{V}(K)}\overline{V}(L);\\
&B)\quad\mathbbm{E}\left|\overline{V}\left(L\cap \bigcup_{l=k+1}^n X_l\right)\right|^m\leq \left(2^{2^d+2d}d^{d/2}\right)^m e^{2^m(2R+1)^d}\overline{V}(L)^m,\quad\text{for all } m\geq 1.
\end{align*}
\end{lemma}
\begin{proof}
To prove A), we apply the inclusion-exclusion formula and the triangle inequality:
\begin{align*}
\mathbbm{E}\left|\overline{V}\left(L\cap \bigcup_{l=k+1}^n X_l\right)\right|\leq&\sum_{l=1}^n{n\choose l}\mathbbm{E}\left[\overline{V}(L\cap X_1\cap\dots\cap X_l)\right]\\
\leq&\sum_{l=1}^n \frac{{n\choose l}}{n^l}\int_{E_n}\dots\int_{E_n}\overline{V}(L\cap(K+x_1)\cap\dots\cap(K+x_l))dx_1\dots dx_l\\
\stackrel{\text{Lemma }\ref{integral_lemma}}\leq&\overline{V}(L)\sum_{l=1}^n\frac{1}{l!}\overline{V}(K)^l
\leq e^{\overline{V}(K)}\overline{V}(L).
\end{align*}

To prove B), we use a strategy similar to that of the proof of \cite[Proposition 22.4]{last_penrose}. Let $Q_0=[-1/2,1/2]^d$ and, for any $z\in\mathbbm{R}^d$, $Q_z=Q_0+z$. Let $I(L)=\{z\in\mathbbm{Z}^d:Q_z\cap L\neq\emptyset\}$ and, for $C\subset\mathbbm{R}^d$, let
$N(C):=\sum_{l=k+1}^n\mathbbm{1}_{[X_l\cap C\neq\emptyset]}$
be the number of grains in $\{X_{k+1},\dots,X_n\}$ hitting $C$. For each non-empty $I\subset I(L)$, we fix some $z(I)\in I$. As in the proof of \cite[Proposition 22.4]{last_penrose}, the inclusion-exclusion formula yields
\begin{align}
\mathbbm{E}\left|V_l\left(L\cap \bigcup_{l=k+1}^nX_l\right)\right|^m\leq&\mathbbm{E}\Bigg|\sum_{I\subset I(L):I\neq \emptyset}\mathbbm{1}\left[\bigcap_{z\in I}Q_z\neq \emptyset\right]2^{N(Q_{Z(I)})}\overline{V}(Q_0)\Bigg|^m\label{penrose1}
\end{align}
Using \cite[(22.26)]{last_penrose} and the display directly below it, we obtain that
\begin{align}
\text{card}\bigg\{I\subset I(L): I\neq\emptyset,\bigcap_{z\in I}Q_z\neq\emptyset\bigg\}\leq 2^{2^d}\sum_{l=0}^d\kappa_{d-l}d^{(d-l)/2}V_l(L)\leq 2^{2^d}d^{d/2}\overline{V}(L).\label{penrose2}
\end{align}
Now, for any $z\in\mathbbm{R}^d$, let $Q_{R,z}=[z-1/2-R,z+1/2+R]^d$. For any $l=k+1,\dots,n$, let $U_l$ be the centre of $X_l$. Note that, for $n$ large enough so that $Q_{R,z}\subset E_n$,
\begin{align}
\mathbbm{E}\left[2^{mN(Q_z)}\right]\leq\mathbbm{E}\left[2^{m\sum_{l=k+1}^n\mathbbm{1}_{[U_l\in Q_{R,z}]}}\right]=\left(\mathbbm{E}\left[2^{m\mathbbm{1}_{[U_1\in Q_{R,z}]}}\right]\right)^{n-k}\leq \left(1+2^{m}\mathbbm{E}\left[\mathbbm{1}_{[U_1\in Q_{R,z}]}\right]\right)^n&\leq \left(1+2^{m}\frac{V_d(Q_{R,z})}{n}\right)^n\notag\\
&\hspace{-1.1cm}\leq \exp\left(2^m(2R+1)^d\right).\label{penrose3}
\end{align}
The result now follows from (\ref{penrose1})-(\ref{penrose3}) and the fact that $\overline{V}(Q_0)=\sum_{l=0}^d\kappa_{d-l}{d\choose l}\leq \sum_{l=0}^d\left(\sqrt{2\pi}\right)^{d-l}{d\choose l}\leq 4^d$ (see \cite[Example 1.3]{tropp} and \cite[page 224-227]{san}).
\end{proof}
\subsection{Proof of Theorem \ref{pos_def_teorem}}\label{pos_def_proof}
\subsubsection{Introduction}\label{intro_cov}
Let $\eta$ be the stationary Poisson process with the intensity measure
$\Lambda(\cdot)=\int\mathbbm{1}_{[K+x\in\cdot]}dx$ and let $Z\equiv Z(\eta)$ be given by $Z:=\bigcup_{L\in\eta}L$. For $i=0,\dots,d$, let $V_i^*(L):=\mathbbm{E}V_i(Z\cap L)-V_i(L)$, for $L\in\mathcal{K}^d$. It follows from \cite[Theorem 6.4, (5.2), Theorem 3.1]{hug_last_schulte} that
\begin{align*}
\Sigma_{i,j}=\sum_{k=2}^{\infty}\frac{1}{k!}\int_{\mathbbm{R}^d}\dots\int_{\mathbbm{R}^d}V_i^*\left(K\cap(K+x_2)\cap\dots\cap(K+x_k)\right)V_j^*\left(K\cap(K+x_2)\cap\dots\cap(K+x_k)\right)dx_2\dots dx_k
\end{align*}
and the infinite sum in the definition of $\Sigma$ converges by \cite[Theorem 3.1]{hug_last_schulte}. Therefore, $\Sigma$ is well-defined.

Now, for every vector $a=(a_0,\dots,a_d)^T\in\mathbbm{R}^{d+1}$,
\begin{align*}
a^T\Sigma a=\sum_{k=2}^{\infty}\frac{1}{k!}\int_{\mathbbm{R}^d}\dots\int_{\mathbbm{R}^d}\left(\sum_{l=0}^da_lV_l^*\left(K\cap(K+x_2)\cap\dots\cap(K+x_k)\right)\right)^2dx_2\dots dx_k.
\end{align*}
In order to prove positive-definiteness of $\Sigma$, it suffices to show that the summand obtained for $k=d+1$ is strictly positive. This is shown in the proof of \cite[Theorem 4.1]{hug_last_schulte} so we readily obtain that $\Sigma$ is positive-definite.

Next, we prove lemmas that will help us bound the distance between $\Sigma_n$ and $\Sigma$.
\subsubsection{Auxiliary lemmas}
Now, we prove the first part of the bound on the rate of convergence to the limiting covariance.
\begin{lemma}\label{lemma_cov1}
For $i=0,\dots,d$, $k=1,2,\dots$, let
$$\psi^i_k(x_1,\dots,x_k):=\mathbbm{E}\left[V_i\left(F_n^{(1,2,\dots,k)}\cap x_1\cap\dots\cap x_k\right)\right]-V_i\left(x_1\cap\dots\cap x_k\right).$$
If $n>\overline{V}(K)e$, then
\begin{align*}
&\left|\text{Cov}\left((f_i(X),f_j(X)\right)-\frac{1}{n}\sum_{k=1}^n{n\choose k}\mathbbm{E}\left[\psi^i_k(X_1,\dots,X_k)\psi^j_k(X_1,\dots,X_k)\right]+\mathbbm{E}\left[\psi^i_1(X_1)\right]\mathbbm{E}\left[\psi^j_1(X_1)\right]\right|\\
\leq &(d!)^2\left(1+e^{\overline{V}(K)}\right)^2\left(\overline{V}(K)^2e^{\overline{V}(K)+3\overline{V}(K)^2}+e^{\overline{V}(K)^2}\right)n^{-1}.
\end{align*}
\end{lemma}
\begin{proof}
Let $[k]=\{1,\dots,k\}$ and for all $i=0,\dots,d$ and $k=1,\dots,n$, let
\begin{align}
\phi_k^{i}(x_1,\dots,x_k)
=&\frac{1}{\sqrt{n}}\sum_{m=0}^k(-1)^m\underset{|T|=m}{\sum_{T\subset[k]}}\varphi^i_{T}(x_1,\dots,x_k),\label{covariance2}
\end{align}
where, for $T\subset[k]$, such that $|T|=m$ and $T=\{t_1,\dots,t_m\}$,
\begin{align*}
\varphi^i_{T}(x_1,\dots,x_k):=&\mathbbm{E}\left[V_i\left( x_{t_1}\cap x_{t_2}\cap\dots \cap x_{t_m}\cap X_1\cap\dots\cap X_{k-m}\right)\right]\\
&-\mathbbm{E}\left[V_i\left(F_n^{(1,2,\dots, k)}\cap x_{t_1}\cap x_{t_2}\cap\dots \cap x_{t_m}\cap X_1\cap\dots\cap X_{k-m}\right)\right]. 
\end{align*}
Note that $\psi_k^i(x_1,\dots,x_k)=-\varphi^i_{[k]}(x_1,\dots,x_k)$.
Using \cite[Theorem 2.2]{peccati_lachieze-rey2017}, we have that
\begin{align}
\text{Cov}\left((f_i(X),f_j(X)\right)=&\sum_{k,l=1}^n\sum_{1\leq r_1<\dots<r_k\leq n}\sum_{1\leq s_1<\dots<s_l\leq n}(-1)^{k+l}\mathbbm{E}\left[\phi_k^i(X_{r_1},\dots,X_{r_k})\phi_l^j(X_{s_1},\dots,X_{s_l})\right]\notag\\
=&\sum_{k=1}^n{n\choose k}\mathbbm{E}\left[\phi_k^i(X_{1},\dots,X_{k})\phi_k^j(X_{1},\dots,X_{k})\right].\label{covariance3}
\end{align}
More precisely, identity (\ref{covariance3}) follows from the fact that, for $L\in\mathcal{K}^d$ and for all $i=0,\dots,d$ and $k,j_1,\dots,j_k=1,\dots,n$,
\begin{align*}
&V_i\left((F_n^{(j_1,\dots,j_k)}\cup X_{j_k})\cap L\right)-V_i\left((F_n^{(j_1,\dots,j_k)}\cup X_{j_k}')\cap L)\right)\\
&=V_i(X_{j_k}\cap L)-V_i(F_n^{(j_1,\dots,j_k)}\cap X_{j_k}\cap L)-V_i(X_{j_k}'\cap L)+V_i(F_n^{(j_1,\dots,j_k)}\cap X_{j_k}'\cap L)
\end{align*}
and, as a consequence, using the notation of \cite{peccati_lachieze-rey2017},
\begin{align*}
\Delta_{1}\Delta_{2}\dots\Delta_{k}f_i(X)=&\frac{(-1)^{k+1}}{\sqrt{n}}\Bigg[V_i\left( X_1\cap\dots\cap X_{k}\right)-V_i\left(F_n^{(1,2,\dots, k)} \cap X_1\cap\dots\cap X_{k}\right)\Bigg]\\
&+\sum_{m=1}^k\frac{(-1)^{m+k+1}}{\sqrt{n}}\sum_{1\leq t_1<\dots<t_m\leq k}\Bigg[V_i\left( X'_{t_1}\cap X'_{t_2}\cap\dots \cap X'_{t_m}\cap X_1\cap\dots\cap X_{k-m}\right)\\
&\hspace{3cm}-V_i\left(F_n^{(1,2,\dots, k)}\cap X'_{t_1}\cap X'_{t_2}\cap\dots \cap X'_{t_m}\cap X_1\cap\dots\cap X_{k-m}\right)\Bigg].
\end{align*}
Now, for $m_1,m_2=0,\dots,k$, $m_1\geq m_2$, $i,j=0,\dots,d$, $T=\lbrace t_1,\dots,t_{m_1}\rbrace\subset [k]$,  $S=\lbrace s_1,\dots,s_{m_2}\rbrace\subset [k]$, such that $|T|=m_1$, $|S|=m_2$,
\begin{align*}
&\frac{1}{n}\left|\mathbbm{E}\varphi^i_{T}(X_1,\dots,X_k)\varphi^j_{S}(X_1,\dots,X_k)\right|\\
=&\frac{1}{n}\Bigg|\mathbbm{E}\bigg\{\left[V_i\left(X_{t_1}\cap\dots\cap X_{t_{m_1}}\cap X_{1}'\cap\dots\cap X_{k-m_1}'\right)-V_i\left(F_n^{(1,2,\dots, k)}(X')\cap X_1\cap\dots\cap X_{t_{m_1}}\cap X_{1}'\cap\dots\cap X_{k-m_1}'\right)\right]\\
&\cdot\left[V_j\left(X_{s_1}\cap\dots\cap X_{s_{m_2}}\cap \tilde{X}_{1}\cap\dots\cap \tilde{X}_{k-m_2}\right)-V_j\left(F_n^{(1,2,\dots, k)}(\tilde{X})\cap X_{s_1}\cap\dots\cap X_{s_{m_2}}\cap \tilde{X}_{1}\cap\dots\cap \tilde{X}_{k-m_2}\right)\right]\bigg\}\Bigg|\\
\stackrel{\text{Lemma }\ref{penrose_lemma}}\leq &\frac{(d!)^2\left(1+e^{\overline{V}(K)}\right)^2}{n}\mathbb{E}\left[\overline{V}\left(X_{t_1}\cap\dots\cap X_{t_{m_1}}\cap X_{1}'\cap\dots\cap X_{k-m_1}'\right)\overline{V}\left(X_{s_1}\cap\dots\cap X_{s_{m_2}}\cap \tilde{X}_{1}\cap\dots\cap \tilde{X}_{k-m_2}\right)\right]\\
\stackrel{\text{Lemma }\ref{integral_lemma}}\leq &(d!)^2\left(1+e^{\overline{V}(K)}\right)^2
\Bigg(\frac{1}{n^{2k-1}}\overline{V}(K)^{2k}\mathbbm{1}_{[T\cap S=\emptyset]}+\frac{1}{n^{2k-|T\cap S|}}\left(\overline{V}(K)^2\right)^{k-|T\cap S|/2+1}\mathbbm{1}_{[T\cap S\neq\emptyset]}\Bigg).
\end{align*}
Now, we consider the case where $T\cap S\neq\emptyset$ and ($T\neq [k]$ or $S\neq [k]$). In the expression below, $l=|T\cap S|$. Note that, using the above calculation and Lemma \ref{integral_lemma},
\begin{align}
&\frac{1}{n}\sum_{k=1}^n{n\choose k}\sum_{m_1,m_2=0}^k\underset{|S\cap T|\not\in\lbrace 0,k\rbrace}{\underset{|T|=m_1,|S|=m_2}{\sum_{T,S\subset[k]}}}\left|\mathbbm{E}\varphi^i_{T}(X_1,\dots,X_k)\varphi^j_{S}(X_1,\dots,X_k)\right|\nonumber\\
\leq&(d!)^2\left(1+e^{\overline{V}(K)}\right)^2\sum_{k=1}^n{n\choose k}\sum_{l=1}^{k-1} {k\choose l}\sum_{r_1=0}^{k-l}{k-l\choose r_1}\sum_{r_2=0}^{k-l-r_1}{k-l-r_1\choose r_2}\frac{\left(\overline{V}(K)^2\right)^{k-l/2+1}}{n^{2k-l}}\nonumber\\
=&(d!)^2\left(1+e^{\overline{V}(K)}\right)^2\sum_{k=1}^n{n\choose k}\sum_{l=1}^{k-1} {k\choose l}\sum_{r_1=0}^{k-l}{k-l\choose r_1}2^{k-l-r_1}\frac{\left(\overline{V}(K)^2\right)^{k-l/2+1}}{n^{2k-l}}\nonumber\\
=&(d!)^2\left(1+e^{\overline{V}(K)}\right)^2\sum_{k=1}^n{n\choose k}\sum_{l=1}^{k-1} {k\choose l}3^{k-l}\frac{\left(\overline{V}(K)^2\right)^{k-l/2+1}}{n^{2k-l}}\nonumber\\
\leq&(d!)^2\left(1+e^{\overline{V}(K)}\right)^2\overline{V}(K)^2\sum_{k=1}^n\frac{n^k}{k!}\left(3\overline{V}(K)^2\right)^k\sum_{l=1}^{k-1} {k\choose l}\frac{1}{n^{2k-l}}\left(3\overline{V}(K)\right)^{-l}\nonumber\\
\leq & \frac{(d!)^2\left(1+e^{\overline{V}(K)}\right)^2\overline{V}(K)^2}{n}\sum_{k=1}^n\frac{1}{k!}\left(3\overline{V}(K)^2\right)^k\left(\left(3\overline{V}(K)\right)^{-1}+1\right)^k\nonumber\\
\leq &(d!)^2\overline{V}(K)^2\left(1+e^{\overline{V}(K)}\right)^2e^{\overline{V}(K)+3\overline{V}(K)^2}n^{-1}.
\label{covariance5}
\end{align}
Also, note that, in the case $T\cap S=\emptyset$, for $k\geq 2$, 
\begin{align}
&\frac{1}{n}\sum_{k=2}^n{n\choose k}\Bigg|\underset{T\cap S=\emptyset}{\sum_{T,S\subset[k]}}\mathbbm{E}\left[\varphi^i_{T}(X_1,\dots,X_k)\varphi^j_{S}(X_1,\dots,X_k)\right]\Bigg|\notag\\
=&\frac{1}{n}\sum_{k=2}^n{n\choose k}\Bigg|\underset{T\cap S=\emptyset}{\sum_{T,S\subset[k]}}(-1)^{|T|+|S|}\mathbb{E}\left[V_i(F_n^{(1,2,\dots,k)}\cap X_1\cap\dots\cap X_k)-V_i(X_1\cap\dots\cap X_k)\right]\notag\\
&\phantom{..............................}\cdot\mathbb{E}\left[V_j(F_n^{(1,2,\dots,k)}\cap X_1\cap\dots\cap X_k)-V_j(X_1\cap\dots\cap X_k)\right]\Bigg|\notag\\
=&\frac{1}{n}\sum_{k=2}^n{n\choose k}\Bigg|\sum_{t=0}^k(-1)^t{k\choose t}\sum_{s=0}^{k-t}(-1)^s{k-t\choose s}\mathbb{E}\left[V_i(F_n^{(1,2,\dots,k)}\cap X_1\cap\dots\cap X_k)-V_i(X_1\cap\dots\cap X_k)\right]\notag\\
&\phantom{..............................}\cdot\mathbb{E}\left[V_j(F_n^{(1,2,\dots,k)}\cap X_1\cap\dots\cap X_k)-V_j(X_1\cap\dots\cap X_k)\right]\Bigg|\notag\\
=&\frac{1}{n}\sum_{k=2}^n{n\choose k}\Bigg|\sum_{t=0}^k(-1)^t{k\choose t}\mathbbm{1}{[k=t]}\mathbb{E}\left[V_i(F_n^{(1,2,\dots,k)}\cap X_1\cap\dots\cap X_k)-V_i(X_1\cap\dots\cap X_k)\right]\notag\\
&\phantom{..............................}\cdot\mathbb{E}\left[V_j(F_n^{(1,2,\dots,k)}\cap X_1\cap\dots\cap X_k)-V_j(X_1\cap\dots\cap X_k)\right]\Bigg|\notag\\
=&\frac{1}{n}\sum_{k=2}^n{n\choose k}\Bigg|(-1)^k\mathbb{E}\left[V_i(F_n^{(1,2,\dots,k)}\cap X_1\cap\dots\cap X_k\cap E_n)-V_i(X_1\cap\dots\cap X_k)\right]\notag\\
&\phantom{..............................}\cdot\mathbb{E}\left[V_j(F_n^{(1,2,\dots,k)}\cap X_1\cap\dots\cap X_k)-V_j(X_1\cap\dots\cap X_k)\right]\Bigg|\notag\\
\leq&\sum_{k=2}^n\frac{(d!)^2\left(1+e^{\overline{V}(K)}\right)^2}{n^{2k-1}}{n\choose k}\overline{V}(K)^{2k}\nonumber\\
\leq &(d!)^2\left(1+e^{\overline{V}(K)}\right)^2\sum_{k=2}^n\frac{\overline{V}(K)^{2k}}{n^{k-1}k!}\nonumber\\
\leq&(d!)^2\left(1+e^{\overline{V}(K)}\right)^2e^{\overline{V}(K)^2}n^{-1}.
\label{covariance5.5}
\end{align}
where we use Lemmas \ref{integral_lemma} and \ref{penrose_lemma} in the last three inequalities.
The result now follows from (\ref{covariance2})---(\ref{covariance5.5}).
\end{proof}
Now, we show the second part of the bound on the rate of convergence to the limiting covariance.
\begin{lemma}\label{lemma_cov2}
Using the notation of Lemma \ref{lemma_cov1}, we have that, for all $i,j=0,\dots,d$ and $\Sigma$ defined in (\ref{matrix_sigma}),
\begin{align*}
&\left|\frac{1}{n}\sum_{k=1}^n{n\choose k}\mathbbm{E}\left[\psi_k^i(X_1,\dots,X_k)\psi_k^j(X_1,\dots,X_k)\right]-\mathbbm{E}\left[\psi_1^i(X_1)\right]\mathbbm{E}\left[\psi_1^j(X_1)\right]-\Sigma_{i,j}\right|\\
\leq& (19+8(4R+1)^d)(d!)^2\max(1,\overline{V}(K)^3)\left(e^{\overline{V}(K)}+1\right)^2e^{\max(1,2\overline{V}(K))}n^{-1/d}.
\end{align*}
\end{lemma}
\begin{proof}
Let $E_{\left(n^{1/d}+4R\right)^d}\subseteq\mathbbm{R}^d$ be the cube with volume $\left(n^{1/d}+4R\right)^d$, centred at the origin.  Note that, for any $l\in\mathbbm{N}$, $i=0,\dots,d$,
\begin{align}
&\frac{1}{n}\left|\int_{E_n}\dots\int_{E_n}V_i\left((K+x_1)\cap\dots\cap(K+x_l)\right)dx_1\dots dx_l-\int_{E_n}\int_{\mathbbm{R}^d}\dots\int_{\mathbbm{R}^d}V_i\left((K+x_1)\cap\dots\cap(K+x_l)\right)dx_1\dots dx_l\right|\notag\\
\leq &\frac{1}{n}\sum_{m=1}^{l-1}\underbrace{\int_{E_n}\dots\int_{E_n}}_{m\text{ times}}\underbrace{\int_{\mathbbm{R}^d}\dots\int_{\mathbbm{R}^d}}_{(l-m-1)\text{ times}}\int_{\mathbbm{R}^d\setminus E_n}V_i\left((K+x_1)\cap\dots\cap(K+x_l)\right)dx_1\dots dx_l\notag\\
\leq &\frac{l-1}{n}\int_{\mathbbm{R}^d}\int_{\mathbbm{R}^d}\dots\int_{\mathbbm{R}^d}\int_{E_{\left(n^{1/d}+4R\right)^d}\setminus E_n}V_i\left((K+x_1)\cap\dots\cap(K+x_l)\right)dx_1\dots dx_l\notag\\
\stackrel{(*)}=&\frac{l-1}{n}\int_{\mathbbm{R}^d}\int_{\mathbbm{R}^d}\dots\int_{\mathbbm{R}^d}V_i\left((K+x_1)\cap\dots\cap(K+x_{l-1})\cap K\right)dx_1\dots dx_{l-1}\cdot\int_{E_{\left(n^{1/d}+4R\right)^d}\setminus E_n}dx_l\notag\\
\stackrel{(**)}\leq &\frac{(l-1)d!}{n}\left(\left(n^{1/d}+4R\right)^d-n\right)\overline{V}(K)^l\notag\\
\leq &\frac{(l-1)(d!)(4R+1)^d}{n^{1/d}}\overline{V}(K)^l,\label{integral_bound}
\end{align}
where we have used translation invariance of $V_i$ and of the Lebesgue measure on $\mathbbm{R}^d$ in $(*)$ and Lemma \ref{integral_lemma} in $(**)$.

Now, we use Fubini's theorem and Hadwiger's general integral geometric theorem \cite[Theorem 5.1.2]{geometry} in a manner analogous to the one in which it was done in the first display on \cite[page 387]{geometry}. We note that $K$ stays invariant under rotations about the origin. We further use the formula of the second display on \cite[page 388]{geometry} (which is derived from the Crofton formula \cite[(5.6)]{geometry}). In total, letting $c_{j}^m=\frac{m!\kappa_m}{j!\kappa_j}$, $m,j\in\{0,\dots,d\}$ and $(V_i)_m=c_i^{d-m}c_d^{m+i}V_{m+i}\mathbbm{1}_{[m+i\leq d]}$, $i,m\in\{0,\dots,d\}$, we obtain that
\begin{align}
&\int_{\mathbbm{R}^d}\dots\int_{\mathbbm{R}^d}V_i\left((K+x_1)\cap\dots\cap(K+x_l)\cap L\right)dx_1\dots dx_l\notag
=\underset{r_0+\dots+r_l=ld}{\sum_{0\leq r_0,\dots r_l\leq d}}c_{d-r_0}^d\left[(V_i)_{r_0}(L)\right]\prod_{m=1}^lc_d^{r_m}V_{r_m}(K)\notag\\
&\hspace{5cm}=\sum_{r_0=0}^d\underset{r_1+\dots+r_l=ld-r_0}{\sum_{0\leq r_1,\dots,r_l\leq d}}c_{d-r_0}^dc_i^{d-r_0}c_d^{r_0+i}V_{r_0+i}(L)\mathbbm{1}_{[r_0+i\leq d]}\prod_{m=1}^lc_d^{r_m}V_{r_m}(K)\notag\\
&\hspace{4.7cm}\stackrel{s:=r_0+i}=\sum_{s=i}^d\tilde{P}_{s,l,i}(d)V_{s}(L).\label{p_identity}
\end{align}
for
\begin{align*}
\tilde{P}_{s,l,i}(d)
=&\underset{r_1+\dots+r_l=ld-s+i}{\sum_{0\leq r_1,\dots,r_l\leq d}}\frac{s!\kappa_{s}}{i!\kappa_{i}}\prod_{m=1}^l\frac{r_m!\kappa_{r_m}}{d!\kappa_d}V_{r_m}(K).
\end{align*}

Furthermore, using the inclusion-exclusion principle, for any $L\in\mathcal{K}^d$,
\begin{align}\label{inclusion_exclusion}
\mathbbm{E}\left[V_i\left(F_n^{(1,\dots,k)}\cap L\right)\right]
=&\sum_{l=1}^{n-k}\frac{ {n-k\choose l}}{n^l}(-1)^{l-1}\int_{E_n}\dots\int_{E_n}V_i\left((K+x_1)\cap\dots\cap(K+x_l)\cap L\right)dx_1\dots dx_l.
\end{align}

It follows that
\begin{align}
&\Bigg|\frac{1}{n}\int_{E_n}\dots\int_{E_n}\left[\mathbbm{E}V_i\left(F_n^{(1,\dots,k)}\cap (K+x_1)\cap\dots\cap (K+x_k)\right)-V_i((K+x_1)\cap\dots\cap (K+x_k))\right]\notag\\
&\hspace{2cm}\cdot\left[\mathbbm{E}V_j\left(F_n^{(1,\dots,k)}\cap (K+x_1)\cap\dots\cap (K+x_k)\right)-V_j((K+x_1)\cap\dots\cap (K+x_k))\right]dx_1\dots dx_k\notag\\
&-\int_{\mathbbm{R}^d}\dots\int_{\mathbbm{R}^d}\bigg[\sum_{l=1}^{n-k}\frac{(-1)^{l-1}{n-k\choose l}}{n^l}\sum_{s=i}^d\tilde{P}_{s,l,i}(d)V_{s}(K\cap(K+x_2)\cap\dots\cap (K+x_k))-V_i(K\cap(K+x_2)\dots\cap (K+x_k))\bigg]\notag\\
&\cdot\bigg[\sum_{l=1}^{n-k}\frac{(-1)^{l-1}{n-k\choose l}}{n^l}\sum_{s=j}^d\tilde{P}_{s,l,j}(d)V_{s}(K\cap(K+x_2)\cap\dots\cap (K+x_k))-V_j(K\cap(K+x_2)\cap\dots\cap (K+x_k))\bigg]dx_2\dots dx_k \Bigg|\notag\\
=&\frac{1}{n}\Bigg|\int_{E_n}\dots\int_{E_n}\left[\mathbbm{E}V_i\left(F_n^{(1,\dots,k)}\cap (K+x_1)\cap\dots\cap (K+x_k)\right)-V_i((K+x_1)\cap\dots\cap (K+x_k))\right]\notag\\
&\hspace{3cm}\cdot\left[\mathbbm{E}V_j\left(F_n^{(1,\dots,k)}\cap (K+x_1)\cap\dots\cap (K+x_k)\right)-V_j((K+x_1)\cap\dots\cap (K+x_k))\right]dx_1\dots dx_k\notag\\
&-\int_{E_n}\int_{\mathbbm{R}^d}\dots\int_{\mathbbm{R}^d}\bigg[\sum_{l=1}^{n-k}\frac{(-1)^{l-1}{n-k\choose l}}{n^l}\sum_{s=i}^d\tilde{P}_{s,l,i}(d)V_{s}((K+x_1)\cap\dots\cap (K+x_k))-V_i((K+x_1)\cap\dots\cap (K+x_k))\bigg]\notag\\
&\hspace{1cm}\cdot\bigg[\sum_{l=1}^{n-k}\frac{(-1)^{l-1}{n-k\choose l}}{n^l}\sum_{s=j}^d\tilde{P}_{s,l,j}(d)V_{s}((K+x_1)\cap\dots\cap (K+x_k))-V_j((K+x_1)\cap\dots\cap (K+x_k))\bigg]dx_1\dots dx_k \Bigg|\notag\\
\leq &2d!\left(\sum_{l=1}^{n-k}\frac{{n-k\choose l}}{n^l}\overline{V}(K)^{l+1}+\overline{V}(K)\right)\left(\sum_{l=1}^{n-k}\frac{{n-k\choose l}}{n^l}\frac{(l+k-1)(d!)(4R+1)^d}{n^{1/d}}\overline{V}(K)^{k+l}+\frac{(k-1)(d!)(4R+1)^d}{n^{1/d}}\overline{V}(K)^{k}\right)\notag\\
\leq &2(d!)^2(4R+1)^d\overline{V}(K)^{k+1}\left(e^{\overline{V}(K)}+1\right)\left((\overline{V}(K)+k-1)e^{\overline{V}(K)}+k-1\right)n^{-1/d},\label{covariance7}
\end{align}
where the first identity holds by the translation invariance of $V_i$ and the Lebesgue measure on $\mathbbm{R}^d$ and the second to last inequality follows from (\ref{inclusion_exclusion}),  (\ref{p_identity}),  (\ref{integral_bound}) and Lemma \ref{integral_lemma}.

Now, we use the trick applied in the last display on \cite[page 387]{geometry}, in the derivation of \cite[Theorem 9.1.3]{geometry}. In the calculation below, $t$ will denote the number of indices $r_1,\dots,r_l$ which are smaller than $d$ and $r=l-t$. We obtain a result analogous to that of the third display on \cite[page 388]{geometry}, namely
\begin{align}
&\sum_{l=1}^{\infty}\frac{(-1)^{l-1}}{l!}\sum_{s=i}^d\tilde{P}_{s,l,i}(d)V_s(L)-V_i(L)\notag\\
=&\sum_{l=1}^{\infty}\frac{(-1)^{l-1}}{l!}\sum_{s=i}^dV_s(L)\underset{r_1+\dots+r_l=ld-s+i}{\sum_{0\leq r_1,\dots,r_l\leq d}}\frac{s!\kappa_{s}}{i!\kappa_{i}}\prod_{m=1}^l\frac{r_m!\kappa_{r_m}}{d!\kappa_d}V_{r_m}(K)-V_i(L)\notag\\
=& \sum_{s=i}^d\sum_{l=1}^{\infty}\frac{(-1)^{l-1}}{l!}V_i(L)\left[\mathbbm{1}_{[s=i]}V_d(L)^l+\mathbbm{1}_{[s\neq i]}V_s(L)\sum_{t=1}^{s-i}{l\choose t}V_d(K)^{l-t}\underset{r_1+\dots+r_t=td-s+i}{\sum_{0\leq r_1,\dots,r_t\leq d-1}}\frac{s!\kappa_s}{i!\kappa_i}\prod_{m=1}^t\frac{r_m!\kappa_{r_m}}{d!\kappa_d}V_{r_m}(K)\right]-V_i(L)\notag
\\
=&-e^{-V_d(K)}V_i(L)-\mathbbm{1}_{[i\neq d]}\sum_{s=i+1}^dV_s(L)\sum_{r=0}^\infty V_d(K)^r\sum_{t=1}^{s-i}\frac{(-1)^r}{(r+t)!}{r+t\choose t}(-1)^t\underset{r_1+\dots+r_t=td-s+i}{\sum_{0\leq r_1,\dots,r_t\leq d-1}}\frac{s!\kappa_s}{i!\kappa_i}\prod_{m=1}^t\frac{r_m!\kappa_{r_m}}{d!\kappa_d}V_{r_m}(K)\notag\\
=&-\sum_{s=i}^dV_s(L)P_{i,s}(d),\label{pp_identity}
\end{align}
for $$P_{i,s}(d)=e^{-V_d(K)}\Bigg[\mathbbm{1}_{[s=i]}+\mathbbm{1}_{[i\neq d]}\frac{s!\kappa_s}{i!\kappa_i}\sum_{t=1}^{s-i}\frac{(-1)^t}{t!}\underset{r_1+\dots+r_t=td+i-s}{\sum_{i\leq r_1,\dots,r_t\leq d-1}}\prod_{m=1}^t\frac{r_m!\kappa_{r_m}}{d!\kappa_d}V_{r_m}(K)\Bigg].$$
It follows that
\begin{align}
&\Bigg|\int_{\mathbbm{R}^d}\dots\int_{\mathbbm{R}^d}\bigg[\sum_{l=1}^{n-k}\frac{(-1)^{l-1}{n-k\choose l}}{n^l}\sum_{s=i}^d\tilde{P}_{s,l,i}(d)V_{s}(K\cap(K+x_2)\cap\dots\cap (K+x_k))-V_i(K\cap(K+x_2)\cap\dots\cap (K+x_k))\bigg]\notag\\
&\cdot\bigg[\sum_{l=1}^{n-k}\frac{(-1)^{l-1}{n-k\choose l}}{n^l}\sum_{s=j}^d\tilde{P}_{s,l,j}(d)V_{s}(K\cap(K+x_2)\cap\dots\cap (K+x_k))-V_j(K\cap(K+x_2)\cap\dots\cap (K+x_k))\bigg]dx_2\dots dx_k\notag \\
&-\sum_{s=i}^d\sum_{r=j}^d\int_{\mathbbm{R}^d}\dots\int_{\mathbbm{R}^d}P_{i,s}(d)P_{j,r}(d)V_{s}(K\cap(K+x_2)\dots\cap (K+x_k))V_{r}(K\cap(K+x_2)\dots\cap (K+x_k))dx_2\dots dx_k\Bigg|\notag\\
\stackrel{(\dagger)}\leq &(d!)^2\bigg[\sum_{l=1}^{\infty}\frac{\overline{V}(K)^{l+1}}{l!}+\sum_{l=1}^{n-k}\frac{{n-k\choose l}}{n^l}\overline{V}(K)^{l+1}+2\overline{V}(K)\bigg]\notag\\
&\cdot\int_{\mathbbm{R}^d}\dots\int_{\mathbbm{R}^d}\bigg[\left(\sum_{l=0}^{\infty}\frac{\overline{V}(K)^{l}}{l!}-\sum_{l=0}^{n-k}\overline{V}(K)^{l}\frac{{n-k\choose l}}{n^l}\right)\overline{V}(K\cap(K+x_2)\cap\dots\cap(K+x_k))\bigg]dx_2\dots dx_k\notag\\
\stackrel{(\ddagger)}\leq &2(d!)^2\overline{V}(K)^{k+1}\left(e^{\overline{V}(K)}+1\right)\left(e^{\overline{V}(K)}-\left(1+\frac{\overline{V}(K)}{n}\right)^{n-k}\right)\notag\\
\leq &2(d!)^2\overline{V}(K)^{k+1}\left(e^{\overline{V}(K)}+1\right)\Bigg[\left(e^{\overline{V}(K)}-\left(1+\frac{\overline{V}(K)}{n}\right)^{n}\right)+\left(\left(1+\frac{\overline{V}(K)}{n}\right)^{n}-\left(1+\frac{\overline{V}(K)}{n}\right)^{n-k}\right)\Bigg]\notag\\
\leq &2(d!)^2\overline{V}(K)^{k+1}\left(e^{\overline{V}(K)}+1\right)\Bigg[\sum_{l=n+1}^\infty\frac{\overline{V}(K)^l}{l!}+\sum_{r=1}^n\left(\frac{\overline{V}(K)^r}{r!}-\frac{\overline{V}(K)^r{n\choose r}}{n^r}\right)\notag\\
&\hspace{7cm}+\sum_{m=1}^k\left(\left(1+\frac{\overline{V}(K)}{n}\right)^{n-m+1}-\left(1+\frac{\overline{V}(K)}{n}\right)^{n-m}\right)\Bigg]\notag\\
\stackrel{(\mathparagraph)}\leq &2(d!)^2\overline{V}(K)^{k+2}\left(e^{\overline{V}(K)}+1\right)\Bigg[\frac{e}{n}+\frac{2e^{2\overline{V}(K)}}{n}+k\frac{e^{\overline{V}(K)}}{n}\Bigg],\label{covariance8}
\end{align}
where inequality $(\dagger)$ follows from (\ref{p_identity}), (\ref{pp_identity}) and Lemma \ref{integral_lemma};  inequality $(\ddagger)$ follows from Lemma \ref{integral_lemma}; and inequality $(\mathparagraph)$ follows from the well-known bound $\sum_{k=B}^{\infty}\frac{A^k}{k!}\leq \left(\frac{Ae}{B}\right)^B$, for $B>A>0$, our assumption $n>\overline{V}(K)e$ and the fact that, for $r,n\in\mathbbm{N}\setminus\{ 0\}$, $r\leq n$,
\begin{align}\label{factorial_bound}
\left|\frac{1}{r!}-\frac{{n\choose r}}{n^r}\right|=\left|\frac{n^r-n(n-1)\dots(n-r+1)}{n^rr!}\right|\leq\frac{n^r-(n-r+1)^r}{n^rr!}\leq\frac{1}{r!}\sum_{l=0}^{r-1}{r\choose l}\frac{r^{r-l}}{n^{r-l}}\leq \frac{1}{(r-1)!n}\sum_{l=0}^{r-1}{r\choose l}\leq\frac{2^r}{n(r-1)!}.
\end{align}

Now, using (\ref{factorial_bound}), (\ref{p_identity}), (\ref{pp_identity}), (\ref{inclusion_exclusion}) and Lemmas \ref{integral_lemma} and \ref{penrose_lemma},
\begin{align}
&\sum_{k=1}^n\left|\frac{{n\choose k}}{n^k}-\frac{1}{k!}\right|\int_{\mathbbm{R}^d}\dots\int_{\mathbbm{R}^d}\sum_{s=i}^d\sum_{r=j}^dP_{i,s}(d)P_{j,r}(d)V_s(K\cap(K+x_2)\cap\dots\cap(K+x_k))\notag\\
&\hspace{6cm}\cdot V_r(K\cap(K+x_2)\cap\dots\cap(K+x_k))dx_2\dots dx_k\notag\\
\leq &\sum_{k=1}^n\frac{2^k}{(k-1)!}(d!)^2\overline{V}(K)^{k+1}\left(e^{\overline{V}(K)}+1\right)^2\frac{1}{n}\notag\\
\leq &2\overline{V}(K)^2\left(e^{\overline{V}(K)}+1\right)^2(d!)^2e^{2\overline{V}(K)}n^{-1}.\label{covariance9}
\end{align}
Similarly,
\begin{align}
&\sum_{k=n+1}^{\infty}\frac{1}{k!}\int_{\mathbbm{R}^d}\dots\int_{\mathbbm{R}^d}\sum_{s=i}^d\sum_{r=j}^dP_{i,s}(d)P_{j,r}(d)V_s(K\cap(K+x_2)\cap\dots\cap(K+x_k))\notag\\
&\hspace{6cm}V_r(K\cap(K+x_2)\cap\dots\cap(K+x_k))dx_2\dots dx_k\notag\\
\leq &\sum_{k=n+1}^{\infty}\frac{1}{k!}(d!)^2\overline{V}(K)^{k+1}\left(e^{\overline{V}(K)}+1\right)^2
\leq \frac{\overline{V}(K)^2e(d!)^2}{n}\left(e^{\overline{V}(K)}+1\right)^2.\label{covariance10}
\end{align}
Using (\ref{covariance7}), (\ref{covariance8}), (\ref{covariance9}), (\ref{covariance10}), we obtain
\begin{align}
&\Bigg|\frac{1}{n}\sum_{k=1}^n{n\choose k}\mathbbm{E}\left[\psi_k^i(X_1,\dots,X_k)\psi_k^j(X_1,\dots,X_k)\right]-\sum_{k=1}^{\infty}\frac{1}{k!}\sum_{s=i}^d\sum_{r=j}^d\int_{\mathbbm{R}^d}\dots\int_{\mathbbm{R}^d}P_{i,s}(d)P_{j,r}(d)\notag\\
&\hspace{3cm}\cdot V_s(K\cap(K+x_2)\cap\dots\cap(K+x_k))V_r(K\cap(K+x_2)\cap\dots\cap(K+x_k))dx_2\dots dx_k\Bigg|\notag\\
\leq &(d!)^2\left(e^{\overline{V}(K)}+1\right)^2\left(e^{2\overline{V}(K)}\left(6\overline{V}(K)^2+2\overline{V}(K)^3\right)+3\overline{V}(K)^2e+6(4R+1)^d\overline{V}(K)^2e^{\overline{V}(K)}\right)n^{-1/d}.\label{covariance13}
\end{align}
Now, using (\ref{integral_bound}), Lemma \ref{integral_lemma}, and the translation invariance of the Lebesgue measure, we obtain
\begin{align}
&\Bigg|\left[\sum_{l=1}^{n-1}\frac{(-1)^l}{n^{l+1}}{n-1\choose l}\int_{E_n}\dots\int_{E_n}V_i((K+x_1)\cap\dots(K+x_{l+1}))dx_1\dots dx_{l+1}-V_i(K)\right]\notag\\
&\cdot\left[\sum_{l=1}^{n-1}\frac{(-1)^l}{n^{l+1}}{n-1\choose l}\int_{E_n}\dots\int_{E_n}V_i((K+x_1)\cap\dots(K+x_{l+1}))dx_1\dots dx_{l+1}-V_i(K)\right]\notag\\
&-\left[\sum_{l=1}^{n-1}\frac{(-1)^l}{n^{l}}{n-1\choose l}\int_{\mathbbm{R}^d}\dots\int_{\mathbbm{R}^d}V_i(K\cap(K+x_2)\cap\dots(K+x_{l+1}))dx_2\dots dx_{l+1}-V_i(K)\right]\notag\\
&\cdot\left[\sum_{l=1}^{n-1}\frac{(-1)^l}{n^{l}}{n-1\choose l}\int_{\mathbbm{R}^d}\dots\int_{\mathbbm{R}^d}V_i(K\cap(K+x_2)\cap\dots(K+x_{l+1}))dx_2\dots dx_{l+1}-V_i(K)\right]\Bigg|\notag\\
\leq &2(d!)\left[\sum_{l=1}^{n-1}\frac{{n-1\choose l}}{n^l}\overline{V}(K)^{l+1}+\overline{V}(K)\right]\left[\sum_{l=1}^{n-1}\frac{{n-1\choose l}}{n^l}\frac{l(d!)(4R+1)^d}{n^{1/d}}\overline{V}(K)^l\right]\notag\\
\leq& 2(d!)^2\overline{V}(K)^2\left(e^{\overline{V}(K)}+1\right)(4R+1)^de^{\overline{V}(K)}n^{-1/d}.\label{covariance11}
\end{align}
Similarly, using (\ref{p_identity}), (\ref{pp_identity}) and Lemma \ref{integral_lemma}, in a manner similar to the calculation (\ref{covariance8}),
\begin{align}
&\Bigg|\left[\sum_{l=1}^{n-1}\frac{(-1)^l}{n^{l+1}}{n-1\choose l}\int_{\mathbbm{R}^d}\dots\int_{\mathbbm{R}^d}V_i(K\cap(K+x_2)\cap\dots(K+x_{l+1}))dx_2\dots dx_{l+1}-V_i(K)\right]\notag\\
&\cdot\left[\sum_{l=1}^{n-1}\frac{(-1)^l}{n^{l}}{n-1\choose l}\int_{\mathbbm{R}^d}\dots\int_{\mathbbm{R}^d}V_i(K\cap(K+x_2)\cap\dots(K+x_{l+1}))dx_2\dots dx_{l+1}-V_i(K)\right]\notag\\
&-\sum_{s=i}^d\sum_{r=j}^dP_{i,s}(d)P_{j,r}(d)V_{s}(K)V_{r}(K)\Bigg|\notag\\
\leq &2(d!)^2\overline{V}(K)^3\left(e^{\overline{V}(K)}+1\right)\frac{e+2e^{2\overline{V}(K)}+e^{\overline{V}(K)}}{n}.\label{covariance12}
\end{align}
From (\ref{covariance11}) and (\ref{covariance12}) we obtain that
\begin{align}
&\Bigg|\mathbbm{E}\left[\psi_1^i(X_1)\right]\mathbbm{E}\left[\psi_1^j(X_1)\right]-\sum_{s=i}^d\sum_{r=j}^dP_{i,s}(d)P_{j,r}(d)V_{s}(K)V_{r}(K)\Bigg|\notag\\
\leq& 2(d!)^2\overline{V}(K)^2\left(e^{\overline{V}(K)}+1\right)\left[(4R+1)^de^{\overline{V}(K)}+\overline{V}(K)\left(e+2e^{2\overline{V}(K)}+e^{\overline{V}(K)}\right)\right]n^{-1/d}\label{covariance14}
\end{align}
and the final bound follows from (\ref{covariance13}) and (\ref{covariance14}) and the fact that $\overline{V}(K)=\sum_{j=0}^d{d\choose j}\kappa_dR^j\leq  3^d(R+1)^d$.
\end{proof}

\subsubsection{Concluding argument}
The proof of Theorem \ref{pos_def_teorem} is finished by applying the discussion of Section \ref{intro_cov}, a direct application of Lemmas \ref{lemma_cov1} and \ref{lemma_cov2} and the fact that $\overline{V}(K)=\sum_{j=0}^d{d\choose j}\kappa_dR^j\leq  3^d(R+1)^d$ (see \cite[Example 1.2]{tropp} and \cite[page 224-227]{san}).\qed

\subsection{Proof of Theorem \ref{theorem_covering}}
 We will say that $Y=(Y_1,\dots,Y_n)$ is a recombination of $X,X',\tilde{X}$ if $Y_i\in\lbrace X_i,X_i',\tilde{X}_i\rbrace$, for $i\in\left[n\right]$. In what follows, $Y,\bar{Y},\hat{Y},\tilde{Y},Y',\bar{Y}',\hat{Y}',\tilde{Y}',Z,\bar{Z},\hat{Z},Z',\bar{Z}',\hat{Z}'$ will denote recombinations of $X,X',\tilde{X}$ and all the suprema in this section will be taken over all such recombinations. First, we prove a useful estimate:

\subsubsection{Auxiliary lemma}
\begin{lemma}\label{lemma_estimates}
For all $m,m_1,m_2,m_3\geq 1$, we have
\begin{align*}
&\mathbbm{E}\left[\|\Delta_1f(X)\|^m\right]
\leq\frac{C_m^{(1)}(d)}{n^{m/2}};\qquad \sup_{(Y,Z)}\mathbbm{E}\left[\overline{V}(Y_1\cap Y_2)\|\Delta_1f(Z)\|^{m}\right]\leq \frac{C_m^{(2)}(d)}{n^{m/2+1}};\\
&\sup_{(Y,Z)}\mathbbm{E}\bigg[\|\Delta_2f(Y)\|^{m_1}\|\tilde{\Delta}_1\Delta_2f(Z)\|^{m_2}\bigg]\leq \frac{C_{m_1,m_2}^{(3)}(d)}{n^{(m_1+m_2)/2+1}};\\ 
&\sup_{(Y,Y',Z,Z')}\mathbbm{E}\bigg[\|\Delta_2f(Y)\|^{m_1}\|\Delta_3f(Y')\|^{m_1}\|\tilde{\Delta}_1\Delta_2f(Z)\|^{m_2}\|\tilde{\Delta}_1\Delta_3f(Z')\|^{m_2}\bigg]\leq\frac{C_{m_1,m_2}^{(4)}(d)}{n^{m_1+m_2+2}};\\
&\sup_{(Y,Y',Z,Z')}\mathbbm{E}\bigg[\overline{V}(Y_1\cap Y_2)\overline{V}(Z_1\cap Z_3)\|\Delta_2f(Y')\|^{m}\|\Delta_3f(Z')\|^{m}\bigg]\leq \frac{C_{m}^{(5)}(d)}{n^{m+2}},\\
&\max\Bigg\{\hspace{-1mm}\sup_{(Y,Y',Z)}\hspace{-2mm}\mathbbm{E}\bigg[\|\tilde{\Delta}_1f(Y)\|^{m_1}\|\Delta_2f(Y')\|^{m_2}\|\tilde{\Delta}_1\Delta_2f(Z)\|^{m_3}\bigg],\hspace{-1mm}\sup_{(Y,Y',Z)}\hspace{-2mm}\mathbbm{E}\bigg[\|\Delta_1f(Y)\|^{m_1}\|\Delta_2f(Y')\|^{m_2}\|\tilde{\Delta}_1\Delta_2f(Z)\|^{m_3}\bigg]\Bigg\}\\
&\hspace{14cm}\leq \frac{C_{m_1,m_2,m_3}^{(6)}(d)}{n^{(m_1+m_2+m_3)/2+1}};\\ 
&\sup_{(Y,Y',Z)}\mathbbm{E}\bigg[\overline{V}(Y_1\cap Y_2)\|\Delta_1f(Y')\|^{m}\|\Delta_2f(Z)\|^{m}\bigg]\leq \frac{C_{2m}^{(2)}(d)}{n^{m+1}},
\end{align*}
for
\begin{align*}
&C_m^{(1)}(d):=(2d+4)^{m}(d!)^m\left(2^{2^d}d^{d/2}12^d(R+1)^{d}\right)^m e^{2^m(2R+1)^d};\\
&C_m^{(2)}(d):=(4d+8)^{m}(d!)^m\left(2^{2^d+2d}d^{d/2}\right)^m\left(3(R+1)\right)^{d(m+2)} e^{2^m(2R+1)^d};\\
&C_{m_1,m_2}^{(3)}(d):=(4d+8)^{m_1+m_2}(d!)^{m_1+m_2}\left(2^{2^d+2d}d^{d/2}\right)^{m_1+m_2} (3(R+1))^{m_1+m_2+1}e^{(2^{2m_1-1}+2^{2m_2-1})(2R+1)^d};\\
&C_{m_1,m_2}^{(4)}(d)\hspace{-1mm}:=\hspace{-1mm}(8d+16)^{2m_1+2m_2}(d!)^{2m_1+2m_2}\hspace{-1mm}\left(2^{2^d}d^{d/2}12^{d}(R+1)^d\right)^{2m_1+2m_2} (3(R+1))^{2d}e^{(2^{4m_1-1}+2^{4m_2-1})(2R+1)^d};\\
&C_m^{(5)}(d):=(8d+16)^{2m}(d!)^{2m}\left(2^{2^d+2d}d^{d/2}\right)^{2m} (3(R+1))^{d(2m+4)}e^{2^{2m}(2R+1)^d};\\
&C_{m_1,m_2,m_3}^{(6)}(d)\hspace{-1mm}:=\hspace{-1mm}(4d+8)^{m_1+m_2+m_3}(d!)^{m_1+m_2+m_3}\hspace{-1mm}\left(2^{2^d}d^{d/2}12^{d}(R+1)^d\right)^{m_1+m_2+m_3}\hspace{-2mm}(3(R+1))^{d} e^{\frac{1}{3}(2^{3m_1}+2^{3m_2}+2^{3m_3})(2R+1)^d}.
\end{align*}
\end{lemma}
\begin{proof}
Note that each $V_l$ is an additive functional on the ring $\mathcal{R}$ of finite unions of convex bodies. Moreover, for $l=0,\dots,d$, $V_l$ is monotone and nonnegative on the (smaller) class of convex bodies.
It follows that, for all $i,j,k=1,\dots,n$ which are pairwise distinct, and for any $Z$ which is a recombination of $\{X,X',\tilde{X}\}$,
\begin{align}
&\left|V_l\left(\left(F_n^{(j)}(Z)\cup Z_j\right)\right)-V_l\left(\left(F_n^{(j)}(Z)\cup X_j'\right)\right)\right|\notag\\
=&\left|V_l\left(F_n^{(j)}(Z)\right)+V_l(Z_j)-V_l\left(F_n^{(j)}(Z)\cap Z_j\right)-V_l\left(F_n^{(j)}(Z)\right)-V_l(X_j')+V_l\left(F_n^{(j)}(Z)\cap X_j'\right)\right|\notag\\
\leq & V_l(Z_j)+\left|V_l(X_j'\cap F_n^{(j)}(Z))\right|+V_l(X_j')+\left|V_l(Z_j\cap F_n^{(j)}(Z))\right|\label{add_one1}\\
\leq &4V_l(K)\hspace{-1mm}+\hspace{-1mm}\left|V_l(X_j'\cap F_n^{(i,j)}(Z))\right|+\left|V_l(X_j'\cap Z_i\cap F_n^{(i,j)}(Z))\right|+\left|V_l(Z_j\cap F_n^{(i,j)}(Z))\right|+\left|V_l(Z_j\cap Z_i\cap F_n^{(i,j)}(Z))\right|\label{add_one2}\\
\leq &8V_l(K)+\left|V_l(X_j'\cap F_n^{(i,j,k)}(Z))\right|+\left|V_l(X_j'\cap Z_i\cap F_n^{(i,j,k)}(Z))\right|+\left|V_l(Z_j\cap F_n^{(i,j,k)}(Z))\right|+\left|V_l(Z_j\cap Z_i\cap F_n^{(i,j,k)}(Z))\right|\notag\\
&+\left|V_l(X_j'\cap Z_k\cap F_n^{(i,j,k)}(Z))\right|+\left|V_l(X_j'\cap Z_i\cap Z_k\cap F_n^{(i,j,k)}(Z))\right|+\left|V_l(Z_j\cap Z_k\cap F_n^{(i,j,k)}(Z))\right|\notag\\
&+\left|V_l(Z_j\cap Z_i\cap Z_k\cap F_n^{(i,j,k)}(Z))\right|.\label{add_one3}
\end{align}
Moreover, for any $i,j,k=1,\dots,n$ which are pairwise distinct, a similar argument yields
\begin{align}
&\left|V_l\left(F_n^{(i,j)}(Z)\cup Z_i\cup Z_j\right)-V_l\left(F_n^{(i,j)}(Z)\cup Z_i\cup X_j'\right)-V_l\left(F_n^{(i,j)}(Z)\cup \tilde{X}_i\cup Z_j\right)+V_l\left(F_n^{(i,j)}\cup \tilde{X}_i\cup X_j'\right)\right|\notag\\
\leq &\left|V_l(Z_i\cap Z_j\cap F_n^{(i,j)}(Z))\right|+\left|V_l(Z_i\cap X_j'\cap F_n^{(i,j)}(Z))\right|+\left|V_l(\tilde{X}_i\cap Z_j\cap F_n^{(i,j)}(Z))\right|+\left|V_l(\tilde{X}_i\cap X_j'\cap F_n^{(i,j)}(Z))\right|\notag\\
&+\left|V_l(Z_i\cap Z_j)\right|+\left|V_l(Z_i\cap X_j')\right|+\left|V_l(\tilde{X}_i\cap Z_j)\right|+\left|V_l(\tilde{X}_i\cap X_j')\right|\label{add_one 4}\\
\leq &\left|V_l(Z_i\cap Z_j\cap F_n^{(i,j,k)}(Z))\right|+\left|V_l(Z_i\cap X_j'\cap  F_n^{(i,j,k)}(Z))\right|+\left|V_l(\tilde{X}_i\cap Z_j\cap F_n^{(i,j,k)}(Z))\right|+\left|V_l(\tilde{X}_i\cap X_j'\cap F_n^{(i,j,k)}(Z))\right|\notag\\
&+\left|V_l(Z_i\cap Z_j\cap Z_k\cap F_n^{(i,j,k)}(Z))\right|+\left|V_l(Z_i\cap X_j'\cap Z_k\cap F_n^{(i,j,k)}(Z))\right|+\left|V_l(\tilde{X}_i\cap Z_j\cap Z_k\cap F_n^{(i,j,k)}(Z))\right|\notag\\
&+\hspace{-0.5mm}\left|V_l(\tilde{X}_i\cap X_j'\cap Z_k\cap F_n^{(i,j,k)}(Z))\right|\hspace{-0.5mm}+\hspace{-0.5mm}2\left|V_l(Z_i\cap Z_j)\right|\hspace{-0.5mm}+\hspace{-0.5mm}2\left|V_l(Z_i\cap X_j')\right|+2\left|V_l(\tilde{X}_i\cap Z_j)\right|+2\left|V_l(\tilde{X}_i\cap X_j')\right|.\label{add_one5}
\end{align}
Therefore, for $m\geq 1$,
\begin{align}
\mathbbm{E}\|\Delta_1f(X)\|^m
\stackrel{(\ref{add_one1})}\leq & \frac{1}{n^{m/2}}\mathbbm{E}\left(2d!\overline{V}(K)+\sum_{l=0}^d\left(\left|V_l\left(X_1\cap F_n^{(1)}\right)\right|+\left|V_l\left(X_1'\cap F_n^{(1)}\right)\right|\right)\right)^{m}\notag\\
\leq &\frac{(2d+4)^{m-1}}{n^{m/2}}\left[2(d!)^{m}\overline{V}(K)^m+2\sum_{l=0}^d\mathbbm{E}\left[\left|V_l\left(X_1\cap F_n^{(1)})\right)\right|^m\right]\right]\notag\\
\stackrel{\text{Lemma }\ref{penrose_lemma}}\leq&\frac{(2d+4)^{m}(d!)^m\left(2^{2^d+2d}d^{d/2}\right)^m e^{2^m(2R+1)^d}\overline{V}(K)^{m}}{n^{m/2}}.\label{estimate1}
\end{align}
Similarly, for $m\geq 1$,
\begin{align}
&\sup_{(Y,Z)}\mathbbm{E}\left[\overline{V}(Y_1\cap Y_2)\|\Delta_1f(Z)\|^{m}\right]
\notag\\
\stackrel{(\ref{add_one2})}\leq&\frac{(4d+8)^{m-1}}{n^{m/2}}\hspace{-3mm}\sup_{(Y,Y',Z,Z')}\hspace{-3mm}\mathbbm{E}\Bigg\{\overline{V}(Y_1\cap Y_2)\bigg[4(d!)^m\overline{V}(K)^m+2\sum_{l=0}^d\bigg(\hspace{-1mm}\left|V_l(Y_1'\cap F_n^{(1,2)}(Z))\right|^{m}\hspace{-1mm}+\left|V_l(Z'_1\cap Z'_2\cap F_n^{(1,2)}(Z))\right|^{m}\hspace{-1mm}\bigg)\bigg]\Bigg\}\notag\\
\leq&\frac{(4d+8)^{m}(d!)^m\left(2^{2^d+2d}d^{d/2}\right)^m e^{2^m(2R+1)^d}\overline{V}(K)^{m+2}}{n^{m/2+1}},\label{estimate2}
\end{align}
where the last inequality follows from Lemmas \ref{integral_lemma} and \ref{penrose_lemma} and, for $m_1,m_2\geq 1$,
\begin{align}
&\sup_{(Y,Z)}\mathbbm{E}\bigg[\|\Delta_2f(Y)\|^{m_1}\|\tilde{\Delta}_1\Delta_2f(Z)\|^{m_2}\bigg]\notag\\
\underset{(\ref{add_one2})}{\stackrel{(\ref{add_one 4})}\leq} &\frac{(4d+8)^{m_1+m_2-2}}{n^{(m_1+m_2)/2}}\sup_{(Y,Y',Z,Z')}\mathbbm{E}\Bigg\{\bigg[4(d!)^{m_1}\overline{V}(K)^{m_1}+2\sum_{l=0}^d\bigg(\left|V_l(Y_2\cap F_n^{(1,2)}(Y))\right|^{m_1}+\left|V_l(Y'_1\cap Y'_2\cap F_n^{(1,2)}(Y))\right|^{m_1}\bigg)\bigg]\notag\\
&\hspace{7cm}\cdot\bigg[4(d!)^{m_2}\overline{V}(Z_1\cap Z_2)^{m_2}+4\sum_{l=0}^d\left|V_l(Z'_1\cap Z'_2\cap F_n^{(1,2)}(Z))\right|^{m_2}\bigg]\Bigg\}\notag\\
\leq& \frac{(4d+8)^{m_1+m_2}(d!)^{m_1+m_2}\left(2^{2^d}d^{d/2}4^{d}\right)^{m_1+m_2} e^{(2^{2m_1-1}+2^{2m_2-1})(2R+1)^d}\overline{V}(K)^{m_1+m_2+1}}{n^{(m_1+m_2)/2+1}}.\label{estimate3}
\end{align}
where the last inequality follows by the Cauchy-Schwarz inequality and Lemmas \ref{integral_lemma} and \ref{penrose_lemma}.\\
Now, let $(\dagger\dagger)=(Y,\bar{Y},\hat{Y},\tilde{Y},Y',\bar{Y}',\hat{Y}',\tilde{Y}',Z,\bar{Z},\hat{Z},Z',\bar{Z}',\hat{Z}')$. In the same manner as above, we obtain
\begin{align}
&\sup_{(Y,Y',Z,Z')}\mathbbm{E}\bigg[\|\Delta_2f(Y)\|^{m_1}\|\Delta_3f(Y')\|^{m_1}\|\tilde{\Delta}_1\Delta_2f(Z)\|^{m_2}\|\tilde{\Delta}_1\Delta_3f(Z')\|^{m_2}\bigg]\notag\\
\leq&  \frac{(8d+16)^{2m_1+2m_2-2}}{n^{m_1+m_2}}\sup_{(\dagger\dagger)}\mathbbm{E}\Bigg\{\bigg[8(d!)^{m_1}\overline{V}(K)^{m_1}+2\sum_{l=0}^d\bigg(\left|V_l(Y_2\cap F_n^{(1,2,3)}(Y))\right|^{m_1}\notag\\
&+\left|V_l(\bar{Y}_1\cap \bar{Y}_2\cap F_n^{(1,2,3)}(Y))\right|^{m_1}+\left|V_l(\hat{Y}_2\cap \hat{Y}_3\cap F_n^{(1,2,3)}(Y))\right|^{m_1}+\left|V_l(\tilde{Y}_1\cap \tilde{Y}_2\cap \tilde{Y}_3\cap F_n^{(1,2,3)}(Y))\right|^{m_1}\bigg)\bigg]\notag\\
&\cdot\bigg[8(d!)^{m_1}\overline{V}(K)^{m_1}+2\sum_{l=0}^d\bigg(\left|V_l(Y'_3\cap F_n^{(1,2,3)}(Y'))\right|^{m_1}+\left|V_l(\bar{Y}'_1\cap \bar{Y}'_3\cap F_n^{(1,2,3)}(Y'))\right|^{m_1}\notag\\
&+\left|V_l(\hat{Y}'_2\cap \hat{Y}'_3\cap F_n^{(1,2,3)}(Y'))\right|^{m_1}+\left|V_l(\tilde{Y}'_1\cap \tilde{Y}'_2\cap \tilde{Y}'_3\cap F_n^{(1,2,3)}(Y'))\right|^{m_1}\bigg)\bigg]\notag\\
&\cdot\Bigg[8(d!)^{m_2}\overline{V}(Z_1\cap Z_2)^{m_2}+4\sum_{l=0}^d\bigg(\left|V_l(\bar{Z}_1\cap \bar{Z}_2\cap F_n^{(1,2,3)}(Z))\right|^{m_2}+\left|V_l(\hat{Z}_1\cap \hat{Z}_2\cap \hat{Z}_3\cap F_n^{(1,2,3)}(Z))\right|^{m_2}\bigg)\Bigg]\notag\\
&\cdot\Bigg[8(d!)^{m_2}\overline{V}(Z_1'\cap Z_3')^{m_2}+4\sum_{l=0}^d\bigg(\left|V_l(\bar{Z}_1'\cap \bar{Z}_3'\cap F_n^{(1,2,3)}(Z'))\right|^{m_2}+\left|V_l(\hat{Z}_1'\cap \hat{Z}_2'\cap \hat{Z}_3'\cap F_n^{(1,2,3)}(Z'))\right|^{m_2}\bigg)\Bigg]\Bigg\}\notag\\
\leq &\frac{(8d+16)^{2m_1+2m_2}(d!)^{2m_1+2m_2}\left(2^{2^d}d^{d/2}4^{d}\right)^{2m_1+2m_2} e^{(2^{4m_1-1}+2^{4m_2-1})(2R+1)^d}\overline{V}(K)^{2m_1+2m_2+2}}{n^{m_1+m_2+2}}.\label{estimate4}
\end{align}
where the first inequality follows from (\ref{add_one3}), (\ref{add_one5}) and the last one follows from Lemmas \ref{integral_lemma} and \ref{penrose_lemma} and H\"older's inequality. Similarly,
\begin{align}
&\sup_{(Y,Y',Z,Z')}\mathbbm{E}\bigg\{\overline{V}(Y_1\cap Y_2)\overline{V}(Z_1\cap Z_3)\|\Delta_2f(Y')\|^{m}\|\Delta_3f(Z')\|^{m}\bigg\}\\
\leq& \frac{(8d+16)^{2m}(d!)^{2m}\left(2^{2^d}d^{d/2}4^{d}\right)^{2m} e^{2^{2m}(2R+1)^d}\overline{V}(K)^{2m+4}}{n^{m+2}}.\label{estimate5}
\end{align}
Furthermore, for $m_1,m_2,m_3\geq 1$,
\begin{align}
&\max\Bigg\{\hspace{-1mm}\sup_{(Y,Y',Z)}\hspace{-2mm}\mathbbm{E}\bigg[\|\tilde{\Delta}_1f(Y)\|^{m_1}\|\Delta_2f(Y')\|^{m_2}\|\tilde{\Delta}_1\Delta_2f(Z)\|^{m_3}\bigg],\hspace{-1mm}\sup_{(Y,Y',Z)}\hspace{-2mm}\mathbbm{E}\bigg[\|\Delta_1f(Y)\|^{m_1}\|\Delta_2f(Y')\|^{m_2}\|\tilde{\Delta}_1\Delta_2f(Z)\|^{m_3}\bigg]\Bigg\}\notag\\
\leq &\frac{(4d+8)^{m_1+m_2+m_3-3}}{n^{(m_1+m_2+m_3)/2}}\sup_{(Y,\tilde{Y},Y',\tilde{Y}',Z,Z')}\mathbbm{E}\Bigg\{\bigg[4(d!)^{m_1}\overline{V}(K)^{m_1}+2\sum_{l=0}^d\bigg(\left|V_l(Y_1\cap F_n^{(1,2)}(Y))\right|^{m_1}\notag\\
&+\left|V_l(\tilde{Y}_1\cap \tilde{Y}_2\cap F_n^{(1,2)}(Y))\right|^{m_1}\bigg)\bigg]\bigg[4(d!)^{m_2}\overline{V}(K)^{m_2}+2\sum_{l=0}^d\bigg(\left|V_l(Y_2'\cap F_n^{(1,2)}(Y'))\right|^{m_2}\notag\\
&+\left|\tilde{Y}_1'\cap\tilde{Y}_2'\cap F_n^{(1,2)}(Y'))\right|^{m_2}\bigg)\bigg]\bigg[4\overline{V}(Z_1\cap Z_2)^{m_3}+4\sum_{l=0}^d\left|V_l(Z'_1\cap Z'_2\cap F_n^{(1,2)}(Z))\right|^{m_3}\bigg]\Bigg\}\notag\\
\leq&                                                                                                                                                                                                                                                                                                                                                                    \frac{(4d+8)^{m_1+m_2+m_3}(d!)^{m_1+m_2+m_3}\left(2^{2^d}d^{d/2}4^{d}\right)^{m_1+m_2+m_3} e^{\frac{1}{3}(2^{3m_1}+2^{3m_2}+2^{3m_3})(2R+1)^d}\overline{V}(K)^{m_1+m_2+m_3+1}}{n^{(m_1+m_2+m_3)/2+1}},\label{estimate6}
\end{align}
where the first inequality follows from (\ref{add_one2}), (\ref{add_one 4}) and the second one from H\"older's inequality and Lemmas \ref{integral_lemma} and \ref{penrose_lemma}. Similarly,
\begin{align}
\sup_{(Y,Y',Z)}\mathbbm{E}\bigg\{\overline{V}(Y_1\cap Y_2)\|\Delta_1f(Y')\|^{m}\|\Delta_2f(Z)\|^{m}\bigg\}
\leq &\frac{(4d+8)^{2m}(d!)^{2m}\left(2^{2^d}d^{d/2}4^{d}\right)^{2m} e^{2^{2m}(2R+1)^d}\overline{V}(K)^{2m+2}}{n^{m+1}}.\label{estimate7}
\end{align}
The result now follows by (\ref{estimate1}) - (\ref{estimate7}) and the fact that $\overline{V}(K)=\sum_{j=0}^d{d\choose j}\kappa_dR^j\leq  3^d(R+1)^d$.
\end{proof}

\subsubsection{Concluding argument}
Using Lemma \ref{lemma_estimates}, and adopting the notation of Theorem \ref{main3}  we have that
\begin{align*}
\gamma_1\leq\frac{C_3^{(1)}(d)}{n^{1/2}}&\qquad\text{ and }\qquad \gamma_2\leq\frac{\sqrt{C_4^{(1)}(d)}}{n^{1/2}}
\end{align*}
and so
\begin{align}
\max(\gamma_1,\gamma_2)\leq 8(d+2)^3(d!)^3\left(2^{2^d+2d}d^{d/2}\right)^3(3(R+1))^{3d}e^{8(2R+1)^d}n^{-1/2}.\label{gamma_1_bound}
\end{align}
Now, note that, by (\ref{add_one 4}), we have that, for $i\neq j$,
\begin{align}
\mathbbm{1}_{[\tilde{\Delta}_i\Delta_jf(X)\neq 0]}\leq& \mathbbm{1}_{[ X_j\cap X_i\neq\emptyset]}+\mathbbm{1}_{[ X_j'\cap X_i\neq\emptyset]}+\mathbbm{1}_{[ X_j'\cap \tilde{X}_i\neq\emptyset]}+\mathbbm{1}_{[ X_j\cap \tilde{X}_i\neq\emptyset]}\nonumber\\
\leq &d!\left[\overline{V}(X_i\cap X_j)+\overline{V}(X_i\cap X_j')+\overline{V}(\tilde{X}_i\cap X_j)+\overline{V}(\tilde{X}_i\cap X_j')\right].\label{indic_bound}
\end{align}

Now, we look at $\gamma_3$ and $\gamma_4$ of Theorem \ref{main3}. For $A\subsetneq [n]$, let $k_{n,A}:=\frac{1}{{n\choose |A|}(n-|A|)}$. We have that, for $p=1,2$,
\begin{align}
\gamma_{p+2}^{p+2}=&\sum_{i=1}^n\sum_{A_1,A_2\subsetneq [n]}k_{n,A_1}k_{n,A_2}\sum_{j\not\in A_1}\sum_{k\not\in A_2}\Bigg[\frac{3}{2}\beta_{n,A_1,A_2}^{(1,p)}(i,j,k)\notag\\
&\hspace{6cm}+\left(\frac{9}{2}+\frac{9}{2p}\right)\bigg(\beta_{n,A_1,A_2}^{(2,p)}(i,j,k)+\beta_{n,A_1,A_2}^{(3,p)}(i,j,k)+\beta_{n,A_1,A_2}^{(4,p)}(i,j,k)\bigg)\Bigg].\label{beta_bound0}
\end{align}
 where, for $A_1,A_2\subsetneq[n]$, $i,j,k\in[n]$ such that $j\not\in A_1$, $k\not\in A_2$, and $p=1,2$,
\begin{align*}
\beta_{n,A_1,A_2}^{(1,p)}:=&\mathbbm{E}\Bigg\{\mathbbm{1}_{[\tilde{\Delta}_i\Delta_jf(X)\neq 0]}\mathbbm{1}_{[\tilde{\Delta}_i\Delta_kf(X)\neq 0]}\sqrt{\|\Delta_jf(X)\|^p+\|\tilde{\Delta}_i\Delta_jf(X)\|^p}\\
&\hspace{1cm}\cdot\sqrt{\|\Delta_kf(X)\|^p+\|\tilde{\Delta}_i\Delta_kf(X)\|^p}\left\|\Delta_{j}f(X^{A_1})\right\|\left\|\Delta_{k}f(X^{A_2})\right\|\left\|\Delta_{j}f(X)\right\|\left\|\Delta_{k}f(X)\right\|\Bigg\}\\
\beta_{n,A_1,A_2}^{(2,p)}:=&\mathbbm{E}\Bigg\{\sqrt{\|\Delta_jf(X)\|^p+\|\tilde{\Delta}_i\Delta_jf(X)\|^p}\\
&\hspace{1cm}\cdot\sqrt{\|\Delta_kf(X)\|^p+\|\tilde{\Delta}_i\Delta_kf(X)\|^p}\left\|\tilde{\Delta}_i\Delta_{j}f(X^{A_1})\right\|\left\|\Delta_{j}f(X)\right\|\left\|\tilde{\Delta}_i\Delta_{k}f(X^{A_2})\right\|\left\|\Delta_{k}f(X)\right\|\Bigg\}\\
\beta_{n,A_1,A_2}^{(3,p)}:=&\mathbbm{E}\Bigg\{\sqrt{\|\Delta_jf(X)\|^p+\|\tilde{\Delta}_i\Delta_jf(X)\|^p}\\
&\hspace{1cm}\cdot\sqrt{\|\Delta_kf(X)\|^p+\|\tilde{\Delta}_i\Delta_kf(X)\|^p}\left\|\tilde{\Delta}_i\Delta_{j}f(X)\right\|\left\|\Delta_{j}f(X^{A_1})\right\|\left\|\tilde{\Delta}_i\Delta_{k}f(X)\right\|\left\|\Delta_{k}f(X^{A_2})\right\|\Bigg\}\\
\beta_{n,A_1,A_2}^{(4,p)}:=&\mathbbm{E}\Bigg\{\sqrt{\|\Delta_jf(X)\|^p+\|\tilde{\Delta}_i\Delta_jf(X)\|^p}\\
&\hspace{1cm}\cdot\sqrt{\|\Delta_kf(X)\|^p+\|\tilde{\Delta}_i\Delta_kf(X)\|^p}\left\|\tilde{\Delta}_i\Delta_{j}f(X^{A_1})\right\|\left\|\tilde{\Delta}_i\Delta_{j}f(X)\right\|\left\|\tilde{\Delta}_i\Delta_{k}f(X^{A_2})\right\|\left\|\tilde{\Delta}_i\Delta_{k}f(X)\right\|\Bigg\}.
\end{align*}
Now, suppose that $j=k=i\not\in A_1\cup A_2$. Then, using H\"older's inequality repeatedly,
\begin{align*}
\beta_{n,A_1,A_2}^{(1,p)}(i,j,k)\leq\mathbbm{E}\bigg\{\left(\|\Delta_jf(X)\|^{p}+\|\tilde{\Delta}_jf(X)\|^{p}\right)\left\|\Delta_{j}f(X^{A_1})\right\|\left\|\Delta_{j}f(X^{A_2})\right\|\|\Delta_jf(X)\|^2\bigg\}\leq &2\mathbbm{E}\left[\|\Delta_jf(X)\|^{p+4}\right]\\
\beta_{n,A_1,A_2}^{(2,p)}(i,j,k)\leq \mathbbm{E}\bigg\{\left(\|\Delta_jf(X)\|^{p}+\|\tilde{\Delta}_jf(X)\|^{p}\right)\left\|\tilde{\Delta}_jf(X^{A_1})\right\|\left\|\tilde{\Delta}_jf(X^{A_2})\right\|\left\|\Delta_jf(X)\right\|^2\bigg\}
\leq &2\mathbbm{E}\left[\|\Delta_1f(X)\|^{p+4}\right]\\
\beta_{n,A_1,A_2}^{(3,p)}(i,j,k)\leq  \mathbbm{E}\bigg\{\left(\|\Delta_jf(X)\|^{p}+\|\tilde{\Delta}_jf(X)\|^{p}\right)\left\|\Delta_jf(X^{A_1})\right\|\left\|\Delta_jf(X^{A_2})\right\|\left\|\tilde{\Delta}_jf(X)\right\|^2\bigg\}
\leq &2\mathbbm{E}\left[\|\Delta_1f(X)\|^{p+4}\right]\\
\beta_{n,A_1,A_2}^{(4,p)}(i,j,k)\leq \mathbbm{E}\bigg\{\left(\|\Delta_jf(X)\|^{p}+\|\tilde{\Delta}_jf(X)\|^{p}\right)\left\|\tilde{\Delta}_jf(X^{A_1})\right\|\left\|\tilde{\Delta}_jf(X^{A_2})\right\|\left\|\tilde{\Delta}_jf(X)\right\|^2\bigg\}
\leq &2\mathbbm{E}\left[\|\Delta_1f(X)\|^{p+4}\right]
\end{align*}
Therefore, applying Lemma \ref{lemma_estimates} and using the notation thereof, we obtain that
\begin{align}
\beta^{(l,p)}_{n,A_1,A_2}(i,j,k)\leq &\frac{2C_{p+4}^{(1)}(d)}{n^{p/2+2}}\quad\text{for all }l=1,\dots,4\text{ and }j=k=i\not\in A_1\cup A_2.\label{beta_bound1}
\end{align}
In what follows we shall use the following notation:
$$\tilde{S}_i(X):=\left(X_0,\dots,X_{i-1},\tilde{X}_i,X_{i+1},\dots,X_n\right).$$

Now, let $ i=k\neq j$ and $j\not\in A_1,k\not\in A_2$. Then, using (\ref{indic_bound}) and H\"older's inequality,
\begin{align}
\beta_{n,A_1,A_2}^{(1,p)}(i,j,k)
\leq &d!\mathbbm{E}\Bigg\{\left[\overline{V}(X_i\cap X_j)+\overline{V}(X_i\cap X_j')+\overline{V}(\tilde{X}_i\cap X_j)+\overline{V}(\tilde{X}_i\cap X_j')\right]\left(\|\Delta_if(X)\|^{p/2}+\|\tilde{\Delta}_if(X)\|^{p/2}\right)\nonumber\\
&\hspace{-1cm}\cdot\left((2^{p/2-1}+1)\|\Delta_jf(X)\|^{p/2}+2^{p/2-1}\|\Delta_jf(S_i(\tilde{X}))\|^{p/2}\right)\left\|\Delta_{j}f(X^{A_1})\right\|\left\|\Delta_{i}f(X^{A_2})\right\|\left\|\Delta_{j}f(X)\right\|\left\|\Delta_{i}f(X)\right\|\Bigg\}\nonumber\\
\leq & 8\left(2^{p/2}+1\right)d!\sup_{(Y,Z)}\mathbbm{E}\Big[\overline{V}(Y_1\cap Y_2)\|\Delta_1f(Z)\|^{p+4}\Big]\label{beta1}\\
\beta_{n,A_1,A_2}^{(2,p)}(i,j,k)\leq &\mathbbm{E}\Bigg\{\left(\|\Delta_if(X)\|^{p/2}+\|\tilde{\Delta}_if(X)\|^{p/2}\right)\left((2^{p/2-1}+1)\|\Delta_jf(X)\|^{p/2}+2^{p/2-1}\|\Delta_jf(S_i(\tilde{X}))\|^{p/2}\right)\nonumber\\
&\hspace{6cm}\cdot\left\|\tilde{\Delta}_if(X^{A_1})\right\|\left\|\Delta_{i}f(X)\right\|\left\|\tilde{\Delta}_i\Delta_{j}f(X^{A_2})\right\|\left\|\Delta_{j}f(X)\right\|\Bigg\}\nonumber\\
\leq &(2^{p/2}+1)\sup_{(Y,Y',Y'',Z)}\mathbbm{E}\Bigg\{\|\Delta_1f(Y)\|^{p/2+1}\|\tilde{\Delta}_1f(Y')\|\|\Delta_2f(Y'')\|^{p/2+1}\|\tilde{\Delta}_1\Delta_2f(Z)\|\Bigg\}\nonumber\\
&+(2^{p/2}+1)\sup_{(Y,Y',Y'',Z)}\mathbbm{E}\Bigg\{\|\Delta_1f(Y)\|\|\tilde{\Delta}_1f(Y')\|^{p/2+1}\|\Delta_2f(Y'')\|^{p/2+1}\|\tilde{\Delta}_1\Delta_2f(Z)\|\Bigg\}\notag\\
\leq &(2^{p/2}+1)\sup_{(Y,Y',Y'',Z)}\sqrt{\mathbbm{E}\Bigg\{\|\tilde{\Delta}_1f(Y)\|^{2}\|\Delta_2f(Y'')\|^{p/2+1}\|\tilde{\Delta}_1\Delta_2f(Z)\|\Bigg\}}\notag\\
&\hspace{5cm}\cdot\sqrt{\mathbbm{E}\Bigg\{\|\Delta_1f(Y')\|^{p+2}\|\Delta_2f(Y'')\|^{p/2+1}\|\tilde{\Delta}_1\Delta_2f(Z)\|\Bigg\}}\notag\\
&+(2^{p/2}+1)\sup_{(Y,Y',Y'',Z)}\sqrt{\mathbbm{E}\Bigg\{\|\Delta_1f(Y)\|^{2}\|\Delta_2f(Y'')\|^{p/2+1}\|\tilde{\Delta}_1\Delta_2f(Z)\|\Bigg\}}\notag\\
&\hspace{5cm}\cdot\sqrt{\mathbbm{E}\Bigg\{\|\tilde{\Delta}_1f(Y')\|^{p+2}\|\Delta_2f(Y'')\|^{p/2+1}\|\tilde{\Delta}_1\Delta_2f(Z)\|\Bigg\}}\label{beta2}\\
\beta_{n,A_1,A_2}^{(3,p)}(i,j,k)\leq &(2^{p/2}+1)\sup_{(Y,Y',Y'',Z)}\sqrt{\mathbbm{E}\Bigg\{\|\tilde{\Delta}_1f(Y)\|^{2}\|\Delta_2f(Y'')\|^{p/2+1}\|\tilde{\Delta}_1\Delta_2f(Z)\|\Bigg\}}\notag\\
&\hspace{5cm}\cdot\sqrt{\mathbbm{E}\Bigg\{\|\Delta_1f(Y)\|^{p+2}\|\Delta_2f(Y'')\|^{p/2+1}\|\tilde{\Delta}_1\Delta_2f(Z)\|\Bigg\}}\notag\\
&+(2^{p/2}+1)\sup_{(Y,Y',Y'',Z)}\sqrt{\mathbbm{E}\Bigg\{\|\Delta_1f(Y)\|^{2}\|\Delta_2f(Y'')\|^{p/2+1}\|\tilde{\Delta}_1\Delta_2f(Z)\|\Bigg\}}\notag\\
&\hspace{5cm}\cdot\sqrt{\mathbbm{E}\Bigg\{\|\tilde{\Delta}_1f(Y)\|^{p+2}\|\Delta_2f(Y'')\|^{p/2+1}\|\tilde{\Delta}_1\Delta_2f(Z)\|\Bigg\}}\label{beta3}\\
\beta_{n,A_1,A_2}^{(4,p)}(i,j,k)\leq&\mathbbm{E}\Bigg\{\left(\|\Delta_if(X)\|^{p/2}+\|\tilde{\Delta}_if(X)\|^{p/2}\right)\left((2^{p/2-1}+1)\|\Delta_jf(X)\|^{p/2}+2^{p/2-1}\|\Delta_jf(S_i(\tilde{X}))\|^{p/2}\right)\nonumber\\
&\hspace{5cm}\left\|\tilde{\Delta}_i\Delta_{j}f(X^{A_1})\right\|\left\|\tilde{\Delta}_i\Delta_{j}f(X)\right\|\left\|\tilde{\Delta}_i\Delta_{j}f(X^{A_2})\right\|\left\|\tilde{\Delta}_i\Delta_{j}f(X)\right\|\Bigg\}\nonumber\\
\leq&(2^{p/2}+1)\sup_{(Y,Y',Z)}\mathbbm{E}\Bigg\{\left(\|\Delta_1f(Y)\|^{p/2}+\|\tilde{\Delta}_1f(Y)\|^{p/2}\right)\|\Delta_2f(Y')\|^{p/2}\left\|\tilde{\Delta}_1\Delta_{2}f(Z)\right\|^{4}\Bigg\}.\label{beta4}
\end{align}
The same argument applies and the same bounds (\ref{beta1})-(\ref{beta4}) hold if $i=j\neq k$ and $j\not\in A_1,k\not\in A_2$. Therefore, applying Lemma \ref{lemma_estimates} and using the notation thereof, for $(i=j\neq k\text{ or }i=k\neq j),j\not\in A_1,k\not\in A_2$,
\begin{align}
&\beta_{n,A_1,A_2}^{(1,p)}(i,j,k)\leq 8\left(2^{p/2}+1\right)(d!)\frac{C_{p+4}^{(2)}(d)}{n^{p/2+3}};\quad \beta_{n,A_1,A_2}^{(4,p)}(i,j,k)\leq 2\left(2^{p/2}+1\right)\frac{C_{p/2,p/2,4}^{(6)}(d)}{n^{p/2+3}};\notag\\
&\beta_{n,A_1,A_2}^{(2,p)}(i,j,k)+\beta_{n,A_1,A_2}^{(3,p)}(i,j,k)\leq 4\left(2^{p/2}+1\right)\frac{\sqrt{C_{2,p/2+1,1}^{(6)}(d)C_{p+2,p/2+1,1}^{(6)}(d)}}{n^{p/2+3}}.\label{beta_bound2}
\end{align}
Now, if $i\neq k=j$, $j\not\in A_1\cup A_2$, we apply (\ref{indic_bound}) and H\"older's inequality to obtain
\begin{align*}
\beta_{n,A_1,A_2}^{(1,p)}(i,j,k)\leq &d!\mathbbm{E}\Bigg\{\left[\overline{V}(X_i\cap X_j)+\overline{V}(X_i\cap X_j')+\overline{V}(\tilde{X}_i\cap X_j)+\overline{V}(\tilde{X}_i\cap X_j')\right]\\
&\cdot\left((2^{p-1}+1)\|\Delta_jf(X)\|^{p}+2^{p-1}\|\Delta_jf(S_i(\tilde{X}))\|^{p}\right)\left\|\Delta_{j}f(X^{A_1})\right\|\left\|\Delta_{j}f(X^{A_2})\right\|\left\|\Delta_{j}f(X)\right\|^2\Bigg\}\\
\leq &4(2^p+1)d!\sup_{(Y,Z)}\mathbbm{E}\Big\{\overline{V}(Y_1\cap Y'_2)\|\Delta_1f(Z)\|^{p+4}\Big\}\\
\beta_{n,A_1,A_2}^{(2,p)}(i,j,k)\leq&\mathbbm{E}\Bigg\{\left(\|\tilde{\Delta}_i\Delta_jf(X)\|^p+\left\|\Delta_{j}f(X)\right\|^p\right)\left\|\tilde{\Delta}_i\Delta_{j}f(X^{A_1})\right\|\left\|\Delta_{j}f(X)\right\|^2\left\|\tilde{\Delta}_i\Delta_{j}f(X^{A_2})\right\|\Bigg\}\\
\leq &\sup_{(Y,Y')}\mathbbm{E}\bigg\{\left\|\Delta_{j}f(Y)\right\|^{2}\left\|\tilde{\Delta}_i\Delta_{j}f(Y')\right\|^{p+2}\bigg\}+\sup_{(Y,Y')}\mathbbm{E}\bigg\{\left\|\Delta_{j}f(Y)\right\|^{p+2}\left\|\tilde{\Delta}_i\Delta_{j}f(Y')\right\|^{2}\bigg\}\\
\beta_{n,A_1,A_2}^{(3,p)}(i,j,k)\leq &\sup_{(Y,Y')}\mathbbm{E}\bigg\{\left\|\Delta_{j}f(Y)\right\|^{2}\left\|\tilde{\Delta}_i\Delta_{j}f(Y')\right\|^{p+2}\bigg\}+\sup_{(Y,Y')}\mathbbm{E}\bigg\{\left\|\Delta_{j}f(Y)\right\|^{p+2}\left\|\tilde{\Delta}_i\Delta_{j}f(Y')\right\|^{2}\bigg\}\\
\beta_{n,A_1,A_2}^{(4,p)}(i,j,k)\leq &\mathbbm{E}\bigg\{\left(\|\Delta_jf(X)\|^p+\|\tilde{\Delta}_i\Delta_jf(X)\|^p\right)\left\|\tilde{\Delta}_i\Delta_{j}f(X^{A_1})\right\|\left\|\tilde{\Delta}_i\Delta_{j}f(X)\right\|^2\left\|\tilde{\Delta}_i\Delta_{j}f(X^{A_2})\right\|\bigg\}\\
\leq & (2^p+1)\sup_{(Y,Y')}\mathbbm{E}\Big\{\|\Delta_jf(Y)\|^p\|\tilde{\Delta}_i\Delta_jf(Y')\|^{4}\Big\}.
\end{align*}
Therefore, applying Lemma \ref{lemma_estimates} and using the notation thereof,
\begin{align}
&\beta_{n,A_1,A_2}^{(1,p)}(i,j,k)\leq\frac{4(2^p+1)d!C_{p+4}^{(2)}(d)}{n^{p/2+3}};\quad \beta_{n,A_1,A_2}^{(2,p)}(i,j,k)+\beta_{n,A_1,A_2}^{(3,p)}(i,j,k)\leq \frac{2\left(C_{2,p+2}^{(3)}(d)+C_{p+2,2}^{(3)}(d)\right)}{n^{p/2+3}};\notag\\
&\beta_{n,A_1,A_2}^{(4,p)}(i,j,k)\leq\frac{(2^p+1)C_{p,4}^{(3)}(d)}{n^{p/2+3}},\quad\text{if }i\neq k=j,\,j\not\in A_1\cup A_2.\label{beta_bound3}
\end{align}
Finally, if $i,j,k$ are pairwise distinct and $k\not\in A_2$, $j\not\in A_1$, using (\ref{indic_bound}) and  H\"older's inequality we have:
\begin{align*}
\beta_{n,A_1,A_2}^{(1,p)}(i,j,k)\leq &16(2^{p/2}+1)^2d!\sup_{(Y,Y',Z,Z')}\mathbbm{E}\bigg\{\overline{V}(Y_1\cap Y_2)\overline{V}(Z_1\cap Z_3)\|\Delta_2f(Y')\|^{p/2+2}\|\Delta_3f(Z')\|^{p/2+2}\bigg\}\\
\beta_{n,A_1,A_2}^{(2,p)}(i,j,k)\leq &(2^{p/2}+1)^2\sup_{(Y,Y',Y'',Z)}\mathbbm{E}\bigg\{\|\tilde{\Delta}_1\Delta_2f(Y)\|\|\tilde{\Delta}_1\Delta_3f(Y')\|\|\Delta_2f(Y'')\|^{p/2+1}\|\Delta_3f(Z)\|^{p/2+1}\bigg\}\\
\beta_{n,A_1,A_2}^{(3,p)}(i,j,k)\leq &(2^{p/2}+1)^2\sup_{(Y,Y',Y'',Z)}\mathbbm{E}\bigg\{\|\tilde{\Delta}_1\Delta_2f(Y)\|\|\tilde{\Delta}_1\Delta_3f(Y')\|\|\Delta_2f(Y'')\|^{p/2+1}\|\Delta_3f(Z)\|^{p/2+1}\bigg\}\\
\beta_{n,A_1,A_2}^{(4,p)}(i,j,k)\leq &(2^{p/2}+1)^2\sup_{(Y,Y',Z,Z')}\mathbbm{E}\bigg\{\|\Delta_2f(Y)\|^{p/2}\|\Delta_3f(Y')\|^{p/2}\|\tilde{\Delta}_1\Delta_2f(Z)\|^{2}\|\tilde{\Delta}_1\Delta_3f(Z')\|^{2}\bigg\}.
\end{align*}

Therefore, applying Lemma \ref{lemma_estimates} and using the notation thereof,
\begin{align}
&\beta_{n,A_1,A_2}^{(1,p)}(i,j,k)\leq\frac{16(2^{p/2}+1)^2d!C_{p/2+2}^{(5)}(d)}{n^{p/2+4}};\quad \beta_{n,A_1,A_2}^{(2,p)}(i,j,k)+\beta_{n,A_1,A_2}^{(3,p)}(i,j,k)\leq \frac{2(2^{p/2}+1)^2C_{p/2+1,1}^{(4)}(d)}{n^{p/2+4}};\notag\\
&\beta_{n,A_1,A_2}^{(4,p)}(i,j,k)\leq\frac{(2^{p/2}+1)^2C_{p/2,2}^{(4)}(d)}{n^{p/2+4}},\quad \text{if } i,j,k \text{ are pairwise distinct and } k\not\in A_2, j\not\in A_1.\label{beta_bound4}
\end{align}
 It now follows from (\ref{beta_bound0}), (\ref{beta_bound1}), (\ref{beta_bound2}), (\ref{beta_bound3}) and (\ref{beta_bound4}) that, for $p=1,2$,
\begin{align}
\gamma_{p+2}\leq&\Bigg\{\frac{3}{2}\bigg[2C_{p+4}^{(1)}(d)+\left(8\left(2^{p/2}+1\right)+4(2^p+1)\right)(d!)C_{p+2}^{(2)}(d)+16(2^{p/2}+1)^2(d!)C_{p/2+2}^{(5)}(d)\bigg]\notag\\
&+\left(\frac{9}{2}+\frac{9}{2p}\right)\bigg[6C_{p+4}^{(1)}(d)+4\left(2^{p/2}+1\right)\sqrt{C_{2,p/2+1,1}^{(6)}(d)C_{p+2,p/2+1,1}^{(6)}(d)}+2\left(2^{p/2}+1\right)C_{p/2,p/2,4}^{(6)}(d)\notag\\
&+2\left(C_{2,p+2}^{(3)}(d)+C_{p+2,p}^{(3)}(d)\right)+(2^p+1)C_{p,4}^{(3)}(d)+2(2^{p/2}+1)^2C_{p/2+1,1}^{(4)}(d)+(2^{p/2}+1)^2C_{p/2,2}^{(4)}(d)\bigg]\Bigg\}^{1/(p+2)} \hspace{-3mm}n^{-1/2}\notag\\
\leq& 215(d+2)^{5/3}(d!)^{5/3}\left(2^{2^d+2d}d^{d/2}\right)^{5/3}(3(R+1))^{5d/3}e^{456(2R+1)^d}n^{-1/2}
\label{gamma_p_bound}
\end{align}
Now, we use an estimate similar to (\ref{indic_bound}) together with Lemmas \ref{lemma_bn} and \ref{lemma_estimates}. Using the notation theoreof,
\begin{align*}
B_n(f)\leq& 4d!\sup_{(Y,Z,Z')}\mathbbm{E}\bigg[\overline{V}(Y_1\cap Y_2)\|\Delta_1f(Z)\|^2\|\Delta_2f(Z')\|^2\bigg]\leq\frac{4d!C_{4}^{(2)}(d)}{n^3};\\
B_n'(f)\leq&16(d!)^2\sup_{(Y,Y',Z,Z')}\mathbbm{E}\bigg[\overline{V}(Y_1\cap Y_2)\overline{V}(Y'_1\cap Y'_3)\|\Delta_1f(Z)\|^2\|\Delta_2f(Z')\|^2\bigg]\leq \frac{16(d!)^2C_2^{(5)}(d)}{n^4}.
\end{align*}
and therefore, applying the estimate on $\mathbbm{E}\|\Delta_1f(X)\|^4$ from Lemma \ref{lemma_estimates} and noting that $\mathbbm{E}T=\Sigma_n$, we obtain from Lemma \ref{lemma_bn} that
\begin{align}
\sqrt{\mathbbm{E}\left\|\mathbbm{E}[T|X]-\Sigma_n\right\|^2_{H.S.}}\leq&\left(8\sqrt{d!C_4^{(2)}(d)}+16d!\sqrt{C_2^{(5)}(d)}+4\sqrt{C_4^{(1)}(d)}\right)n^{-1/2}\notag\\
\leq& 1168(d!)^3(d+2)^2\left(2^{2^d+2d}d^{d/2}\right)^2\left(3(R+1)\right)^{4d}e^{8(2R+1)^d}n^{-1/2}\label{t_x_bound}
\end{align}
Moreover, notice that, by Theorem \ref{pos_def_teorem}, matrix $\Sigma_n$ is positive definite for large enough $n$ and, for all $n>\overline{V}(K)e$,
\begin{align}\label{matrix_bound}
\left\|\Sigma-\Sigma_n\right\|_{H.S.}\leq 116\cdot 108^d(R+1)^{4d}e^{6\cdot 9^d(R+1)^{2d}}(d!)^2(d+1) n^{-1/d}.
\end{align}
The result now follows from (\ref{gamma_1_bound}), (\ref{gamma_p_bound}) - (\ref{matrix_bound}).\qed

\section{Proofs of the results about structures with local dependence}\label{proofs_ap}
\subsection{Proof of Theorem \ref{main1}}
\subsubsection{Step 1 - an upper bound on $\sqrt{\mathbbm{E}\left\|\mathbbm{E}[T|X]-\mathbbm{E}[T]\right\|_{H.S.}^2}$}
In this step, we follow a strategy analogous to the one of the proof of \cite{new_stein}. First, let $k_{n,A}:=\frac{1}{{n\choose |A|}(n-|A|)}$ for $A\subsetneq[n]$. Using (\ref{t_bound}) and (\ref{t_bound1}), we have that
\begin{align}\label{proof_thm31_1}
\sqrt{\mathbbm{E}\left\|\mathbbm{E}[T|X]-\mathbbm{E}[T]\right\|_{H.S.}^2}\leq\frac{1}{\sqrt{8}}\sum_{A\subsetneq [n]}k_{n,A}\sqrt{\sum_{i=1}^{n}\mathbbm{E}\left\|\tilde{\Delta}_iT_A\right\|_{H.S.}^2}.
\end{align}
Moreover, a straightforward adaptation of the proof of \cite[Lemma 4.5]{new_stein} yields the following bound:
\begin{align}\label{proof_thm31_2}
\mathbbm{E}\left\|\tilde{\Delta}_iT_A\right\|_{H.S.}^2\leq c\left(\mathbbm{E}\left(M^8\right)\right)^{1/2}\left(\mathbbm{E}\left(\delta^4\right)\right)^{1/2}\left(\mathbbm{1}_{\left[i\not\in A\right]}+\sqrt{\frac{n-|A|}{n}}\right),
\end{align}
for some universal constant $c>0$ which does not depend on $n$, $d$, $A$ or $i$ and for any $i\in[n]$ and $A\subsetneq[n]$. Using (\ref{proof_thm31_1}) and (\ref{proof_thm31_2}) we therefore obtain that
\begin{align}
\sqrt{\mathbbm{E}\left\|\mathbbm{E}[T|X]-\mathbbm{E}[T]\right\|_{H.S.}^2}\leq&\frac{\sqrt{c}}{\sqrt{8}}\sum_{A\subsetneq [n]}k_{n,A}\sqrt{\left(\mathbbm{E}\left(M^8\right)\right)^{1/2}\left(\mathbbm{E}\left(\delta^4\right)\right)^{1/2}\left(n-|A|+\sqrt{n(n-|A|)}\right)}\notag\\
\leq & \frac{\sqrt{c}}{2}\left(\mathbbm{E}\left(M^8\right)\right)^{1/4}\left(\mathbbm{E}\left(\delta^4\right)\right)^{1/4}\sum_{A\subsetneq [n]}\frac{\left(n(n-|A|)\right)^{1/4}}{{n\choose |A|}(n-|A|)}\notag\\
= &\frac{\sqrt{c}}{2}\left(\mathbbm{E}\left(M^8\right)\right)^{1/4}\left(\mathbbm{E}\left(\delta^4\right)\right)^{1/4}\sum_{k=1}^n n^{1/4}k^{-3/4}\notag\\
\leq &2\sqrt{c}\left(\mathbbm{E}\left(M^8\right)\right)^{1/4}\left(\mathbbm{E}\left(\delta^4\right)\right)^{1/4} n^{1/2}.  \label{proof_thm31_3}
\end{align}

\subsubsection{Step 2 - an upper bound on $\gamma_3$ and $\gamma_4$ of Theorem \ref{main3}}
Fix $A\subsetneq [n]$ and for $i\in[n]$ let
\begin{align*}
&K_i:=\max_{j\not\in A}\sqrt{\|\Delta_jf(X)\|+\|\tilde{\Delta}_i\Delta_jf(X)\|}\cdot\|\Delta_jf(X^A)\|\,\|\Delta_jf(X)\|,\\
&\tilde{K}_i:=\max_{j\not\in A}\sqrt{\|\Delta_jf(X)\|^2+\|\tilde{\Delta}_i\Delta_jf(X)\|^2}\cdot\|\Delta_jf(X^A)\|\,\|\Delta_jf(X)\|.
\end{align*}
We can bound the terms appearing in the definition of $\gamma_3$ and $\gamma_4$ applying the strategy used to bound the term $h_i$ in the proof of \cite[Lemma 4.5]{new_stein}. In particular, just like in display \cite[(21)]{new_stein} and the seventh display on page 25 of \cite{new_stein}, we can derive the following estimates:
\begin{align}
&\mathbbm{E}\left[\left(\sum_{j\not\in A}\mathbbm{1}_{[\tilde{\Delta}_i\Delta_jf(X)\neq 0]}\sqrt{\|\Delta_jf(X)\|+\|\tilde{\Delta}_i\Delta_jf(X)\|}\cdot\|\Delta_jf(X^A)\|\,\|\Delta_jf(X)\|\right)^2\right]\notag\\
\leq &\sqrt{\mathbbm{E}(K_i^4)}\tilde{c}\left(\mathbbm{E}\delta^4\right)^{1/2}\left(\mathbbm{1}_{[i\not\in A]}+\sqrt{\frac{n-|A|}{n}}\right)\notag\\
\leq&\tilde{c}\,\sqrt{\mathbbm{E}\left[\max_{j\not\in A}\left(2\|\Delta_jf(X)\|+\|\Delta_jf(\tilde{X})\|\right)^2\left(\|\Delta_jf(X^A)\|\,\|\Delta_jf(X)\|\right)^4\right]}\left(\mathbbm{E}\delta^4\right)^{1/2}\left(\mathbbm{1}_{[i\not\in A]}+\sqrt{\frac{n-|A|}{n}}\right)\notag\\
\leq& 3\tilde{c}\sqrt{\mathbbm{E}M^{10}}\left(\mathbbm{E}\delta^4\right)^{1/2}\left(\mathbbm{1}_{[i\not\in A]}+\sqrt{\frac{n-|A|}{n}}\right).\label{proof_thm31_6}
\end{align}
for a universal constant $\tilde{c}>0$ which does not depend on $n$, $d$, $A$ or $i$. Similarly,
\begin{align*}
&\mathbbm{E}\left(\sum_{j\not\in A}\mathbbm{1}_{[\tilde{\Delta}_i\Delta_jf(X)\neq 0]}\sqrt{\|\Delta_jf(X)\|^2+\|\tilde{\Delta}_i\Delta_jf(X)\|^2}\cdot\|\Delta_jf(X^A)\|\,\|\Delta_jf(X)\|\right)^2\\
\leq &\sqrt{\mathbbm{E}(\tilde{K}_i^4)}\tilde{c}\left(\mathbbm{E}\delta^4\right)^{1/2}\left(\mathbbm{1}_{[i\not\in A]}+\sqrt{\frac{n-|A|}{n}}\right)\\
\leq&\tilde{c}\,\left[\mathbbm{E}\left[\max_{j\not\in A}\left(3\|\Delta_jf(X)\|^2+2\|\Delta_jf(\tilde{X})\|^2\right)^2\left(\|\Delta_jf(X^A)\|\,\|\Delta_jf(X)\|\right)^4\right]\right]^{1/2}\left(\mathbbm{E}\delta^4\right)^{1/2}\left(\mathbbm{1}_{[i\not\in A]}+\sqrt{\frac{n-|A|}{n}}\right)\\
\leq& 5\tilde{c}\sqrt{\mathbbm{E}M^{12}}\left(\mathbbm{E}\delta^4\right)^{1/2}\left(\mathbbm{1}_{[i\not\in A]}+\sqrt{\frac{n-|A|}{n}}\right).
\end{align*}
We can treat the other summands appearing in the definitions of $\gamma_3$ and $\gamma_4$ in a similar way.
Therefore, using (\ref{proof_thm31_6}) and Minkowski's inequality, we obtain that for a universal constant $\bar{c}>0$ which does not depend on $n$ or $d$, 
\begin{align}
\gamma_3^3\leq &\bar{c}\sum_{i=1}^n\left(\sum_{A\subsetneq [n]}\frac{1}{{n\choose|A|}(n-|A|)}\left(\mathbbm{E}M^{10}\right)^{1/4}\left(\mathbbm{E}\delta^4\right)^{1/4}\left(\mathbbm{1}_{[i\not\in A]}+\left(\frac{n-|A|}{n}\right)^{1/4}\right)\right)^2\notag\\
\leq& 2\bar{c}\left(\mathbbm{E}M^{10}\right)^{1/2}\left(\mathbbm{E}\delta^4\right)^{1/2}\left\lbrace\sum_{i=1}^n\left(\sum_{A\subsetneq [n]}\frac{\mathbbm{1}_{[i\not\in A]}}{{n\choose|A|}(n-|A|)}\right)^2+\left(\sum_{A\subsetneq [n]}\frac{n^{1/4}(n-|A|)^{1/4}}{{n\choose|A|}(n-|A|)}\right)^2\right\rbrace\notag\\
\leq &2\bar{c}\left(\mathbbm{E}M^{10}\right)^{1/2}\left(\mathbbm{E}\delta^4\right)^{1/2}\left\lbrace \left(\sum_{A\subsetneq [n]}\frac{(n-|A|)^{1/2}}{{n\choose|A|}(n-|A|)}\right)^2+\left(\sum_{A\subsetneq [n]}\frac{n^{1/4}(n-|A|)^{1/4}}{{n\choose|A|}(n-|A|)}\right)^2\right\rbrace\notag\\
\leq& 4\bar{c}\left(\mathbbm{E}M^{10}\right)^{1/2}\left(\mathbbm{E}\delta^4\right)^{1/2} \left(\sum_{A\subsetneq [n]}\frac{n^{1/4}(n-|A|)^{1/4}}{{n\choose|A|}(n-|A|)}\right)^2\notag\\
\leq &64\bar{c}\left(\mathbbm{E}M^{10}\right)^{1/2}\left(\mathbbm{E}\delta^4\right)^{1/2}n\label{proof_thm31_4}
\end{align}
and, similarly, for a universal constant $\hat{c}>$ that does not depend on $n$ or $d$,
\begin{align}
\gamma_4^4\leq &\hat{c}\left(\mathbbm{E}M^{12}\right)^{1/2}\left(\mathbbm{E}\delta^4\right)^{1/2}n.\label{proof_thm31_5}
\end{align}
\subsubsection{Conclusion}
The result now follows from Theorems \ref{main} and \ref{main3}, together with (\ref{proof_thm31_3}), (\ref{proof_thm31_4}) and (\ref{proof_thm31_5}). \qed

\subsection{Proof of Theorem \ref{main2}}

It is shown in the proof of \cite[Theorem 3.4]{new_stein} that a function of the form of the form (\ref{function}) admits a symmetric interaction rule $G$ posessing a symmetric extension on $\mathcal{X}^{n+4}$ whose maximum degree is bounded by $\alpha(d)(k+1)(k+5)$.
We let
$$M_f:=\max_l\|f_l(X)\|\vee\max_{j,l}\|f_l(X^j)\|.$$
By a straightforward adaptation of the proof of \cite[Theorem 3.4]{new_stein}, the random variable $M$ of Theorem \ref{main1} can be bounded by $4n^{-1/2}\alpha(d)kM_f$, under the assumptions of Theorem \ref{main2}. Next,  for any $p\geq 8$, it holds that
$$\mathbbm{E}[M_f^8]\leq \left[\mathbbm{E} M_f^p\right]^{8/p}\leq \left[\sum_l\mathbbm{E}\|f_l(X)\|^p+\sum_{j,l}\mathbbm{E}\|f_l(X^j)\|^p\right]^{8/p}\leq (n^2+n)^{8/p}\eta_p^{8/p}.$$
and, similarly, for any $r\leq p$, $\mathbbm{E}\|\Delta_j f(X)\|^r\leq C\alpha(d)^r k^r n^{-r/2}(n\,\eta_p)^{r/p}$.

Furthermore, for any $p\geq 12$,
\begin{align*}
\mathbbm{E}[M_f^{10}]\leq \left[\mathbbm{E} M_f^p\right]^{10/p}\leq \left[\sum_l\mathbbm{E}\|f_l(X)\|^p+\sum_{j,l}\mathbbm{E}\|f_l(X^j)\|^p\right]^{10/p}\leq (n^2+n)^{10/p}\eta_p^{10/p};\\
\mathbbm{E}[M_f^{12}]\leq \left[\mathbbm{E} M_f^p\right]^{12/p}\leq \left[\sum_l\mathbbm{E}\|f_l(X)\|^p+\sum_{j,l}\mathbbm{E}\|f_l(X^j)\|^p\right]^{12/p}\leq (n^2+n)^{12/p}\eta_p^{12/p}.
\end{align*}
Combining the bounds for all the terms and using Theorem \ref{main1} completes the proof.\qed

\section*{Acknowledgements}
The authors were supported by the \textbf{FNR grant FoRGES (R-AGR- 3376-10)} at Luxembourg University. This work is also part of project \textbf{Stein-ML} that has received funding from the European Union’s Horizon 2020 research and innovation programme under the Marie Skłodowska-Curie grant agreement \textbf{No 101024264}.

\bibliographystyle{alpha}
\bibliography{Bibliography}

\newcommand{\etalchar}[1]{$^{#1}$}
\begin{thebibliography}{LMN{\etalchar{+}}20}

\bibitem[Bar90]{diffusion}
A.D. Barbour.
\newblock {Stein's Method for Diffusion Approximation}.
\newblock {\em Probability Theory and Related Fields}, 84:297--322, 1990.

\bibitem[BB83]{BB}
P.~J. Bickel and L.~Breiman.
\newblock Sums of functions of nearest neighbor ditances, moment bounds, limit
  theorems and a goodnes of ifit test.
\newblock {\em Ann. Probab.}, 11(1), 1983.

\bibitem[Ben05]{bentkus}
V.~Bentkus.
\newblock A lyapunov type bound in $r^d$.
\newblock {\em Theory Probab. Appl}, 49:311--323, 2005.

\bibitem[BLM13]{BLM}
S.~Boucheron, G.~Lugosi, and P.~Massart.
\newblock {\em Concentration inequalities}.
\newblock Oxford University Press, 2013.

\bibitem[BR76]{batrao}
R.N. Bhattacharya and R.R. Rao.
\newblock {\em Normal Approximation and Asymptotic Expansions}.
\newblock Wiley, New York, 1976.

\bibitem[CDM05]{chatterjee}
S.~Chatterjee, P.~Diaconis, and E.~Meckes.
\newblock Exchangeable pairs and poisson approximation.
\newblock {\em Probab. Surveys}, 2:64--106, 2005.

\bibitem[CGS11]{normal_approx}
L.H.Y Chen, L.~Goldstein, and Q.-M. Shao.
\newblock {\em Normal Approximation by Stein's Method}.
\newblock {Probability and Its Applications}. Springer Verlag, 2011.

\bibitem[Cha08]{new_stein}
S.~Chatterjee.
\newblock A new method of normal approximation.
\newblock {\em The Annals of Probability}, 36:1584--1610, 2008.

\bibitem[CM08]{meckes}
S.~Chatterjee and E.~Meckes.
\newblock Multivariate normal approximation using exchangeable pairs.
\newblock {\em ALEA Lat. Am. J. Probab. Math. Stat.}, 4:257--283, 2008.

\bibitem[CS17]{chatterjee_sen}
S.~Chatterjee and S.~Sen.
\newblock Minimal spanning trees and stein's method.
\newblock {\em Ann. Appl. Probab.}, 27(3):1588--1645, 2017.

\bibitem[DP17]{DP_ejp}
Ch. D{\"o}bler and G.~Peccati.
\newblock Quantitative de jong theorems in any dimension.
\newblock {\em Electron. J. Probab.}, 22(2):1--35, 2017.

\bibitem[Due21]{duerinckx}
M.~Duerinckx.
\newblock On the size of chaos via glauber calculus in the classical mean-field
  dynamics.
\newblock {\em Comm. Math. Physics}, 382:613--653, 2021.

\bibitem[Dun19]{dung}
N.T. Dung.
\newblock Explicit rates of convergence in the multivariate clt for nonlinear
  statistics.
\newblock {\em Acta Mathematica Hungarica}, 158:173--201, 2019.

\bibitem[FK20]{FK_21c}
X.~Fang and Y.~Koike.
\newblock Large-dimensional central limit theorem with fourth-moment error
  bounds on convex sets and balls, 2020.
\newblock arXiv:2009.00339.

\bibitem[FK21]{FK_21}
X.~Fang and Y.~Koike.
\newblock High-dimensional central limit theorems by stein's method.
\newblock {\em Ann. Appl. Probab.}, 31(4):1660--1686, 2021.

\bibitem[FK22]{FK_21b}
X.~Fang and Y.~Koike.
\newblock New error bounds in multivariate normal approximations via
  exchangeable pairs with applications to wishart matrices and fourth moment
  theorems.
\newblock {\em Ann. Appl. Probab.}, to appear 2022+.

\bibitem[GN16]{gloria_nolen}
A.~Gloria and J.~Nolen.
\newblock A quantitative central limit theorem for the effective conductance on
  the discrete torus.
\newblock {\em Comm. Pure Appl. Math.}, 69(12):2304--2348, 2016.

\bibitem[G{\"o}t91]{gotze}
F.~G{\"o}tze.
\newblock {On the rate of convergence in the multivariate CLT}.
\newblock {\em The Annals of Probability}, 19(2):724--739, 1991.

\bibitem[GP10]{goldstein_penrose}
L.~Goldstein and M.~Penrose.
\newblock Normal approximation for coverage models over binomial point
  processes.
\newblock {\em Ann. Appl. Probab.}, 20(2):696--721, 2010.

\bibitem[Hal88]{hall_book}
P.~Hall.
\newblock {\em Introduction to the theory of coverage processes}.
\newblock Wiley, New York, 1988.

\bibitem[HLS16]{hug_last_schulte}
Daniel Hug, G{\"u}nter Last, and Matthias Schulte.
\newblock Second-order properties and central limit theorems for geometric
  functionals of boolean models.
\newblock {\em Ann. Appl. Probab.}, 26(1):73--135, 02 2016.

\bibitem[LB05]{LB}
E.~Levina and P.~J. Bickel.
\newblock Maximum likelihood estimation of intrinsic dimension.
\newblock In K.L. Saul, Y.~Weiss, and L.~Bottou, editors, {\em Advances in
  NIPS}, volume~17, 2005.

\bibitem[LMN{\etalchar{+}}20]{tropp}
Martin Lotz, Michael~B McCoy, Ivan Nourdin, Giovanni Peccati, and Joel~A Tropp.
\newblock Concentration of the intrinsic volumes of a convex body.
\newblock {\em Geometric Aspects of Functional Analysis -- Israel Seminar
  (GAFA) 2017-2019}, Lecture Notes in Mathematics 2256, 2020.

\bibitem[LP17]{last_penrose}
G{\"u}nter Last and Mathew Penrose.
\newblock {\em Lectures on the Poisson Process}.
\newblock Institute of Mathematical Statistics Textbooks. Cambridge University
  Press, 2017.

\bibitem[LRP17]{peccati_lachieze-rey2017}
Rapha{\"e}l Lachi{\`e}ze-Rey and Giovanni Peccati.
\newblock {New Berry--Esseen bounds for functionals of binomial point
  processes}.
\newblock {\em Ann. Appl. Probab.}, 27(4):1992--2031, 2017.

\bibitem[LRPY22]{lrpy}
R.~Lachi{\`e}ze-Rey, G.~Peccati, and X.~Yang.
\newblock Quantitative two-scale stabilization on the poisson space.
\newblock {\em Ann. Appl. Probab.}, To appear (2022+).

\bibitem[Mec09]{meckes09}
E.~Meckes.
\newblock {\em On Stein's method for multivariate normal approximation}, volume
  Volume 5 of {\em Collections}, pages 153--178.
\newblock Institute of Mathematical Statistics, Beachwood, Ohio, USA, 2009.

\bibitem[Mor58]{moran}
P.A.P Moran.
\newblock {Random Processes in Genetics}.
\newblock {\em Proc. Camb. Phil. Soc.}, (54):60--71, 1958.

\bibitem[MV97]{eigenvalue}
Jorma~Kaarlo Merikoski and Ari Virtanen.
\newblock Bounds for eigenvalues using the trace and determinant.
\newblock {\em Linear Algebra and its Applications}, 264:101--108, 1997.
\newblock Sixth Special Issue on Linear Algebra and Statistics.

\bibitem[Naz03]{nazarov}
F.~Nazarov.
\newblock On the maximal perimeter of a convex set in $\mathbb{R}^n$ with
  respect to a gaussian measure.
\newblock In V.D. Milman and G.~Schechtman, editors, {\em Geometric Aspects of
  Functional Analysis}, volume 1807 of {\em Lecture Notes in Mathematics},
  pages 169--187, Israel, 2003. Springer-Verlag.

\bibitem[NP12]{nourdin}
I.~Nourdin and G.~Peccati.
\newblock {\em Normal Approximations with Malliavin Calculus}.
\newblock Cambridge tracts in Mathematics. Cambridge University Press, 2012.

\bibitem[NPR10]{nourdin_peccati_reveillac}
I.~Nourdin, G.~Peccati, and A.~R{\'e}veillac.
\newblock {Multivariate normal approximation using Stein's method and Malliavin
  calculus}.
\newblock {\em Annales de l'Institut Henri Poincar{\'e}, Probabilit{\'e}s et
  Statistiques}, 46(1):45 -- 58, 2010.

\bibitem[NPY22]{NPY_jtp}
I.~Nourdin, G.~Peccati, and X.~Yang.
\newblock Multivariate normal approximation on the wiener space: new bounds in
  the convex distance.
\newblock {\em Journal of Theoretical Probability}, to appear, 2022+.

\bibitem[Pen03]{penrose}
Mathew Penrose.
\newblock {\em Random Geometric Graphs}.
\newblock Oxford Studies in Probability. Oxford University Press, 2003.

\bibitem[PY01]{PY01}
M.~Penrose and J.E. Yukich.
\newblock Central limit theorems for some graphs in computational geometry.
\newblock {\em Ann. Appl. Probab.}, 11(1005-1041), 2001.

\bibitem[Rai19]{raic}
M.~Raic.
\newblock A multivariate berry--esseen theorem with explicit constants.
\newblock {\em Bernoulli}, 25(4(A)):2824--2853, 2019.

\bibitem[RR96]{RR96}
Y.~Rinott and V.~Rotar.
\newblock A multivariate clt for local dependence with $n^{-1/2} \log n$ rate
  and applications to multivariate graph related statistics.
\newblock {\em Journal of Multivariate Analysis}, 56:333--350, 1996.

\bibitem[RR97]{rinott}
Y.~Rinott and V.~Rotar.
\newblock {On coupling constructions and rates in the CLT for dependent
  summands with applications to the antivoter model and weighted
  $U$-statistics}.
\newblock {\em Ann. Appl. Probab.}, 7(4):1080--1105, 11 1997.

\bibitem[RR09]{reinert_roellin}
G.~Reinert and A.~R{\"o}llin.
\newblock {Multivariate normal approximation with Stein's method of
  exchangeable pairs under a general linearity condition}.
\newblock {\em The Annals of Probability}, 37(6):2150--2173, 2009.

\bibitem[RR10]{reinert_roellin1}
G.~Reinert and A.~R{\"o}llin.
\newblock Random subgraph counts and $u$-statistics: Multivariate normal
  approximation via exchangeable pairs and embedding.
\newblock {\em Journal of Applied Probability}, 47(2):378--393, 2010.

\bibitem[SK04]{san}
Luis~A. Santal{\'o} and Mark Kac.
\newblock {\em Integral Geometry and Geometric Probability}.
\newblock Cambridge Mathematical Library. Cambridge University Press, 2
  edition, 2004.

\bibitem[SS05]{SS_05}
Q.-M. Shao and Z.-G. Su.
\newblock The berry-esseen bounds for character ratios.
\newblock {\em Proc. Am. Math. Soc.}, 134(7):2153--2159, 2005.

\bibitem[Ste86]{steele}
J.~Michael Steele.
\newblock {An Efron-Stein Inequality for Nonsymmetric Statistics}.
\newblock {\em Ann. Statist.}, 14(2):753--758, 06 1986.

\bibitem[SW08]{geometry}
Rolf Schneider and Wolfgang Weil.
\newblock {\em {Stochastic and Integral Geometry}}.
\newblock Probability and Its Applications. Springer, Berlin, Heidelberg, 2008.

\bibitem[SY19]{schulte_yukich}
M.~Schulte and J.E. Yukich.
\newblock {Multivariate second order Poincare inequalities for Poisson
  functionals}.
\newblock arXiv:1803.11059, 2019.

\bibitem[SY21]{schulte_yukich_21}
M.~Schulte and J.E. Yukich.
\newblock Rates of multivariate normal approximation for statistics in
  geometric probability.
\newblock {\em ArXiv Preprint}, 2021.

\end{thebibliography}
\end{document}